%% file: BDL.tex
\documentclass[11pt,oneside]{amsart}

\usepackage{amsmath,amsthm,amssymb,amsfonts,amscd,verbatim}
\usepackage{url}
\usepackage{color}
\usepackage{setspace}
\usepackage{enumerate}
\usepackage[all,cmtip]{xy}
\usepackage[normalem]{ulem}

\usepackage{epsfig}
\usepackage{graphicx}
\usepackage{wrapfig}
\usepackage{mathrsfs}
\usepackage{pinlabel}

\usepackage{subfig}
\usepackage{arydshln}

\definecolor{grey}{rgb}{0.7,0.7,0.7}

\newcounter{notes}%

\newcommand{\ignore}[1]{}

\theoremstyle{plain}
\newtheorem{theorem}{Theorem}
\newtheorem{proposition}[theorem]{Proposition}
\newtheorem{corollary}[theorem]{Corollary}

\newtheorem{lemma}[theorem]{Lemma}

\newtheoremstyle{theoremwithref}{}{}{\itshape}{}{\bfseries}{.}{.5em}{#1 #2 #3}
\theoremstyle{theoremwithref}

\theoremstyle{definition}
\newtheorem{definition}[theorem]{Definition}

\newtheorem{remark}[theorem]{Remark}

\numberwithin{theorem}{section}
\numberwithin{equation}{section}

\newcommand{\bs}{\backslash}

\newcommand{\ZZ}{\mathbb{Z}}

\newcommand{\RR}{\mathbb{R}}
\newcommand{\AF}{\mathbb{A}}
\newcommand{\CC}{\mathbb{C}}
\newcommand{\HH}{\mathbb{H}}

\newcommand{\RP}{\mathbb{RP}}
\renewcommand{\SS}{\mathbb{S}}

\newcommand{\SL}{\mathrm{SL}}
\newcommand{\GL}{\mathrm{GL}}
\newcommand{\SO}{\mathrm{SO}}
\newcommand{\OO}{\mathrm{O}}
\newcommand{\PO}{\mathrm{PO}}

\newcommand{\PSL}{\mathrm{PSL}}
\newcommand{\PGL}{\mathrm{PGL}}

\newcommand{\g}{\mathfrak{g}}
\newcommand{\so}{\mathfrak{so}}
\newcommand{\ssl}{\mathfrak{sl}}
\newcommand{\gl}{\mathfrak{gl}}

\newcommand{\Ad}{\operatorname{Ad}}

\newcommand{\Hom}{\mathrm{Hom}}

\newcommand{\ie}{ie\ }
\newcommand{\eg}{e.g.\ }

\newcommand{\Rep}{\Hom}
\newcommand{\Char}{\chi}

\newcommand{\fund}[1]{\pi_1 #1}

\newcommand{\slice}{\mathscr{S}}
\newcommand{\res}{\mathsf{res}}
\newcommand{\rhyp}{\rho_{hyp}}

\newcommand{\co}{\!:}

\title[Convex projective structures on non-hyperbolic 3--manifolds]{Convex projective structures on \\ non-hyperbolic three-manifolds}

\author{Samuel A. Ballas}
\address{Department of Mathematics, Florida State University, Tallahassee, FL 32306, USA}
\email{ballas@math.fsu.edu}
\urladdr{https://www.math.fsu.edu/~ballas/}

\author{Jeffrey Danciger}
\address{Department of Mathematics, The University of Texas at Austin, 1 University Station C1200, Austin, TX 78712, USA}
\email{jdanciger@math.utexas.edu}
\urladdr{https://www.ma.utexas.edu/users/jdanciger/}

\author{Gye-Seon Lee}
\address{Mathematisches Institut, Ruprecht-Karls-Universit\"{a}t Heidelberg, D-69120 Heidelberg, Germany}
\email{lee@mathi.uni-heidelberg.de}
\urladdr{https://www.mathi.uni-heidelberg.de/~lee/}

\begin{document}

\begin{abstract}
Y. Benoist proved that if a closed three-manifold M admits an indecomposable convex real projective structure, then M is topologically the union along tori and Klein bottles of finitely many sub-manifolds each of which admits a complete finite volume hyperbolic structure on its interior. We describe some initial results in the direction of a potential converse to Benoist's theorem. We show that a cusped hyperbolic three-manifold may, under certain assumptions, be deformed to convex projective structures with totally geodesic torus boundary. Such structures may be convexly glued together whenever the geometry at the boundary matches up. In particular, we prove that many doubles of cusped hyperbolic three-manifolds admit convex projective structures.
\end{abstract}

\maketitle

\section{Introduction}
The previous decade has seen tremendous progress in the study of three-dimensional manifolds. Much of that progress stems from Perelman's proof of Thurston's Geometrization Conjecture which states that any closed orientable prime three-manifold admits a decomposition into geometric pieces modeled on the eight homogeneous Thurston geometries. However, because these geometric pieces do not glue together in any sensible geometric way, there are some questions about three-manifolds for which a Thurston geometric decomposition of the manifold may not be useful.
One example is the question of linearity of three-manifold groups, \ie whether a three-manifold fundamental group admits a faithful linear representation and in which dimensions. While, in most cases, the Thurston geometric structure on each piece of a geometric decomposition determines a faithful linear representation of its fundamental group, these representations can not be directly synthesized into a representation of the fundamental group of the entire manifold. 
In order to make progress on this and other problems, it is natural, given a manifold of interest, to search for a homogeneous geometry capable of describing the entire manifold all at once.

This article studies properly convex real projective structures on three-manifolds.
A domain $\Omega$ in the real projective space $\RP^n$ is called properly convex if there is an affine chart containing $\Omega$ in which $\Omega$ is convex and bounded. A properly convex projective $n$--manifold is the quotient $\Gamma \backslash \Omega$ of a convex domain $\Omega$ by a discrete group $\Gamma$ of projective transformations preserving $\Omega$.
Given a manifold $N$, a properly convex projective structure on $N$ is a diffeomorphism of $N$ with some properly convex projective manifold $\Gamma \backslash \Omega$, considered up to certain equivalence. A convex projective structure therefore induces a representation, called the holonomy representation, identifying $\pi_1 N$ with the discrete subgroup $\Gamma \subset \PGL_{n+1} \RR$.
Hyperbolic structures are special examples of convex real projective structures, but there are many non-hyperbolic manifolds that admit such structures as well. See Benoist~\cite{BenoistCD4,BenoistQ} or Kapovich~\cite{Kapovich} for some examples. See Benoist~\cite{Benoist-survey} for a survey of the subject of convex projective structures on closed manifolds.

We mention that there are simple examples of convex projective structures on three-manifolds coming from a convex hull construction applied to lower dimensional domains; such structures are called \emph{decomposable} and are not of interest to us in the present article. By work of Benoist~\cite{BenoistCD4}, the Thurston geometric decomposition of any closed three-manifold that admits an indecomposable properly convex projective structure contains only hyperbolic pieces glued together along tori and Klein bottles. We are concerned with the converse problem: If a closed three-manifold $N$ has geometric decomposition containing only hyperbolic pieces, must $N$ admit a properly convex projective structure? Our main theorem gives a positive answer to this question in a special case.

\begin{theorem}\label{thm:main}
Let $M$ be a compact, connected, orientable three-manifold with a union of tori as boundary such that the interior of $M$ admits a finite volume hyperbolic structure which is infinitesimally projectively rigid rel boundary. Then the double $N = 2M$ of $M$ admits a properly convex projective structure.
\end{theorem}

Using cube complex techniques, Przytycki--Wise~\cite{PrzWi} showed that any mixed three-manifold, and therefore any manifold $N$ as in Theorem~\ref{thm:main}, has linear fundamental group.
However, their methods give no control on the dimension of the linear representation.  On the 
other hand, it was shown by Button~\cite{Button} that there exist three-manifold groups which admit no linear representation in dimension $4$ or lower. To determine the smallest possible dimension of a linear
representation of a three-manifold group remains an interesting open question.
Since the holonomy representation of a convex projective structure lifts to the special linear group (see Section~\ref{ssec:propconvex}), we obtain:

\begin{corollary}
Let $N = 2M$ be as in Theorem~\ref{thm:main}. Then $\pi_1 N$ admits a discrete faithful representation into $\SL_4 \RR$.
\end{corollary}

The Corollary says that the property of linearity of the two hyperbolic pieces in $N$ may be extended to all of $N$. Indeed, the proof of Theorem~\ref{thm:main} will show that the representation of $\pi_1 N$ in the Corollary may be chosen such that the restriction to each copy of $\pi_1 M$ is arbitrarily close to the holonomy representation of the finite volume hyperbolic structure on $M$.

We note that the assumption of infinitesimal projective rigidity rel boundary (Definition~\ref{def:infrig}), studied by Heusener--Porti in \cite{HePo}, is satisfied for many hyperbolic manifolds, for example for infinitely many fillings of one component of the Whitehead link. On the other hand, this assumption does fail in certain cases, for example when the hyperbolic structure on $M$ contains a totally geodesic surface. A related rigidity condition in closed hyperbolic three-manifolds was studied in Cooper--Long--Thistlethwaite~\cite{CoLoThist2,CoLoThist1} and shown experimentally to hold very often in small examples. However, in the setting of cusped hyperbolic manifolds, it is not yet known in what degree of generality infinitesimal projective rigidity rel boundary will hold.
Nonetheless, the Theorem gives a large new source of examples of convex projective structures on non-hyperbolic three-manifolds. The only other known examples come from taking covers of convex projective reflection orbifolds. Benoist~\cite{BenoistCD4} constructed the first example of such an orbifold by realizing a truncation polyhedron, \ie a polyhedron obtained from a three-dimensional tetrahedron by successively truncating vertices, as a reflection polyhedron in projective space.
Later, Marquis~\cite{Mar} completely classified the three-dimensional convex projective orbifolds obtained from projective truncation polyhedra.
Generalizing Benoist's examples, Choi--Lee--Marquis~\cite{CGLM} are currently classifying convex projective reflection polyhedra and studying their deformation theory.

\subsection{Convex projective structures with totally geodesic boundary}
The proof of Theorem~\ref{thm:main} is motivated by Benoist's beautiful theory~\cite{BenoistCD4} describing the geometry of convex projective structures on three-manifolds, which we briefly review here. If $N = \Gamma \backslash \Omega$ is a properly convex projective closed three-manifold which is indecomposable, then either (i) $\Omega$ is strictly convex and $N$ admits a hyperbolic structure or (ii) $\Omega$ is not strictly convex and the points on $\partial \Omega$ at which strict convexity fails form a dense set in $\partial \Omega$ each component of which bounds a \emph{properly embedded triangle}, which is the intersection of $\Omega$ with a hyperplane. We will refer to these triangles as \emph{totally geodesic triangles}. Each totally geodesic triangle descends to a totally geodesic embedded torus or Klein bottle in $N$ and after cutting along these tori and Klein bottles, $N$ is decomposed into a union of properly convex sub-manifolds $M_i$ with \emph{totally geodesic boundary}.
Each piece in this decomposition (which topologically is exactly the JSJ or geometric decomposition of $N$) admits a hyperbolic structure.

In light of the Benoist theory, in order to construct convex projective structures on non-hyperbolic three-manifolds, we first need a source of convex projective building blocks, \ie convex projective manifolds with totally geodesic boundary.
Under suitable assumptions, we are able to find such structures by deforming the hyperbolic structure.

\begin{theorem}\label{thm:deform}
Let $M$ be a connected, orientable, finite volume, non-compact hyperbolic three-manifold which is infinitesimally projectively rigid rel boundary. Then $M$ admits nearby properly convex projective structures where each cusp becomes a principal totally geodesic boundary component.
\end{theorem}

The term \emph{principal totally geodesic boundary} (see Definition~\ref{def:principal}, following Goldman~\cite{Go} in the setting of convex projective surfaces) refers to a totally geodesic boundary component which admits a convex thickening. That all totally geodesic boundary components are principal is a necessary condition for a convex projective manifold to appear as a sub-manifold in the Benoist decomposition of a closed convex projective manifold described above.

The deformations in the previous theorem may be understood by analogy with the related phenomenon of the deformation of a two-dimensional finite volume hyperbolic surface whose cusp ``opens up'' to a very small geodesic circle coming in from infinity. Hyperbolic surfaces with geodesic boundary are indeed convex projective structures; the associated convex domain, a subset of the hyperbolic plane, has in its boundary a dense collection of segments, each of which covers the geodesic boundary circle.
Although cusp opening is not possible in three-dimensional hyperbolic geometry by Mostow--Prasad rigidity, Theorem \ref{thm:deform} shows that it is possible in the category of convex projective manifolds.
Indeed, as the convex projective structures in the conclusion of the Theorem approach the hyperbolic structure, the totally geodesic boundary tori become very small (with respect to the Hilbert metric) and escape to infinity as the totally geodesic triangles in the boundary of the associated convex domain converge to points; see Figure~\ref{fig:intro}.

\begin{figure}[ht!]
\centering
 \includegraphics[scale=.28]{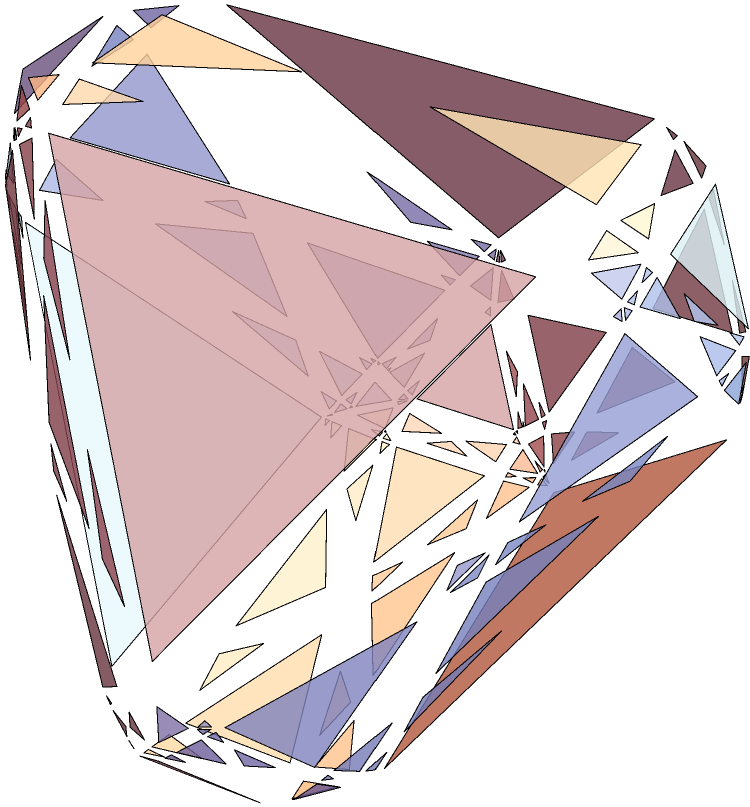} \quad
 \includegraphics[scale=.28]{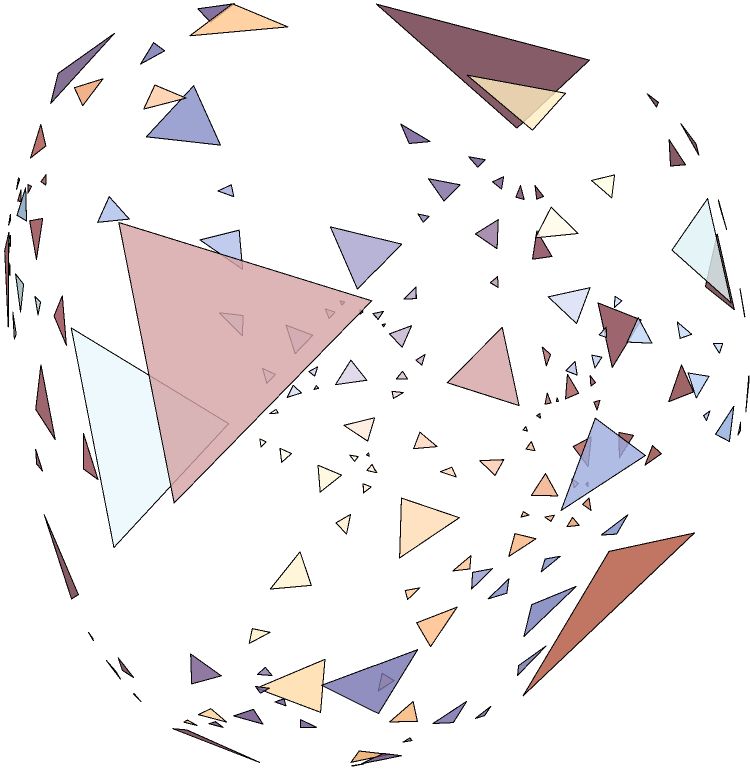} \quad
 \includegraphics[scale=.28]{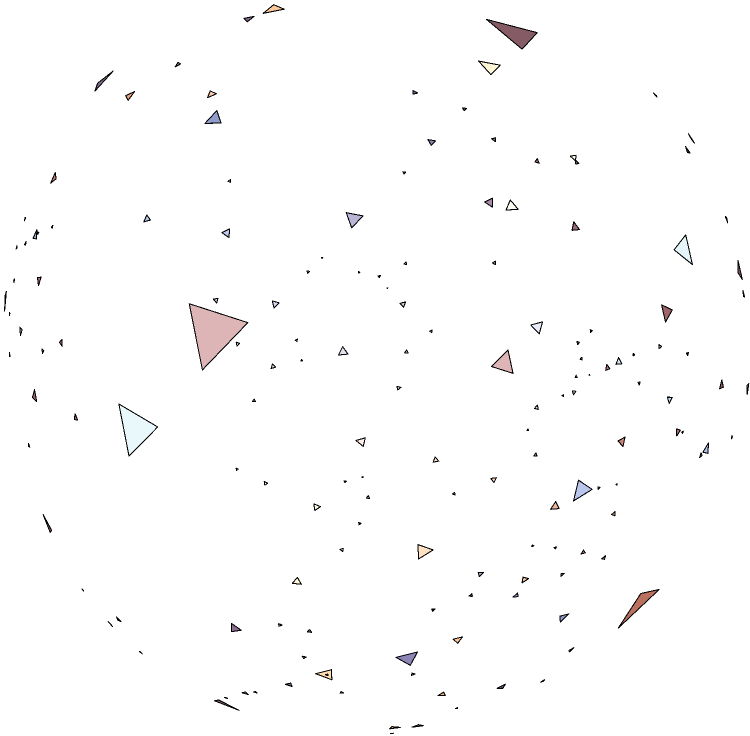}
 \caption{The principal totally geodesic triangles in the boundaries of the convex domains from Theorem \ref{thm:deform} collapse to points as the convex domains converge to the round ball.}
 \label{fig:intro}
\end{figure}

 The proof of Theorem \ref{thm:deform} boils down to a transversality argument in~the space $\Rep(\pi_1 \partial M, \PGL_4 \RR)$ of representations of the peripheral fundamental groups $\pi_1 \partial M$. The assumption of infinitesimal projective rigidity rel boundary, also appearing in Theorem~\ref{thm:main}, guarantees that $\Rep(\pi_1 M, \PGL_4 \RR)$ is smooth at the holonomy representation $\rho_{hyp}$ of the complete hyperbolic structure and that (an augmented version of) the restriction map $\res \co \Rep(\pi_1 M, \PGL_4 \RR)  \to \Rep(\pi_1 \partial M, \PGL_4 \RR)$ submerses a neighborhood of $\rho_{hyp}$ onto a submanifold of $\Rep(\pi_1 \partial M, \PGL_4 \RR)$.
We prove that this submanifold transversely intersects a certain family of diagonalizable representations constructed explicitly in Lemma~\ref{lem:slice}.
This family of diagonalizable representations is a partial slice (never tangent to conjugation orbits) whose
construction is geometric in nature (see
Section \ref{subsec:slice} for details). For dimensional reasons the intersection is positive dimensional and thus we can find representations of $\pi_1 M$ into $\PGL_4 \RR$ whose restriction to the fundamental group of each boundary component is diagonalizable over the reals. Finally, we are able to conclude that the resulting representations are the holonomy representations of convex projective
structures as in the Theorem using a ``holonomy principle'' that
follows from recent work of Cooper--Long--Tillmann~\cite{CoLoTi} or Choi~\cite{Choi2}.

We note that the proof of Theorem~\ref{thm:deform} encounters immediate problems upon removing the assumption that $M$ is orientable. Indeed, Poincar\'e duality, applied to both $M$ and its torus boundary components, is used to determine the dimensions of relevant representation spaces.

\subsection{Convex gluing}

Consider a finite disjoint union $\mathscr M$ of properly convex projective three-manifolds and two disjoint closed sub-manifolds $\partial$ and $\partial'$ of its boundary, each of which is a finite disjoint union of principal totally geodesic tori. 
Given a homeomorphism $f \co \partial \to \partial'$, let $\mathscr M_f$ denote the (topological) manifold obtained by identifying $\partial$ with $\partial'$ via $f$; the topology of $\mathscr M_f$ depends only on the isotopy class of $f$. By straightforward general arguments, the projective structure on $\mathscr M$ defines a projective structure (in fact many different structures) on $\mathscr M_f$ provided that the \emph{holonomy matching condition} is satisfied: For each component $\partial_j$ of $\partial$, there exists $g_j \in \PGL_4 \RR$ such that
\begin{equation}\label{eqn:matching}
\rho_j'( f_* \gamma)= g_j\rho_j(\gamma)g_j^{-1}
\end{equation}
for all $\gamma \in  \pi_1 \partial_j$, where $\rho_j$ and $\rho_j'$ denote the holonomy representations for the projective structures on neighborhoods of $\partial_j$ and $\partial_j' = f(\partial_j)$ respectively. Indeed, if this condition is satisfied then $f$ may be isotoped to a projective map on each component of $\partial$ and the pieces of $\mathscr M$ may be glued together projectively.
We prove that the resulting projective structure on $\mathscr M_f$ may in fact be taken to be properly convex. The following theorem is the three-dimensional analogue of a result of Goldman~\cite{Go} in the setting of convex projective surfaces.

\begin{theorem} \label{thm:convex-gluing}
Let $\mathscr M$ be a finite disjoint union of properly convex projective three-manifolds, let $\partial$ and $\partial'$ be  disjoint closed sub-manifolds of its boundary, each of which is a disjoint union of principal totally geodesic tori, and let $f \co \partial \to \partial'$ be a homeomorphism. Assume that the holonomy matching condition~\eqref{eqn:matching} is satisfied along all components of $\partial$. Then $\mathscr M_f$ admits a projective structure in which all connected components are properly convex. The natural map $\mathscr M \to \mathscr M_f$ is isotopic to a projective map.
\end{theorem}

When $\mathscr M = M \sqcup M$ is the disjoint union of two copies of the same properly convex projective manifold $M$ with principal totally geodesic boundary and $f \co \partial M \to \partial M$ is the identity map, then the holonomy matching condition~\eqref{eqn:matching} is trivially satisfied. In this case $\mathscr M_f = 2M$ is the double of $M$. Hence Theorem~\ref{thm:deform} and Theorem~\ref{thm:convex-gluing} imply Theorem~\ref{thm:main}.

In general, any given projective manifold $\mathscr M$ and homeomorphism $f \co \partial \to \partial'$ as in the hypotheses of Theorem~\ref{thm:convex-gluing} are unlikely to satisfy the holonomy matching condition~\eqref{eqn:matching}. To find a properly convex projective structure on the glued manifold $\mathscr M_f$, one may attempt to deform the properly convex projective structures on the connected components of $\mathscr M$ with the goal of aligning the geometry of the principal totally geodesic boundary tori of $\partial$ and $\partial'$ so that~\eqref{eqn:matching} is satisfied. However, global deformation theory problems such as this are in general very difficult. 
Consider for example the case that $\mathscr M = M_1 \sqcup M_2$ is a disjoint union of two three-manifolds, each with boundary homeomorphic to a torus $\partial = \partial M_1$ and $\partial' = \partial M_2$, and that $f \co \partial \to \partial'$ is any fixed gluing homeomorphism. 
 The space of representations $\Rep(\pi_1 \partial_j, A)$ into the diagonal subgroup $A \subset \PGL_4 \RR$ is six-dimensional. Furthermore, the subset of representations which extend to $\pi_1 M_1$ is a half-dimensional (Lagrangian) subvariety as is the subset of representations which extend via $f_*$ to $\pi_1 M_2$. Therefore, the expected intersection between these two sets is zero-dimensional.  We do not know any reason in general to expect this intersection to be non-empty. On the other hand, no example seems to be known in which the intersection turns out to be empty. An experimental study of some basic cases could prove enlightening; the authors hope to conduct such experiments in future work.
Of course, even if the matching problem~\eqref{eqn:matching} is solved at the level of representations, one needs to find convex projective structures realizing those representations as their holonomy representations in order to apply Theorem~\ref{thm:convex-gluing}. The following theorem shows that any deformation of the holonomy representation of a properly convex projective manifold through representations with diagonalizable peripheral holonomy is indeed the holonomy representation of a properly convex projective structure with totally geodesic boundary.

\begin{theorem}\label{thm:closedness}
Let $M$ be a complete finite volume hyperbolic three-manifold. Then the set of holonomy representations of properly convex projective structures on $M$ with principal totally geodesic boundary is closed in the subspace of representations in $\Hom(\pi_1 M, \PGL_4 \RR)$ whose restriction to $\pi_1 \partial M$ is diagonalizable.
\end{theorem}

Indeed, this theorem is not difficult using our techniques. Any convex projective structure on $M$ with totally geodesic boundary yields a convex projective structure on the double $2M$ by Theorem~\ref{thm:convex-gluing}. By a theorem of Benoist~\cite{BenoistCD3}, the space of holonomy representations of convex projective structures on this closed manifold $2M$ is closed in $\Hom(\pi_1 2M, \PGL_4 \RR)$ and any deformation of the original structure on $2M$ may be cut back into pieces with principal totally geodesic boundary.

\subsection{Gluing together covers of reflection orbifolds}

In Section~\ref{reflectionorbs}, we apply the work of Benoist~\cite{BenoistCD4} and Marquis~\cite{Mar} and Theorem~\ref{thm:convex-gluing} to produce many new examples of convex projective structures on non-hyperbolic three-manifolds $N$ which, by contrast to Theorem~\ref{thm:main}, are not doubles. In each of these examples, the pieces $\{M_i\}$ of the JSJ decomposition of $N$ come from covers of a reflection orbifold $\mathcal O$ with a cusp $\partial$ isomorphic to the Euclidean $(3,3,3)$--triangle orbifold. Any convex projective structure on $\mathcal O$ induces a convex projective structure at infinity on $\partial$, the space of which is well-known to be one-dimensional. Considering only projective structures on the pieces $M_i$ which cover a convex projective structure on such an orbifold $\mathcal O$ greatly simplifies the holonomy matching problem of Theorem~\ref{thm:convex-gluing}. This technique allows us, for example, to identify manifolds in the census of tetrahedral cusped hyperbolic manifolds and gluing maps so 
that the resulting
glued up manifold admits properly convex projective structures. See~Section~\ref{reflectionorbs} for precise results.

\subsection*{Acknowledgements}
We are thankful for helpful conversations with Yves Benoist, Suhyoung Choi, Daryl Cooper, Ludovic Marquis, Maria Beatrice Pozzetti and Anna Wienhard. We thank Joan Porti for suggesting the proofs of Theorem~\ref{smoothnessthm} and Lemma~\ref{H^2injects} and for other helpful comments. We also thank Daryl Cooper, Darren Long and Stephan Tillmann for making a preliminary version of their results~\cite{CoLoTi} available to us. This work grew out of discussions at the GEAR network retreats of 2012 and 2014. Additionally, some of this work was done during the spring 2015 semester on Dynamics on Moduli Spaces of Geometric Structures at MSRI and we thank MSRI for support and for the stimulating environment there. Finally, we thank the referees for suggesting some improvements.

S. Ballas was partially supported by the National Science Foundation under the grant DMS 1709097. J. Danciger was partially supported was partially supported by the National Science Foundation grants DMS 1103939 and DMS 1510254 and by an Alfred P. Sloan Foundation fellowship. G.-S. Lee was supported by the DFG research grant ``Higher Teichm\"{u}ller Theory'' and by the European Research Council under ERC-Consolidator Grant 614733. The authors acknowledge support from U.S. National Science Foundation grants DMS 1107452, 1107263, 1107367 \lq\lq RNMS: GEometric structures And Representation varieties\rq\rq\,(the GEAR Network).



\section{Background}

\subsection{Properly convex domains}\label{ssec:propconvex}

The $n$--dimensional real projective space $\RP^n$ is the quotient of $\RR^{n+1}\bs \{0\}$ by the action of $\RR^\times$ by scaling. A point in $\RP^n$ is an equivalence class $[v]$ of vectors $v\in \RR^{n+1}\bs \{0\}$. The projective general linear group $\PGL_{n+1} \RR$, the quotient of $\textrm{GL}_{n+1} \RR$ by its center, acts faithfully on $\RP^n$. 

The image, or projectivization, of a two-dimensional vector subspace under the quotient map is called a \emph{projective line} and the image of an $n$--plane in $\RP^n$ is called a \emph{projective hyperplane}. Each hyperplane is determined by a unique projective class of linear functionals on $\RR^{n+1}$, so the space of hyperplanes identifies with the dual projective space $\RP^{n \ast}$, which is the projective space associated to $(\RR^{n+1})^\ast$. This correspondence is known as \emph{projective duality}. The natural action of $\GL_{n+1} \RR$ on $(\RR^{n+1})^\ast$ descends to a faithful action of $\PGL_{n+1} \RR$ on $\RP^{n\ast}$.

Let $H$ be a projective hyperplane. The complement of $H$ in $\RP^n$ is called an \emph{affine patch} and is denoted $\AF_H$. Up to change of coordinates by a projective transformation, \ie an element of $\PGL_{n+1} \RR$, any affine patch $\AF_H$ may be identified with the standard affine patch:
$$\{ [x_1, \ldots, x_n, 1] \in \RP^n \mid (x_1,\ldots,x_n)\in \RR^n\}.$$

A subset $\Omega$ of $\RP^n$ is called \emph{convex} if it is contained in some affine patch (\ie is disjoint from a projective hyperplane) and its intersection with every projective line is connected. If in addition its closure $\overline{\Omega}$ is contained in an affine patch then $\Omega$ is \emph{properly convex}. Equivalently, $\Omega$ is properly convex if $\Omega$ does not contain any complete affine line. The boundary $\partial \Omega := \overline \Omega \setminus \mathrm{Int}(\Omega)$ is said to be \emph{strictly convex} at $p \in \partial \Omega$ if $p$ is not contained in the interior of any affine line segment in $\partial \Omega$. If $\Omega$ is properly convex and $\partial \Omega$ is strictly convex at every point $p\in \partial \Omega$, then we say that $\Omega$ is \emph{strictly convex}.

Every open properly convex domain $\Omega$ gives rise to a dual domain
$$\Omega^\ast=\{[\phi] \in \RP^{n \ast}  \mid \phi(v) \neq 0\ \forall\ [v] \in \overline{\Omega}  \}.$$
It is easily verified that $\Omega^\ast$ is also open, non-empty, properly convex and that $(\Omega^\ast)^\ast = \Omega$. For each $p=[v]\in \partial \Omega$ there is a (possibly non-unique) $q=[\phi]\in \partial \Omega^\ast$ such that $\phi(v)=0$. The projective hyperplane dual to $q$ is called a \emph{supporting hyperplane} at $p$. A point $p\in \partial \Omega$ has a unique supporting hyperplane if and only if $p$ is a $C^1$ point of the boundary.

Let $\SS^n$ denote the $n$--sphere, realized as the quotient of $\RR^{n+1}\bs \{0\}$ by positive scaling, and let $\pi \co \SS^n\to \RP^n$ denote the 2-to-1 covering map. The automorphisms of $\SS^n$ are given by the linear transformations $\SL^{\pm}_{n+1} \RR$ with determinant $\pm 1$.  Let $[T]\in \PGL_{n+1} \RR$ be an equivalence class of linear transformations. By scaling we may arrange that $T\in \SL^{\pm}_{n+1} \RR$. Additionally, since $T\in \SL^{\pm}_{n+1} \RR$ if and only if $-T\in \SL^{\pm}_{n+1} \RR$, there is a 2-to-1 covering $\SL^{\pm}_{n+1} \RR \to \PGL_{n+1} \RR$, which by abuse we also call $\pi$, given by $\pi(T)= [T]$.  If $\Omega$ is a properly convex domain then we let $\SL^{\pm}(\Omega)$ and $\PGL(\Omega)$ denote the subgroups of
$\SL^{\pm}_{n+1} \RR$ and $\PGL_{n+1} \RR$ preserving $\pi^{-1}(\Omega)$ and $\Omega$, respectively. Then $\pi$ restricts to a 2-to-1 covering homomorphism from $\SL^{\pm}(\Omega)$ onto $\PGL(\Omega)$.
When $\Omega$ is properly convex, a homomorphic section of $\pi$ is constructed as follows. Since $\Omega\subset \RP^n$ is properly convex, the preimage of $\Omega$ under $\pi$ will consist of two connected components. An element $T \in \SL^{\pm}(\Omega)$ either preserves both of these components individually or it interchanges them. Furthermore, $T$ preserves both components if and only if $-T$ interchanges them. The desired section of $\pi$ is defined by mapping $[T]\in \PGL(\Omega)$ to the unique lift of $[T]$ to $\SL^{\pm}(\Omega)$ that preserves both components of $\pi^{-1}(\Omega)$.  Using this section we are able to identify $\PGL(\Omega)$ with a subgroup of $\SL^{\pm}(\Omega)$. We may therefore regard elements of $\PGL(\Omega)$ as linear transformations when convenient.

\subsection{Projective structures on manifolds}
Let $M$ be an  $n$--manifold.  A~\emph{projective atlas} on $M$ is a collection of charts, $\phi_\alpha \co U_\alpha\to \RP^n$, that cover $M$ with the property that if $U_\alpha$ and $U_\beta$ are charts with non-empty intersection then $\phi_\alpha \circ \phi_\beta^{-1}$ is locally the restriction of an element of $\PGL_{n+1} \RR$. Every projective atlas determines a unique maximal projective atlas and we call a maximal projective atlas on $M$ a \emph{projective structure} on $M$. In other words, a projective structure on $M$ is a $(G,X)$ structure on $M$ (see Ratcliffe~\cite{Rat} for an introduction to $(G,X)$ structures) where $(G,X)$ is real projective geometry: $G = \PGL_{n+1} \RR$ and $X = \RP^n$. A manifold equipped with a projective structure is called a \emph{projective manifold}. Note that a projective manifold is also a smooth manifold.

If $M$ and $M'$ are projective manifolds of the same dimension, then a continuous map $f \co M\to M'$ is \emph{projective} if for each pair of charts $\phi \co U\to \RP^n$ of $M$ and $\psi \co V\to \RP^n$ of $M'$ such that $U\cap f^{-1}(V)\neq \emptyset$ the map $\psi\circ f\circ\phi^{-1} \co \phi(U\cap f^{-1}(V))\to \psi(f(U)\cap V)$ agrees with an element of $\PGL_{n+1} \RR$ on each connected component. Such a map is necessarily smooth. If in addition $f$ is a diffeomorphism we say that $f$ is a \emph{projective equivalence}.

After fixing a universal covering $\widetilde M \to M$ the local data of a projective structure may be replaced, via analytic continuation of the charts, by global data $(dev,\rho)$, where $dev \co \widetilde M\to \RP^n$ is a local diffeomorphism called a \emph{developing map} which is equivariant with respect to a representation $\rho \co \fund{M}\to \PGL_{n+1} \RR$ called the \emph{holonomy representation} in the sense that
$$dev(\gamma x)=\rho(\gamma)dev(x),$$
for all $\gamma\in \fund{M}$ and all $x\in \widetilde M$.  Any two developing maps $dev$ and $dev'$ for the same structure satisfy that $dev'=g\circ dev$ for some $g\in \PGL_{n+1} \RR$; the respective holonomy representations are related by conjugation: $\rho'=g\rho g^{-1}$.
More generally, two projective structures on $M$ are considered equivalent if developing maps $dev, dev'$ for the respective structures are related by the equation $dev' = g \circ dev \circ \widetilde \varphi$ where $\widetilde \varphi$ is the lift to $\widetilde M$ of a diffeomorphism $\varphi$ of $M$ that is isotopic to the identity.

Suppose that we are given a projective structure on $M$ with development pair $(dev,\rho)$. If $dev$ is a diffeomorphism onto a convex (resp. properly convex  or strictly convex) domain  $\Omega$ of $\RP^n$ then we say that the projective structure is \emph{convex} (resp. \emph{properly convex}  or \emph{strictly convex}). In this case the holonomy representation $\rho$ is faithful with image a discrete subgroup of $\PGL(\Omega)$. Here are some useful equivalent characterizations of convexity.
\begin{theorem}[{Goldman~\cite[Proposition 3.1]{Go}}]\label{convexitycriteria}
 Let $M$ be a projective $n$--manifold. Then the following are equivalent:
 \begin{enumerate}
  \item $M$ is convex.
  \item Every path in $M$ is homotopic (rel endpoints) to a \emph{unique} geodesic segment, \ie a segment that develops into a projective line.
  \item $M$ is projectively equivalent to the natural projective structure on $\Gamma \backslash \Omega$ where $\Omega = dev(\widetilde{M}) \subset \RP^n$ is convex and $\Gamma = \rho(\fund{M}) \subset \PGL(\Omega)$ is a discrete group acting properly discontinuously and freely on $\Omega$.
 \end{enumerate}

\end{theorem}

\begin{remark}
 If $M=\Gamma\backslash \Omega$ is properly convex then $M$ can be equipped with the Hilbert metric (see de la Harpe~\cite{dlH}). Segments of projective lines, also known as projective geodesics, are always geodesics in the Hilbert metric. If $\Omega$ is strictly convex, then Hilbert geodesics are always projective geodesics.
\end{remark}

Via the Klein model, the hyperbolic $n$--space $\HH^n$ is realized as a ball in $\RP^n$ and its group of isometries is realized as $\PO(n,1)\subset \PGL_{n+1} \RR$ (see Ratcliffe~\cite{Rat} for details). Thus a hyperbolic structure on a manifold is a projective structure and a \emph{complete} hyperbolic structure is a strictly convex projective structure. Incomplete hyperbolic structures could be convex, \eg the interior of the convex core of a convex cocompact hyperbolic structure, or could fail to be convex. For example, given a cusped hyperbolic manifold $M$, Thurston's generalized Dehn surgery space~\cite{Thu} consists of one complete structure, which is strictly convex, and many incomplete structures, none of which are convex. Some of those incomplete structures can be completed to give hyperbolic structures on Dehn fillings of $M$, but even those incomplete structures are not convex: the developing map is an infinite sheeted covering onto the complement of a discrete countable collection of lines.

\subsection{Deformation theory and projective manifolds with generalized cusps} \label{subsec:gen-cusps}

There is a natural correspondence between $(G,X)$ structures and their holonomy representations often called the \emph{holonomy principle} or \emph{Ehresmann--Thurston principle}. The correspondence, originally discovered by Ehresmann~\cite{Ehresmann} and popularized by Thurston~\cite{Thu}, is a crucial tool in the study of deformation spaces of geometric structures.
We describe this holonomy principle in the projective setting as well as some relevant generalizations in the setting of convex projective structures. For this discussion we denote $G = \PGL_{n+1} \RR$.
Let $M$ be the interior of a compact $n$--manifold possibly with or without boundary, let $\Gamma = \fund{M}$, and let $\mathfrak D(M)$ be the space of equivalence classes of marked projective structures on $M$. Let $\Hom(\Gamma, G)$ denote the space of representations of $\Gamma$ into $G$ and let $\chi(\Gamma, G)$ denote its quotient by the $G$ action by conjugation. Let $hol \co \mathfrak D(M)\to \chi(\Gamma, G)$ denote the map that associates to each equivalence class of projective structure the conjugacy class of its holonomy representation. Each space is equipped with a natural topology; we refer the reader to Goldman~\cite{GoII} for details. When $M$ is closed, the holonomy principle states simply that the map $hol \co \mathfrak D(M) \to \chi(\Gamma, G)$ is a local homeomorphism; in other words, \emph{the small deformations of a projective structure are, up to equivalence, parameterized by small deformations of the conjugacy class of its holonomy representation}. In the case that $\partial M$ 
is non-empty, the principle holds as stated only
once the definition of equivalence of projective structures is relaxed so that two projective structures are considered equivalent if there is a diffeomorphism of $M$ which is projective, with respect to the one structure in the domain and the other in the target, away from some collar neighborhood of $\partial M$. In general, the holonomy principle does not guarantee any control of the geometry at the boundary; this is an important issue in many studies of deformations of geometric structures (\eg in the context of cone-manifold structures, see Hodgson--Kerckhoff~\cite{HoKe} or Danciger~\cite{Dan}), including this one.
Proofs of the holonomy principle are found in Canary--Epstein--Green~\cite{CaEpGr} or Goldman~\cite{GoII}.

In the context of this paper, we need a more powerful holonomy principle that allows for control of more refined geometric properties, specifically that of convexity. Koszul~\cite{Ko} proved that when $M$ is closed, proper convexity is an open condition: a small deformation of the holonomy representation of a properly convex projective structure is the holonomy representation of a nearby properly convex projective structure. However, when $\partial M \neq \emptyset$, the same statement fails. A simple example, that of incomplete hyperbolic structures nearby a complete hyperbolic structure, was already given in the previous subsection. We now describe recent work of Cooper--Long--Tillman~\cite{CoLoTi} which, given a properly convex projective structure satisfying certain assumptions, determines which deformations of the holonomy
representation are the holonomy representation of a nearby properly convex projective structure. Further, the result allows for some control over the geometry at the boundary. We note that related results of Choi~\cite{ChoiI,Choi2} on projective structures with radial ends also imply the holonomy principle needed in our context.

 A \emph{generalized cusp} is a properly convex manifold $B$ (with boundary) such that $B\cong \partial B\times [0,\infty)$, $\partial B$ is compact and strictly convex (\ie locally the graph of a strictly convex function), and $\fund{B}$ is virtually nilpotent.
 The manifold $\partial B$ is called the \emph{cusp cross section}. We now discuss some motivating examples for this definition. The first is an end of a finite volume hyperbolic manifold. In this case the strictly convex boundary $\partial B$ of the generalized cusp is the quotient of a horosphere by a virtually abelian group. The second example, which is more relevant for our purposes, is a regular neighborhood of a principal totally geodesic torus boundary component.    The following holonomy principle follows immediately from \cite[Theorem 0.1]{CoLoTi}.

\begin{theorem}\label{thm:koszul}
Let $M=M^c\cup\mathcal{V}$ be a connected $n$--manifold such that $M^c$ is compact and connected, $\mathcal{V}\cong \partial \mathcal{V}\times [0,\infty)$, and $\partial M^c=\partial \mathcal{V}$. Suppose that $V_1,\ldots, V_k$ are the components of $\mathcal{V}$ and that $\rho$ is the holonomy of a properly convex projective structure on $M$ in which each $V_i$ is a generalized cusp. If $\rho'$ is sufficiently close to $\rho$ in $\Hom(\pi_1 M, G)$ and for each $i$, $\rho'\vert_{\fund{V_i}}$ is the holonomy of a generalized cusp structure on $V_i$, then $\rho'$ is the holonomy of a properly convex structure on $M$ in which each $V_i$ is a generalized cusp.
\end{theorem}
We will use this theorem in Section \ref{boundarygeom} to prove Theorem \ref{thm:deform}.



\section{Smoothness}
\label{sec:smoothness}

We begin by setting some notation that will be used for the remainder of the article. Henceforth, we let $M$ denote an orientable finite volume non-compact hyperbolic three-manifold with $k$ torus cusps, which we denote by $\partial_i$ for $1\leq i\leq k$. Let $\Gamma = \pi_1 M$ denote the fundamental group of $M$, and for $1\leq i\leq k$ let $\Delta_i$ denote a representative of the conjugacy class of peripheral subgroups of the $i^{th}$ cusp. Each $\Delta_i$ is isomorphic to $\ZZ\times \ZZ$. We denote the groups $\PGL_4\RR$, $\SL^\pm_4\RR$ and $\SL_4\RR$ by $G$, $\tilde G$ and $\tilde G_0$, respectively. All three of these Lie groups are locally isomorphic and thus have the same Lie algebra which we will denote by $\mathfrak{g}$.

Let $\rhyp$ be a representative of the unique conjugacy class of discrete faithful representations of $\Gamma$ into the isometry group of hyperbolic three-space. Via the Klein model, the isometry group of hyperbolic three-space is realized as the subgroup $\PO(3,1)$ of $G$ that preserves the standard round ball $\HH^3$ in $\RP^3$. Hence we may regard $\rhyp$ as a point in the representation space $\Rep(\Gamma, G)$. We also, by abuse, regard $\rhyp$ as a point in the quotient $\Char(\Gamma, G) = \Rep(\Gamma, G)/G$ of the representation space by the action of $G$ by conjugation. The sense in which the quotient is taken may be left ambiguous; if one desires $\Char(\Gamma, G)$ to have the structure of an algebraic
variety, then the Mumford geometric invariant theory quotient is needed rather than the naive topological quotient. It is a standard fact that the two quotients agree (as topological spaces) locally near any irreducible representation such as $\rhyp$.

We will also need to study the representation space $\Rep(\Delta_i, G)$ of a peripheral subgroup $\Delta_i$, which has no irreducible representations. In this case the quotient by conjugation is not well-behaved. In the main sections of the paper, we will avoid such issues by working exclusively in
$\Rep$ rather than in the quotient~$\Char$.

This section is dedicated to some basic results about the local structure of $\Rep(\Gamma, G)$ and $\Char(\Gamma, G)$ near $\rhyp$. These results are straightforward and many of them seem to be well-known to experts although we are not aware of their existence anywhere in the current literature. Some analogous results from the context of deformations in $\SL(2,\CC)$ and $\SL(n, \CC)$ can be found in Boden--Friedl~\cite{BodenFriedl}, Heusener--Medjerab~\cite{HeusenerMed}, and Heusener--Porti--Su\'arez Peir\'o~\cite{HeusenerPortiSuarez}. The main result of this section shows that if a certain cohomological condition is satisfied (Definition \ref{def:infrig}), then $\rhyp$ is a smooth point of $\Rep(\Gamma, G)$ and $\Char(\Gamma, G)$. Since we are only interested in local behavior, it is equivalent and will be marginally less cumbersome to work in the representation space $\Rep(\Gamma, \tilde G)$ for the matrix group $ \tilde G$. Note that, since $\rhyp$ is the holonomy representation of a convex projective structure, it admits a unique lift, denoted by abuse again by $\rhyp$, to $\tilde G$ and so does every nearby representation (see Section \ref{ssec:propconvex}).

\subsection{Infinitesimal deformations and cohomology}

Let $\rho_t \co \Gamma\to \tilde G$ be a smooth path of representations into the matrix group $\tilde G $. Near $t= 0$, we may express the path $\rho_t$ as:
$$\rho_t(\gamma)=\exp( t u(\gamma)+O(t^2))\rho(\gamma),$$
where $u \co \Gamma\to \g$, defined by $\gamma\mapsto \left.\frac{d}{dt}\right|_{t=0}\rho_t(\gamma)\rho(\gamma)^{-1}$, is a group cocycle with coefficients in $\g$ twisted by the adjoint action of $\rho = \rho_0$. We denote the set of such cocycles, defined by the condition $u(\gamma_1\gamma_2) = u(\gamma_1) + \Ad_{\rho(\gamma_1)} u(\gamma_2)$, by $Z^1_{\rho}(\Gamma,\g)$ and refer to its elements as \emph{infinitesimal deformations} of $\rho$; it is the Zariski tangent space of $\Rep(\Gamma, \tilde G)$ at $\rho$. The representation $\rho$ is a smooth point of $\Rep(\Gamma, \tilde G)$ if and only if each infinitesimal deformation $u$ is \emph{integrable},
\ie $u$ is tangent to some path $\rho_t$ as above.
The space of coboundaries, denoted by $B^1_{\rho}(\Gamma, \g)$, is the subspace of those cocycles $b$ satisfying the \emph{infinitesimal conjugacy} condition: that there exists $v \in \g$ such that $b(\gamma) = v - \Ad_{\rho (\gamma)} v$. Each such coboundary $b$ is tangent to the conjugation path $\rho_t=c_t\rho c_t^{-1}$ at $t = 0$, where $c_t  = \exp(t v)$. The first cohomology group $H^1_{\rho}(\Gamma,\g)=Z^1_{\rho}(\Gamma,\g)/B^1_{\rho}(\Gamma,\g)$ with coefficients in $\g$ twisted by the adjoint action of $\rho$ describes the infinitesimal deformations of $\rho$ up to infinitesimal conjugacy. In the case that $\rho$ determines a smooth point of $\chi(\Gamma, \tilde G)$, this cohomology group describes its Zariski tangent space.

The higher cohomology groups, with twisted coefficients in $\g$, will not be of use to us except in the following subsection. Given an infinitesimal deformation $u$, there is an infinite sequence of obstructions to the integrability of $u$, each of which is an element of the second cohomology group $H^2_{\rho}(\Gamma, \g)$.

It will be important to understand the relationship between the deformations of the representations of $\Gamma$ with the deformations of representations of the peripheral subgroups~$\Delta_1, \ldots, \Delta_k$. The \emph{restriction map} $$\res \co \Rep(\Gamma, \tilde G) \to \Rep(\Delta_1,\tilde G)\times\ldots \times \Rep(\Delta_k,\tilde G)$$ is the product $\res = \res_1 \times \cdots \times \res_k$, 
where $\res_i \co \Rep(\Gamma, \tilde G) \to \Rep(\Delta_i, \tilde G)$ denotes the restriction map induced by the inclusion $\Delta_i \hookrightarrow \Gamma$ of the $i^{th}$ peripheral subgroup into $\Gamma$.
Each such map $\res_i$ defines a restriction map on group cohomology, $(\res_i)_* \co H^1_{\rho}(\Gamma,\g) \to H^1_{\res_i \rho}(\Delta_i,\g)$ and it is convenient to synthesize these into one linear map
 $$\res_* \co H^1_\rho(\Gamma, \g) \to \bigoplus_{i=1}^k H^1_{\res_i \rho}(\Delta_i, \g)$$ defined by 
$$
\res_* = (\res_1)_* \oplus \cdots \oplus (\res_k)_*.
$$
When clear, we will abuse notation using $\res$ and $\res_i$ to denote both the restriction map on representations and on cohomology. We will also conserve space using $H^*_{\rho}(\Delta_i, \g)$ to mean $H^*_{\res_i \rho}(\Delta_i, \g)$.

We note that, since $M$ is aspherical, the group cohomology groups $H^*_\rho(\Gamma, \g)$ coincide with the de Rham cohomology groups $H^*_\rho(M, \g)$ with coefficients in the flat $\g$-bundle over $M$ associated to $\rho$. Similarly, there is a natural identification between $H^*_\rho(\Delta_i, \g)$ and $H^*_\rho(\partial_i, \g)$ for each $1 \leq i \leq k$ and between $\bigoplus_{i=1}^kH^*_\rho(\Delta_i, \g)$ and $H^*_\rho(\partial M, \g)$. Although it will be more convenient for us to work with group cohomology, this identification makes available the tools commonly used in the 
study of cohomology of manifolds, \eg the long exact sequence of a relative pair and Poincar\'e duality (see \eg Heusener--Porti~\cite{HePo} for details).

\subsection{Infinitesimal rigidity implies smoothness}
 The following property, introduced and studied by Huesener--Porti~\cite{HePo}, is the critical assumption in Theorems~\ref{thm:main} and~\ref{thm:deform}.

\begin{definition}\label{def:infrig}
Let $\rhyp \co \Gamma \to \SO(3,1) \subset \tilde G$ denote the holonomy representation of the complete finite volume hyperbolic structure on $M$.
Then $M$ is called \emph{infinitesimally projectively rigid rel $\partial M$} if the restriction map $\res \co H^1_{\rhyp}(\Gamma, \g) \to \bigoplus_{i=1}^kH^1_{\rhyp}(\Delta_i, \g)$ is an injection.
\end{definition}
The main theorem of this section is:

\begin{theorem}\label{smoothnessthm}
Let $M$ be an orientable complete finite volume hyperbolic manifold with fundamental group $\Gamma$, and let $\rhyp \co \Gamma \to \SO(3,1) \subset \tilde G$ be the holonomy representation of the complete hyperbolic structure. If $M$ is infinitesimally projectively rigid rel $\partial M$, then $\rhyp$ is a smooth point of $\Rep(\Gamma, \tilde G)$ and its conjugacy class is a smooth point of $\Char(\Gamma, \tilde G)$.
 \end{theorem}

 \begin{remark}
Huesener--Porti~\cite{HePo} prove that the condition of infinitesimal projective rigidity rel boundary persists under infinitely many Dehn fillings. They then show that there are infinitely many examples of one-cusped hyperbolic 3--manifolds that satsify the condition by studying fillings of the Whitehead link, a two-cusped manifold that satisfies the condition. For example, the figure-eight knot complement as well as all but finitely many twist knots are infinitesimally projective rigid rel boundary. 
In future work, we hope to determine exactly which manifolds of the Hodgson--Weeks cusped census are infinitesimally projectively rigid rel boundary. 
A related rigidity condition in closed hyperbolic three-manifolds was studied in Cooper--Long--Thistlethwaite~\cite{CoLoThist2,CoLoThist1} and shown experimentally to hold very often in small examples.
\end{remark}

The proof of Theorem \ref{smoothnessthm} requires several lemmas. Note that Lemma~\ref{Z^2smoothness} does not require $M$ to be orientable, but that Lemma~\ref{H^2injects} does.
\begin{lemma}\label{Z^2smoothness}
For each $1\leq i\leq k$, the restriction ${\rhyp}_i = \res_i(\rhyp)$ of $\rhyp$ to the $i^{th}$ peripheral subgroup is a smooth point of $\Rep(\Delta_i, \tilde G)$.
Hence,  $\res(\rhyp)$ is a smooth point of $\Rep(\Delta_1,\tilde G) \times \ldots \times \Rep(\Delta_k,\tilde G)$.
\end{lemma}

\begin{proof}
 We may work in $\tilde G_0 = \SL_4 \RR$ in place of $\tilde G$, since ${\rhyp}_i$ has image in this smaller group.
  The variety $\Rep(\Delta_i, \tilde G_0)$ is the set of real points of a complex affine variety $\Rep(\Delta_i,\SL_4 \CC)$, which is defined over $\RR$. It therefore suffices to prove that ${\rhyp}_i$ is a smooth point of $\Rep(\Delta_i,\SL_4 \CC)$.
By work of Richardson~\cite[Theorem C]{Rich}, $\Rep(\Delta_i,\SL_4 \CC)$ is an irreducible (complex) affine variety and contains a dense set of representations whose images consist of diagonalizable representations. Thus $\Rep(\Delta_i,\SL_4 \CC)$ is an 18--dimensional complex variety. Heuristically this can be seen as follows: pick two generators $\gamma_1,\gamma_2$ for $\Delta_i$. An element $\rho \in \Rep(\Delta_i,\SL_4 \CC)$ can map $\gamma_1$ arbitrarily and thus contributes $15$ degrees of freedom. The only condition on $\rho(\gamma_2)$ is that it commute with $\rho(\gamma_1)$. Generically, the image of $\gamma_1$ will be diagonalizable with distinct eigenvalues, so the centralizer $Z(\rho(\gamma_1))$ is conjugate to the diagonal subgroup, which is three-dimensional.

To prove the Lemma, we must therefore show that the Zariski tangent space to $\Rep(\Delta_i,\SL_4 \CC)$ is 18--dimensional at ${\rhyp}_i$. The image of $\Delta_i$ under ${\rhyp}_i$ is conjugate to a lattice in the Lie group of matrices of the form:
  $$\begin{pmatrix}
     1 & u & v & \frac{1}{2}(u^2+v^2)\\
     0 & 1 & 0 & u\\
     0 & 0 & 1 & v\\
     0 & 0 & 0 & 1
    \end{pmatrix}
$$
A simple calculation shows that the infinitesimal centralizer $H^0_{{\rhyp}_i}(\Delta_i,\g)$ is 3--dimensional. By Poincar\'e duality (see \eg Hodgson--Kerckhoff~\cite{HoKe}),
$$\dim H^0_{{\rhyp}_i}(\Delta_i,\g) = \dim H^2_{{\rhyp}_i}(\Delta_i,\g),$$ and therefore since the Euler characteristic of $\partial M$ is zero, we have that $H^1_{{\rhyp}_i}(\Delta_i,\g)$ has dimension six. Furthermore,
$$\dim H^0_{{\rhyp}_i}(\Delta_i,\g) + \dim B^1_{{\rhyp}_i}(\Delta_i,\g) = \dim \g = 15,$$ so the coboundaries $B^1_{{\rhyp}_i}(\Delta_i,\g)$ are 12--dimensional. Hence the Zariski tangent space $Z^1_{{\rhyp}_i}(\Delta_i, \g)$ has dimension 18 as desired.
\end{proof}

\begin{remark}
Despite Lemma~\ref{Z^2smoothness}, the representation ${\rhyp}_i$ does \emph{not} determine a smooth point of $\Char(\Delta_i, \tilde G)$. Indeed, the conjugacy class of ${\rhyp}_i$ contains the trivial representation in its closure. Hence, the naive quotient of $\Rep(\Delta_i, \tilde G)$ by conjugation is not Hausdorff at this point. The Mumford GIT quotient avoids failure of Hausdorff-ness by identifying ${\rhyp}_i$ with the trivial representation, which is not even a smooth point of $\Rep(\Delta_i, \tilde G)$.
\end{remark}
Next, we derive some relevant information about second cohomology groups from the infinitesimal projective rigidity condition.

\begin{lemma}\label{H^2injects}
 If $M$ is infinitesimally rigid rel boundary then the map
 $$\res_* \co H^2_{\rhyp}(M,\g)\to H^2_{\rhyp}(\partial M,\g)$$
 is an injection.
\end{lemma}

\begin{proof}
 Consider the long exact sequence of the pair $(M,\partial M)$, where the first injection is by assumption:
\begin{align*}
 0\to H^1_{\rhyp}(M,\g)&\stackrel{\res_*}{\hookrightarrow} H^1_{\rhyp}(\partial M,\g)\\ {\to}
 &H^2_{\rhyp}(M,\partial M,\g)\to H^2_{\rhyp}(M,\g)\stackrel{\res_*}{\to} H^2_{\rhyp}(\partial M,\g)
\end{align*}
We showed above that $\dim H^1_{{\rhyp}_i}(\partial_i, \g) = 6$ for each $i$ and so we have that $\dim H^1_{\rhyp}(\partial M,\g)=6k$. By a standard Poincar\'e duality argument, known as ``half lives, half dies'' (see \eg Hodgson--Kerckhoff~\cite{HoKe}), the image of $H^1_{\rhyp}(M,\g)$ under the restriction map $\res_\ast$ is $3k$--dimensional and so $H^1_{\rhyp}(M,\g)$ is itself $3k$--dimensional. By Poincar\'e duality, we have that the group $H^2_{\rhyp}(M,\partial M, \g)$ is also $3k$--dimensional. We conclude that the map $H^1_{\rhyp}(\partial M,\g)\to H^2_{\rhyp}(M,\partial M,\g)$ must be a surjection and by exactness that  $\res_* \co H^2_{\rhyp}(M,\g)\to H^2_{\rhyp}(\partial M,\g)$ is also injective.
\end{proof}

We now prove Theorem \ref{smoothnessthm}.

\begin{proof}[Proof of Theorem \ref{smoothnessthm}.]
 To prove Theorem \ref{smoothnessthm} we must show that each infinitesimal deformation $u \in Z^1_{\rhyp}(\Gamma,\g)$ is integrable. It is well-known that integrability of $u$ follows if we can show that the infinitely many obstructions all vanish (see e.g. Heusener--Porti~\cite[Section 8.2.4]{HePo}). The obstructions to integrability are cohomology classes in $H^2_{\rhyp}(\Gamma, \g)$. However, if one of these obstruction classes is non-zero it would, by Lemma~\ref{H^2injects}, map to a non-zero class in $\bigoplus_{i=1}^k H^2_{\rhyp}(\Delta_i, \g)$ obstructing the integrability of $\res_* u$ in $\bigoplus_{i=1}^k H^1_{\rhyp}(\Delta_i, \g)$. By Lemma~\ref{Z^2smoothness}, this is impossible. Hence all the obstructions vanish and $u$ is an integrable infinitesimal deformation. This proves that $\Rep(\Gamma, \tilde G)$ is smooth at $\rhyp$. Since $\rhyp$ is an irreducible representation, the orbits of $\rhyp$ and all nearby representations are closed. Since the centralizer of all representations nearby $\rhyp$ is constant (equal to $\pm I$), we conclude that $\Char(\Gamma,\tilde G)$ is a manifold near the conjugacy class of $\rhyp$.
 \end{proof}

\subsection{The augmented restriction map}
In this subsection, we formulate and prove some results that are needed for the main transversality argument in the next section, where the basic goal will be to find deformations of the discrete faithful $\SO(3,1)$ representation with certain desired behavior along each of the peripheral subgroups. More specifically, we will construct a sub-manifold $\slice \subset \Rep(\ZZ \times \ZZ , \tilde G)$ consisting of representations with the desired behavior and then look for representations in $\Hom(\Gamma, \tilde G)$ whose restriction to each peripheral subgroup are conjugate into $\slice$. For technical reasons, in the execution of this strategy it will be more convenient to work with the following augmented restriction map.

\begin{definition}\label{augrest}
 Let $M$ be a finite volume hyperbolic 3--manifold with $k$ torus cusps and with fundamental group $\Gamma$. Then we define the \emph{augmented restriction map of $M$} denoted by $\widetilde \res \co \Rep(\Gamma,\tilde G)\times \tilde G^{k-1}\to \Rep(\Delta_1,\tilde G)\times \cdots \times \Rep(\Delta_k,\tilde G)$ by the formula $$(\rho,g_2,\ldots g_{k})\mapsto \left(\res_1(\rho),c(g_2) \cdot\res_2(\rho),\ldots,c(g_k)\cdot\res_{k}(\rho)\right),$$
 where $c(g)$ denotes the conjugation action by $g$.
\end{definition}

The main result about $\widetilde \res$ that we will need for the transversality argument in Section \ref{sec:diaghol} is:

\begin{theorem}\label{thm:submanifold}
Let $M$ be an orientable finite volume hyperbolic 3--manifold with $k$ cusps and with fundamental group  $\Gamma$, and let $\rhyp$ be the holonomy representation of the complete hyperbolic structure. Assume that $M$ satisfies Theorem~\ref{smoothnessthm} so that $\Rep(\Gamma, \tilde G)$ is smooth at $\rhyp$. Then the augmented restriction map 
$$\widetilde \res \co \Rep(\Gamma, \tilde G)\times \tilde G^{k-1} \to \Rep(\Delta_1, \tilde G)\times \cdots \times \Rep(\Delta_k, \tilde G)$$
is a local submersion onto a submanifold of codimension $3k$ at the point $(\rhyp, g_2, \ldots, g_k)$ where $g_2, \ldots, g_k$ are any elements of $\tilde G$.
\end{theorem}

\begin{proof}
Let $U$ be a smooth neighborhood of $\rho_{hyp}$ in $\Rep(\Gamma, \tilde G)$ whose elements are all irreducible representations, and let $W=U\times \tilde G^{k-1}$. The proof proceeds by showing that the rank of the augmented restriction map $\widetilde \res$ is constant and equal to $15k$ for the points in $W$. 

We begin by identifying the relevant tangent spaces. The tangent space $T_\rho \Rep(\Gamma,\tilde G)$ is $Z^1_\rho(\Gamma,\g)$ and the tangent space $T_\rho \Rep(\Delta_i,\tilde G)$ is  $Z^1_\rho(\Delta_i,\g)$. 
Since any $\rho \in U$ is irreducible we may identify $T_{g} \tilde G = B^1_{c(g)\rho}(\Gamma,\g)$ for any $g \in \tilde G$. As a result, for $p=(\rho ,g_2,\ldots,g_{k})\in W$, we may identify:   
\begin{equation*} 
 T_p\left(\Rep(\Gamma,\tilde G)\times \tilde G^{k-1}\right) = Z^1_\rho(\Gamma,\g)\oplus \bigoplus_{i=2}^{k}B^1_{c(g_i)\rho}(\Gamma,\g).
\end{equation*}
The following diagram commutes:
$$
\xymatrix{\ar[d]_{\varpi_1} Z^1_\rho(\Gamma,\g) \oplus \bigoplus_{i=2}^{k}B^1_{c(g_i)\rho}(\Gamma,\g)\ar[r]^-{\widetilde{\res}_*} & \ar[d]^{\varpi_2} Z^1_{\rho}(\Delta_{1},\g)  \oplus \bigoplus_{i=2}^{k} Z^1_{c(g_i)\rho}(\Delta_{i},\g)\\
H^1_\rho(\Gamma,\g)\ar[r]^{\res_*} & \bigoplus_{i=1}^k H^1_\rho(\Delta_i,\g)
}
$$
where $\varpi_1$ is projection to cohomology and $\varpi_2$ is projection to cohomology in each factor followed by the identification of $H^1_\rho(\Delta_i,\g)$ and $H^1_{c(g_i)\rho}(\Delta_i,\g)$ induced by $c(g_i)$. Then $\ker(\varpi_1)$ is given by  $B^1_\rho(\Gamma,\g)\oplus\bigoplus_{i=2}^{k}B^1_{c(g_i)\rho}(\Gamma,\g)$ and $\ker(\varpi_2)$ is given by $B^1_\rho(\Delta_1,\g)\oplus\bigoplus_{i=2}^{k}B^1_{c(g_i)\rho}(\Delta_{i},\g)$ so it is clear that $\widetilde \res_*(\ker(\varpi_1))=\ker(\varpi_2)$. 
After possibly shrinking $U$ we may assume for any $\rho \in U$ that for each $i$, the dimension of $H^1_{\rho}(\Delta_i,\g)$ is $6$, that $\Rep(\Delta_i,\tilde G)$ is smooth and $18$--dimensional at $c(g_i)\res_i(\rho)$, and therefore that the dimension of $\ker(\varpi_2)$ is $12k$. 
Furthermore, the ``half lives, half dies'' argument from Lemma \ref{H^2injects} shows that the rank of $\res_* \co H^1_\rho(\Gamma,\g)\to \bigoplus_{i=1}^kH^1_{\rho}(\Delta_i,\g)$, and hence of $\varpi_2\circ \widetilde \res_*$, is $3k$. Combining these facts we see that the rank of $\widetilde\res_*$ must be $15k$.  
\end{proof}

\begin{remark}
When $k = 1$, $\res = \widetilde \res$ are the same. However, for $k \geq 2$, while the local image of $\res$ is still a smooth submanifold, the codimension is larger than that of $\widetilde \res$.  
\end{remark}

%
%
%
%
\section{Diagonalizable peripheral holonomy}\label{sec:diaghol}
Recall the notation $G = \PGL_4 \RR$, $\tilde G = \SL^\pm_4 \RR$ and $\tilde G_0 = \SL_4 \RR$ from the previous section. Also recall that $M$ is an orientable finite volume hyperbolic three-manifold with $k$ torus cusps, $\Gamma=\pi_1 M$, the $i^{th}$ cusp is denoted $\partial_i$, and $\Delta_i$ is a peripheral subgroup for $\partial_i$. In this section we prove:

\begin{theorem}
\label{thm:diagonalizable}
Assume that $M$ is infinitesimally projectively rigid rel $\partial M$. Then there exists a smooth path of representations $\rho_t \in \Rep(\Gamma, \tilde G)$ such that $\rho_0 = \rhyp$ is the holonomy representation of the complete hyperbolic structure and $\rho_t(\Delta_i)$ is diagonalizable over the reals for all $t\neq 0$ and $i \in \{1, \ldots, k\}$.
\end{theorem}

 The proof is a transversality argument in the product of the representation spaces of the boundary tori.
 Let us give the rough idea in the simpler case that there is only one cusp, whose peripheral subgroup we denote by~$\Delta$: Theorem~\ref{thm:submanifold} gives that the restriction map $\res$ maps a neighborhood of $\rhyp$ in $\Rep(\Gamma, \tilde G)$ onto a submanifold of $\Rep(\Delta, \tilde G)$. This submanifold has codimension three and is smoothly foliated by conjugation orbits.
  Now, the key ingredient for the proof is the construction of a smooth four-dimensional partial slice $\slice$ in $\Rep(\Delta, \tilde G)$ which is transverse to $\res$ at $\rhyp$ and all of whose representations are either diagonalizable with real eigenvalues or lie in a unipotent (parabolic) subgroup of $\SO(3,1)$. The transverse intersection of $\slice$ with the image of  $\res$ gives a one-dimensional family in $\Rep(\Delta, \tilde G)$ through $\rhyp$, in which all representations, except the restriction of $\rhyp$, are diagonalizable. This one-dimensional family is the image of a one-dimensional path in $\Rep(\Gamma, \tilde G)$ as desired.

Let us begin the proof by describing the four-dimensional partial slice in the following general setting: Let $\Delta \cong \mathbb Z \times \mathbb Z$ and let $\rhyp$ be a representation taking $\Delta$ to a lattice in a unipotent subgroup of $\SO(3,1)$. The four-dimensional partial slice~$\slice$ is the image of the following map $\Phi \co \RR^4 \to \Rep(\Delta, \tilde G_0)$.
 We use coordinates $(a,b,x,y)$ on $\RR^4$ and generators $\gamma_1, \gamma_2$ for  $\Delta \cong \ZZ \times \ZZ$. Define:
\begin{align*} 
\Phi(a,b,x,y)(\gamma_1) &:=  \exp \begin{pmatrix}
     0 & 1 & 0 & 0\\
     0 & a & b & 1\\
     0 & b & -a & 0\\
     0 & 2(a^2+b^2) & 0 & 0
    \end{pmatrix}\\ \nonumber
   \Phi(a,b,x,y)(\gamma_2) &:= \exp  \begin{pmatrix}
     0 & x &y & 0\\
     0 & ax+by & bx-ay & x\\
     0 & bx-ay & -ax-by & y\\
     0 & 2(a^2+b^2)x & 2(a^2+b^2)y & 0
    \end{pmatrix}
\end{align*}
A simple computation checks that $\Phi(a,b,x,y)(\gamma_1)$ and $\Phi(a,b,x,y)(\gamma_2)$ commute.
We may assume using conjugacy that $\rhyp$ coincides with $\Phi(0,0,u,v)$ where $u+iv$ is the cusp shape of the cusp $\rhyp(\Delta)\backslash \HH^3$ with respect to the generators $\gamma_1$ and $\gamma_2$. When $a = b = 0$ and $x,y$ are allowed to vary, $\Phi$~gives a global slice for the discrete faithful unipotent $\SO_0(3,1)$ representations of $\Delta$; the restriction of $\Phi$ to the $xy$--coordinate plane is well-known to be transverse to conjugation. Geometrically, $\Phi(0,0,x,y)$ parameterizes the conjugacy classes of holonomy representations of all possible torus cusps of hyperbolic three-manifolds.

\begin{definition}
We refer to the collection of representations $\Phi(0,0,x,y)$, for all $y \neq 0$, as the \emph{cusp shape locus}.
\end{definition}

The following Lemma is the most important ingredient in the proof of Theorem~\ref{thm:diagonalizable}. Its proof will be given in the following subsection.

 \begin{lemma}\label{lem:slice}
Let $u,v \in \RR$ with $v \neq 0$. Then near the point $(a,b,x,y) = (0,0,u,v)$, the map $\Phi$ is a local immersion which is never tangent to conjugation orbits. Further, for each $(a,b) \neq (0,0)$, the representation $\Phi(a,b,x,y)$ is diagonalizable over the reals.
 \end{lemma}

The eigenvalues of the generators $\Phi(a,b,x,y)(\gamma_i)$ may be computed explicitly. They are most naturally described using certain branched polar coordinates around the cusp shape locus: $(a,b,x,y) = (t\cos 3\theta, t \sin 3\theta,x,y)$. In these coordinates the eigenvalues of $\Phi(t\cos 3\theta,t \sin 3\theta,x,y)(\gamma_1)$ are
 \begin{align*}
 (1, \exp(2t\cos \theta), \exp(-t(\cos \theta +& \sqrt{3} \sin \theta)),\\ \exp(t(-\cos \theta + \sqrt{3} \sin \theta))),
 \end{align*}
 and the eigenvalues of  $\Phi(t\cos 3\theta,t \sin 3\theta,x,y)(\gamma_2)$ (listed with respect to the same eigenbasis) are
 \begin{align*}
 ( 1, \exp(2t(x\cos \theta + y \sin \theta)), \exp(-t((x - \sqrt{3}y)\cos \theta + &(\sqrt{3}x + y) \sin \theta)),  \\ \exp(-t((x + \sqrt{3}y)\cos \theta + (-\sqrt{3}x + y) \sin \theta))).
 \end{align*}

 Observe that when moving away from the cusp shape locus (\ie increasing $t$ from zero) in any direction (\ie for any value of $3\theta$), the four eigenvalues of $\Phi(t\cos 3 \theta, t \sin 3 \theta,x,y)(\gamma)$ vary as smooth real-valued functions of~$t$ with distinct first derivative for at least some (generic) $\gamma \in \Delta$. However, more is required to show that $\Phi$ is not tangent to the conjugation orbit (see Remark~\ref{rem:mistake}). The complete proof of Lemma~\ref{lem:slice} will be given in the following subsection, together with a more geometric description of the representations $\Phi(a,b,x,y)$.

Let us now return to the context of Theorem~\ref{thm:diagonalizable}.
In order to state the next lemma, let us introduce a useful splitting of $\g$ (see Johnson--Millson~\cite{JohnMill} for more details). Let $\g = \so(3,1) \oplus \mathfrak v$ be the Killing-orthogonal splitting of the Lie algebra~$\g$; the splitting is invariant under the adjoint action of $\OO(3,1)$. Since the representation $\rhyp$ has image in the subgroup $\OO(3,1) \subset \tilde G$, all relevant cohomology groups split and the restriction map $\res \co H^1_{\rhyp}(\Gamma, \g) \to \bigoplus_{i=1}^k H^1_{{\rhyp}}(\Delta_i, \g)$ splits into the direct sum of the two maps:
\begin{align*}
\res_{\so(3,1)} \co & H^1_{\rhyp}(\Gamma, \so(3,1)) \to \bigoplus_{i=1}^k H^1_{{\rhyp}}(\Delta_i, \so(3,1)),\\
\res_{\mathfrak{v}} \co & H^1_{\rhyp}(\Gamma, \mathfrak{v}) \to \bigoplus_{i=1}^k H^1_{{\rhyp}}(\Delta_i, \mathfrak{v}).
\end{align*}
Note that for each $i \in \{1, \ldots, k\}$, $\dim H^1_{{\rhyp}}(\Delta_i, \mathfrak{v}) = 2$ (see Heusener--Porti~\cite[Section 5.1]{HePo}).

\begin{lemma}\label{lem:nontrivial-def}
There exists a cohomology class of infinitesimal deformations $[z] \in H^1_{\rhyp}(\Gamma, \mathfrak v)$ whose restriction $\res_i([z]) \in H^1_{{\rhyp}}(\Delta_i, \mathfrak v)$ is non-trivial for all $i \in \{1, \ldots, k\}$.
\end{lemma}
\begin{proof}
The image $L = \res(H^1_{\rhyp}(\Gamma, \mathfrak v))$ is a Lagrangian subspace (see \cite[Section 5.1]{HePo})  for the cup product pairing in $\bigoplus_{i=1}^k H^1_{{\rhyp}}(\Delta_i, \mathfrak{v})$ and thus has dimension $k$. Note that under the cup product pairing, the direct sum is orthogonal. The projection onto any single factor is not zero, or else $L$ would be a Lagrangian subspace of the direct sum of the complementary $k-1$ factors, which is impossible since $\dim L = k$.  Hence for each $j \in \{1, \ldots, k\}$, there exists $[z_j] \in H^1_{\rhyp}(\Gamma, \mathfrak v)$ such that $\res_j([z_j]) \neq 0$. Then some linear combination $[z]$ of $\{[z_j]\}_{j=1}^k$ satisfies the conclusion of the Lemma.
\end{proof}

 The next Lemma is a basic consequence of Calabi--Weil rigidity.
\begin{lemma}\label{lem:complex-evals}
Let $i \in \{1,\ldots, k\}$, let $\mu_i$ be a non-trivial element of $\Delta_i$, and let $[v] \in H^1_{\rhyp}(\Gamma, \so(3,1))$ have non-trivial restriction to $\mu_i$. Then the eigenvalues of at least one element of $\Delta_i$ (although possibly not $\mu_i$) become complex along any path in $\Rep(\Gamma, \tilde G)$ tangent to any cocycle representative $v \in Z^1_{\rhyp}(\Gamma, \g)$ of the class $[v]$.
\end{lemma}
\begin{proof}
Let $\rho_t$ be any path of representations which is tangent to $v$. Let us first assume that the representations $\rho_t$ remain in $\SO(3,1)$. Then the restriction $\res_i(\rho_t)$ lies in $\SO_0(3,1) \cong \PSL_2 \CC$ and we may regard the image under $\rho_t$ of any individual element as either a $4\times4$ real matrix or a $2\times2$ complex matrix. It follows easily from Calabi--Weil rigidity~\cite{calabi,WeilI,WeilII} (or see Kapovich~\cite{kapbook}) that the $\PSL_2 \CC$ traces of $\res_i \rho_t$ become complex to first order. Further, there exists an element $\nu_i \in \Delta_i$ such that the derivative of the $\PSL_2 \CC$ trace of $\rho_t(\nu_i)$ has imaginary part larger than its real part. It follows that the $\SL_4 \RR$ trace of $\rho_t(\nu_i)$ has strictly negative derivative. This first order trace behavior holds for any path of representations $\rho_t$ into $\tilde G = \SL_4^{\pm} \RR$ that is tangent to~$v$.
Therefore, for any sufficiently small time $t > 0$, the trace of $\rho_t(\nu_i)$ is strictly smaller than four and so $\rho_t(\nu_i)$ has at least one pair of complex eigenvalues (by the arithmetic mean vs geometric mean inequality applied to the eigenvalues).
\end{proof}

For each $i \in \{1, \ldots, k\}$, let $\slice_i \subset \Rep(\Delta_i, \tilde G)$ be a copy of the four-dimensional partial slice $\slice$ described above with $\Delta = \Delta_i$ and let $g_i \in \tilde G$ be such that $c(g_i)\cdot \res_i(\rhyp) \in \slice_i$. Without loss of generality we assume $g_1 = 1$.
Let $S = \slice_1 \times \cdots \times \slice_k \subset \Rep(\Delta_1, \tilde G) \times \cdots \times \Rep(\Delta_k, \tilde G)$, let $V_{\slice_i} \subset H^1_{\rhyp}(\Delta_i, \g)$ denote the subspace of cohomology classes of all infinitesimal deformations tangent to $\slice_i$, and let $V_S = V_{\slice_1} \oplus \cdots \oplus V_{\slice_k}$.
We now prove:

\begin{lemma}\label{lem:transverse}
The augmented restriction map $\widetilde \res$ is transverse to $S$ at $(\rhyp, g_2, \ldots, g_k)$ with $k$--dimensional local intersection.
\end{lemma}
\begin{proof}
By Lemma~\ref{lem:complex-evals}, the intersection $V_S \cap \res_*(H^1_{\rhyp}(\Gamma, \so(3,1)))$ is trivial, since none of the representations in $S$ have complex eigenvalues. Hence, since $\dim V_S = 4k$ by Lemma~\ref{lem:slice}, and $H^1_{\rhyp}(\Gamma, \so(3,1))$ and its image in $\bigoplus_{i=1}^k H^1_{{\rhyp}}(\Delta_i, \so(3,1))$ have dimension $2k$, we have that $$\bigoplus_{i=1}^k H^1_{{\rhyp}}(\Delta_i, \g) = V_S \oplus \res_*(H^1_{\rhyp}(\Gamma, \so(3,1))$$
and it follows that the subspaces $V_S$ and $\res_*(H^1_{\rhyp}(\Gamma, \g))$ intersect transversely in a $k$--dimensional subspace. Therefore the $4k$--dimensional tangent space to $S$ in $\bigoplus_{i=1}^k Z^1_{c(g_i){\rhyp}}(\Delta_i, \g)$ intersects the codimension $3k$ image of the augmented restriction map transversely in a $k$--dimensional subspace. The result follows.
\end{proof}

Finally we prove Theorem~\ref{thm:diagonalizable}.
\begin{proof}[Proof of Theorem~\ref{thm:diagonalizable}]
Lemma~\ref{lem:nontrivial-def} guarantees the existence of a cohomology class of infinitesimal deformation $[z] \in H^1_{\rhyp}(\Gamma, \mathfrak v)$ whose restrictions $(\res_i)_* [z]$ are, for each $i \in \{1, \ldots, k\}$, non-trivial in $H^1_{{\rhyp}}(\Delta_i, \mathfrak v)$. The span of $[z]$ and $H^1_{\rhyp}(\Gamma, \so(3,1))$ is a $(2k +1)$--dimensional subspace of $H^1_{\rhyp}(\Gamma, \g)$ whose restriction, also $(2k+1)$--dimensional by the assumption that $\res_*$ is injective, must intersect the codimension $2k$ subspace $V_S$ non-trivially (and indeed, transversally). Let $\res_* [u] = \alpha\res_* [z] + \res_* [w]$ be a non-trivial element of the intersection, where $[w] \in H^1_{\rhyp}(\Gamma, \so(3,1))$. Since $V_S \cap \res_*(H^1_{\rhyp}(\Gamma, \so(3,1))) = 0$, we must have $\alpha \neq 0$. In particular, for each $i \in \{1, \ldots, k\}$, $(\res_i)_* [u]$ does not lie in $H^1_{\rhyp}(\Delta_i, \so(3,1))$. Hence $(\res_i)_\ast [u] \in V_{\slice_i}$ is not tangent to the cusp shape locus. We may therefore find a representative cocycle $u \in Z^1_{\rhyp}(\Gamma, \g)$ and coboundaries $b_i \in B^1_{\rhyp}(\Delta_i, \g)$ for $i \in \{2, \ldots, k\}$ such that $\res_{1*}u$ is tangent to $\slice_1$ and for each $i \in \{2, \ldots, k\}$, $\res_{i*}u + b_i$ is tangent to $\slice_i$. By Lemma~\ref{lem:transverse} there exists a path $\rho_t \in \Rep(\Gamma, \tilde G)$ based at $\rho_0 = \rhyp$ with tangent $u$ at $t= 0$ and paths $g_{1,t} = 1$ (constant), $g_{2, t}, \ldots, g_{k,t}$ in $\tilde G$ with $g_{2,0} = g_2,\ldots,g_{k,0} = g_k$ such that for each $i \in \{1,\ldots,k\}$, $c(g_{i,t})\cdot \res_i \rho_t \in \Rep(\Delta_i, \tilde G)$ lies in $\slice_i$. Since $(\res_i)_\ast[u]$ is not tangent to the cusp shape locus, for sufficiently small $t  > 0$, each of the representations $c(g_{i,t}) \cdot \res_i \rho_t$ does not lie in the cusp-shape locus and is therefore diagonalizable by Lemma~\ref{lem:slice}. The theorem is proved.
\end{proof}
%
\begin{remark}
The properties characterizing the infinitesimal deformation $[z]$ from Lemma~\ref{lem:nontrivial-def} are stable under perturbation. Therefore there is an open $k$--dimensional cone of $[z] \in H^1_{\rhyp}(\Gamma, \mathfrak v)$ satisfying the conclusion of Lemma \ref{lem:nontrivial-def}.  The proof of Theorem~\ref{thm:diagonalizable} implies that this cone parametrizes a $k$--dimensional family of representations satisfying the conclusion of theorem.
\end{remark}

\subsection{More on the four-dimensional partial slice $\slice$}\label{subsec:slice}

In this section we give the proof of Lemma~\ref{lem:slice}, which describes the essential properties of the four-dimensional partial slice $\slice = \operatorname{Im}(\Phi)$ used in the transversality argument for Theorem~\ref{thm:diagonalizable}. Along the way, we will give a geometric description of the representations in $\slice$ and indicate some of the intuition behind its construction.

Let $C$ denote the two-dimensional abelian subgroup of $\SO(3,1)$ consisting of unipotent matrices fixing a point $p_\infty$ on the ideal boundary of hyperbolic space. Let us work in the paraboloid model of $\HH^3$, with the ideal boundary $\partial_\infty \HH^3$ described by the paraboloid $$\partial_\infty \HH^3 = \{[(x_1^2 + x_2^2)/2, x_1,x_2,1] \in \RP^n \mid x_1,x_2 \in \RR\} \cup [1,0,0,0],$$ and let us take $p_\infty = [1,0,0,0]$. Then each of the cusp-shape representations $\Phi(0,0,x,y)$, where $y \neq 0$, is a lattice inside $C$. 
Note that $C$ is contained in its centralizer $Z(C)$ in $\SL_4 \RR$, a three-dimensional abelian subgroup, maximal with respect to inclusion. Let $\mathfrak a$ denote the Cartan subalgebra of $\ssl_4(\RR)$ consisting of traceless diagonal matrices. Let $A = \exp \mathfrak a$ denote the corresponding Cartan subgroup. In order to find representations nearby $\res (\rhyp)$ which are diagonalizable, we must study the space of maximal (\ie three-dimensional) abelian subgroups of $\SL_4 \RR$ near $Z(C)$ and attempt to locate (at least some of) those that are conjugates of~$A$.

We now construct a smooth two-dimensional family of three-dimensional abelian subgroups $A_{a,b}$, which are conjugate to $A$ for all $(a,b) \neq (0,0)$, and such that $A_{0,0} = Z(C)$. We work in the affine chart with coordinates $[x_3, x_1, x_2, 1]$. For each $t >0$, consider the intersection, $S_t$, of the paraboloid $\partial_\infty \HH^3$ with the affine plane $P_t$ parallel to the $x_1 x_2$--plane at height $x_3 = 1/(2t^{2})$. In these coordinates, $S_t$ is a round circle contained in $P_t$, invariant under the rotation $R(\theta)$ by any angle $\theta$ in the $x_1 x_2$--plane about the $x_3$--axis. Let $p(t) = [1/(2t^2), 1/t ,0,1]$ and let $p_1(t,\theta) = R(\theta) p(t)$, let $p_2(t, \theta) = R(\theta+2\pi/3) p(t)$ and $p_3(t, \theta) = R(\theta+4\pi/3) p(t)$. Then $p_1(t,\theta), p_2(t,\theta), p_3(t,\theta)$ are the vertices of an equilateral triangle inscribed in $S_t$. 

\begin{figure}[ht!]
{
\centering

\def\svgwidth{4.8in}
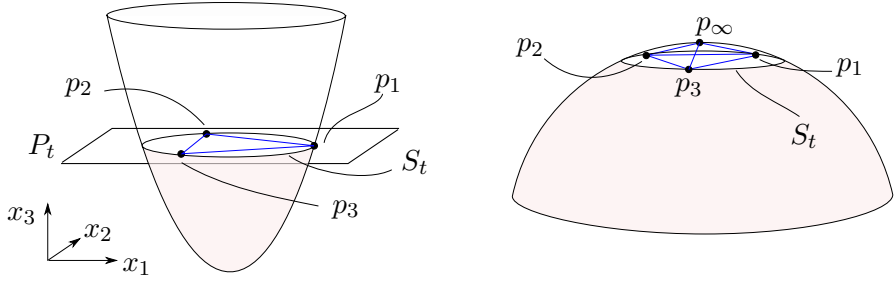

}
\caption[]{\emph{Left panel:} The points $p_1(t,\theta), p_2(t, \theta)$ and $p_3(t,\theta)$ form an equilateral triangle on the circle $S_t$ at height $1/(2t^2)$ on $\partial \HH^3$ in the paraboloid model. \emph{Right panel:} The same picture shown in the affine chart in which $\partial \HH^3$ is the round sphere. \label{paraboloid}}
\end{figure}

Let $A_{t,\theta}$ denote the subgroup of projective transformations which fix the vertices $p_1(t,\theta)$, $p_2(t,\theta)$, $p_3(t,\theta)$, $p_\infty$. Then $A_{t,\theta}$ is given by the explicit formula $A_{t, \theta} = R(\theta) Q(t) AQ(t)^{-
1}R(\theta)^{-1}$, where
$$Q(t) = \footnotesize \begin{pmatrix} 1 & 1 & 1 & 1\\ 2t & -t & -t & 0 \\ 0 & \sqrt{3} t & -\sqrt{3}t & 0\\ 2t^2 & 2t^2 & 2t^2 & 0 \end{pmatrix}, \ \ R(\theta) =  \footnotesize \begin{pmatrix} 1 & 0 & 0 & 0\\ 0 & \cos \theta & -\sin \theta & 0 \\ 0 & \sin \theta & \cos \theta & 0\\ 0 & 0 & 0 & 1 \end{pmatrix}$$
 Note that $A_{t, \theta} = A_{t, \theta + 2\pi/3}$.
The Hausdorff limit as $t \to 0$ of $A_{t,\theta}$ is exactly $Z(C)$, independent of $\theta$. Indeed, $(t,\theta)$ are branched polar coordinates for a smooth two-dimensional family of three-dimensional abelian subgroups of $\SL_4 \RR$.
To see this, consider the three families of traceless diagonal matrices $x_{t,\theta}, y_{t,\theta}, z_{t, \theta} \in \mathfrak a = \operatorname{Lie}(A)$:
\begin{align*} 
x_{t,\theta} &= \footnotesize \begin{pmatrix} 2t \cos\theta & & &\\ & 2t \cos(\theta + 2 \pi/3) & &\\ & & 2t \cos(\theta + 4 \pi/3) \\ & & & 0 \end{pmatrix},\\
y_{t,\theta} &= \footnotesize \begin{pmatrix} 2t \sin\theta & & &\\ & 2t \sin(\theta + 2 \pi/3) & &\\ & & 2t \sin(\theta + 4 \pi/3) \\ & & & 0 \end{pmatrix},\\
z_{t,\theta} &= \footnotesize \begin{pmatrix} t^2 & & &\\ & t^2 & &\\ & & t^2 \\ & & & -3t^2 \end{pmatrix}.
\end{align*}
For all $t \neq 0$ and any $\theta$, these three vectors form a basis of $\mathfrak a$. A simple computation shows that
\begin{align*}
x'_{t,\theta} := R(\theta) Q(t) x_{t,\theta} Q(t)^{-1}R(\theta)^{-1} &=
\footnotesize \begin{pmatrix} 0 & 1 & 0 & 0 \\ 0  & t \cos 3\theta & t \sin 3 \theta & 1\\ & t \sin 3 \theta & -t \cos 3 \theta& 0 \\ & 2t^2 & & 0 \end{pmatrix}\\
\end{align*}

\begin{align*}
y'_{t,\theta} := R(\theta) Q(t) y_{t,\theta}Q(t)^{-1}R(\theta)^{-1} &=
\footnotesize \begin{pmatrix} 0 & 0 & 1 & 0 \\ 0  & t \sin 3\theta & -t \cos 3 \theta & 0 \\ & -t \cos 3 \theta & -t \sin 3 \theta& 1 \\ & 0 & 2t^2 & 0 \end{pmatrix}.\\
z'_{t,\theta} := R(\theta) Q(t) z_{t,\theta}Q(t)^{-1}R(\theta)^{-1} &= \footnotesize \begin{pmatrix} -3t^2 & & & 2\\ & t^2 & &\\ & & t^2 \\ & & & t^2 \end{pmatrix}
\end{align*}
These three vectors form a basis for $\frak a_{t,\theta} = \operatorname{Lie}(A_{t, \theta})$.
We now set $a = t \cos 3 \theta$ and $b = t \sin 3\theta$ and rewrite in these coordinates:
\begin{align*}
x'_{a,b} &=
\footnotesize \begin{pmatrix} 0 & 1 & 0 & 0 \\ 0  & a & b & 1\\ 0 & b &  -a & 0 \\ 0 & 2(a^2 + b^2)& 0 & 0 \end{pmatrix}, \ 
 &y'_{a,b} =
\footnotesize \begin{pmatrix} 0 & 0 & 1 & 0 \\ 0  & b & -a & 0\\ 0 & -a & -b & 1 \\ 0 & 0 & 2(a^2+b^2) & 0 \end{pmatrix},\\
z'_{a,b} &= \footnotesize \begin{pmatrix} -3(a^2 + b^2) & & & 2\\ & a^2 + b^2 & &\\ & & a^2 + b^2\\ & & & a^2+b^2 \end{pmatrix}.
\end{align*}
In these coordinates it is transparent that the Lie algebra elements $x'_{a,b},y'_{a,b}, z'_{a,b}$ span a smooth (in fact algebraic) family of three-dimensional abelian subalgebras $\mathfrak a_{a,b} = \mathfrak a_{t,\theta}$ with $\mathfrak a_{0,0} = \operatorname{Lie}(Z(C))$. More relevant for Lemma~\ref{lem:slice}, for each $a,b \in \RR$, the Lie algebra elements $x'_{a,b}$ and $y'_{a,b}$ span a two-dimensional Lie subalgebra $\mathfrak c_{a,b}$ which generates a two-dimensional abelian subgroup $C_{a,b}$ of $\SL_4 \RR$. Indeed both $\mathfrak c_{a,b}$ and $C_{a,b}$ are smooth families and, of course, $C_{0,0} = C$.
\begin{proposition}\label{prop:algebras-transverse}
The maps $(a,b) \mapsto \mathfrak a_{a,b}$ and $(a,b) \mapsto \mathfrak c_{a,b}$ are both transverse to conjugation at $(a,b) = (0,0)$.
\end{proposition}
\begin{proof}
Since $\mathfrak c_{a,b} \subset \mathfrak a_{a,b}$ it is enough to show that the map $(a,b) \mapsto \mathfrak c_{a,b}$ is transverse to conjugation. To do so, we must simply show that the two-parameter family of projective classes $[x'_{a,b} \wedge y'_{a,b}]$ of bivectors in $\mathbb P \left(\wedge^2 \ssl_4\RR\right)$ is never tangent to the conjugation orbit at $[x'_{0,0} \wedge y'_{0,0}]$. This is a straightforward calculation. First, compute:
\begin{align*}
 \left.\partial_a ( x'_{a,b} \wedge y'_{a,b})\right|_{(a,b)=(0,0)} &= \footnotesize \begin{pmatrix} 0 &  & &  \\  & 1 & 0 & \\ & 0 &  -1 &  \\ & & & 0 \end{pmatrix} \wedge \begin{pmatrix} 0 & 0 & 1 & 0 \\   & 0 & 0 & 0\\ & 0 & 0 & 1 \\ &  &  & 0 \end{pmatrix}\\
 +&\footnotesize \begin{pmatrix} 0 & 1 & 0 & 0 \\ 0  & 0 & 0 & 1\\ & 0 &  0 & 0 \\ & & & 0 \end{pmatrix}\wedge \begin{pmatrix} 0 & 0 & 0 & 0 \\   & 0 & -1 & 0\\ & -1 & 0 & 0 \\ &  &  & 0 \end{pmatrix} \\
\end{align*}
\begin{align*}
 \left.\partial_b  (x'_{a,b} \wedge y'_{a,b})\right|_{(a,b)=(0,0)} &= \footnotesize \begin{pmatrix} 0 &  & &  \\  & 0 & 1 & \\ & 1 &  0 &  \\ & & & 0 \end{pmatrix} \wedge \begin{pmatrix} 0 & 0 & 1 & 0 \\   & 0 & 0 & 0\\ & 0 & 0 & 1 \\ &  &  & 0 \end{pmatrix}\\
 &+\footnotesize \begin{pmatrix} 0 & 1 & 0 & 0 \\ 0  & 0 & 0 & 1\\ & 0 &  0 & 0 \\ & & & 0 \end{pmatrix}\wedge \begin{pmatrix} 0 & 0 & 0 & 0 \\   & 1 & 0 & 0\\ & 0 & -1 & 0 \\ &  &  & 0 \end{pmatrix} \\
\end{align*}
Next, the infinitesimal action by conjugation of an arbitrary element $v = (v_{ij}) \in \mathfrak g$ on $x'_{0,0} \wedge y'_{0,0}$ is given by:
\begin{align*}
ad(v)x_{0,0}' \wedge y_{0,0}' +& x_{0,0}' \wedge ad(v) y_{0,0}'=\\ &\footnotesize \begin{pmatrix} -v_{21} & v_{11} - v_{22} & -v_{23} & v_{12} - v_{24} \\ -v_{41} & v_{21}-v_{42}& -v_{43} & v_{22} - v_{44} \\ 0 & v_{31} & 0 & v_{32} \\ 0 & v_{41} & 0 & v_{42} \end{pmatrix} \wedge \footnotesize  \begin{pmatrix} 0 & 0 & 1 & 0 \\   & 0 & 0 & 0\\ & 0 & 0 & 1 \\ &  &  & 0 \end{pmatrix} \\ + & \footnotesize \begin{pmatrix} 0 & 1 & 0 & 0 \\ 0  & 0 & 0 & 1\\ & 0 &  0 & 0 \\ & & & 0 \end{pmatrix} \wedge \begin{pmatrix} -v_{31} & - v_{32} & v_{11} - v_{33} & v_{13} - v_{34} \\ 0 & 0 & v_{21} & v_{23} \\ -v_{41} & - v_{42} & v_{31}-v_{43} & v_{33}-v_{44} \\ 0 & 0 & v_{41} & v_{43} \end{pmatrix}.
\end{align*}

Now, let $(e_{ij})$ be the usual basis for $\gl_4 \RR$, thought of as the space of $4 \times 4$ real matrices, and work in the basis for $\bigwedge^2 \mathfrak{gl}_4\RR \supset \bigwedge^2 \ssl_4\RR$ consisting of all $e_{ij} \wedge e_{mn}$ such that either $i < m$ or $i = m$ and $j < n$.
Suppose that some tangent vector to the family $[x'_{a,b} \wedge y'_{a,b}]$ is tangent to the conjugation orbit at $[x'_{0,0} \wedge y'_{0,0}]$. This is equivalent to the equation:
 $$\alpha \partial_a(x'_{a,b} \wedge y'_{a,b}) + \beta \partial_b (x'_{a,b} \wedge y'_{a,b}) = ad(v)x'_{0,0} \wedge y'_{0,0} + x'_{0,0} \wedge ad(v) y'_{0,0} + \gamma x'_{0,0} \wedge y'_{0,0}.$$
The coefficient of $e_{13}\wedge e_{33}$ on the left-hand side is $\alpha$ while the same coefficient on the right-hand side is zero. The coefficient of $e_{12}\wedge e_{22}$ on the left-hand side is $\beta$ while the same coefficient on the right-hand side is again zero. It follows that $\alpha = \beta = 0$: any vector tangent to both the family $[x'_{a,b} \wedge y'_{a,b}]$ and the conjugation orbit is trivial. The Proposition is proved.
\end{proof}

\begin{proof}[Proof of Lemma~\ref{lem:slice}]
First, let us rewrite $\Phi$ in the notation of this section:
\begin{align*}
\Phi(a,b,x,y)(\gamma_1) &= \exp(x'_{a,b})\\
\Phi(a,b,x,y)(\gamma_2) &= \exp(x x'_{a,b} + y y'_{a,b})
\end{align*}

That the representations $\Phi(a,b,x,y)$ are diagonalizable with real eigenvalues whenever $(a,b) \neq (0,0)$ has already been demonstrated. We must show that $\Phi$ is not tangent to the conjugation orbit in any direction at any point $(0,0,x,y)$.

Consider the tangent vector $$w = \dot a \partial_a \Phi + \dot b \partial_b \Phi + \dot x \partial_x \Phi + \dot y \partial_y \Phi$$
at the point $\Phi(0,0,x,y)$ and suppose that $w\in B^1_{\Phi({0,0,x,y})}(\Delta, \ssl_4 \RR)$. By Proposition~\ref{prop:algebras-transverse}, the two-dimensional subgroup $C_{a,b}$ generated by $\Phi(a,b,x,y)$ is changing to first order in the direction of $w$ if and only if $(\dot a, \dot b) \neq (0,0)$. So we must have $(\dot a,\dot b) = (0,0)$. Hence $w$ is tangent to the cusp shape locus. Since $w$ is a coboundary, the cusp shape must not change to first order. It follows that $\dot x = 0$ and $\dot y = 0$. This completes the proof of Lemma~\ref{lem:slice}.
\end{proof}

 \begin{remark}\label{rem:mistake}
For fixed $(a_0, b_0) \neq (0,0)$, the eigenvalues of $\Phi(t a_0, t b_0, u,v)$ are smooth functions of $t$ with non-zero derivative at $t = 0$. It is tempting to conclude that $\mathrm{d} \Phi_{(0,0,u,v)}(a_0, b_0,0,0)$ is not an infinitesimal conjugation. Although this turns out to be the case, it is fallacious to conclude so purely from the given behavior of eigenvalues; the above proof of Lemma~\ref{lem:slice} is needed. We give an example in the simpler setting of $2 \times 2$ matrices. Consider two paths of matrices $M_1(t) = \begin{pmatrix} e^t & 1 \\ 0 & e^{-t} \end{pmatrix}$ and $M_2(t) = \begin{pmatrix} 1+t & 1\\ -t^2 & 1-t \end{pmatrix}.$ Note that $M_1(0) = M_2(0)$ and that $M_1'(0) = M_2'(0)$. However, the eigenvalues of $M_1(t)$, which are $e^t = 1 + t + O(t^2)$ and $e^{-t} = 1 - t + O(t^2)$, vary to first order in $t$, while the eigenvalues of $M_2(t)$ are both constant equal to one for all $t$. Indeed $M_2(t)$ is a conjugation path.
\end{remark}

%
%
%

\section{Geometry of manifolds with totally geodesic boundary}\label{boundarygeom}

This section is dedicated to the geometry of convex projective three-manifolds with diagonalizable peripheral holonomy. We will complete the proof of Theorem~\ref{thm:deform}, the algebraic side of which was given in the previous section. Recall that we denote the Lie groups $\PGL_4\RR$, $\SL^{\pm}_4\RR$ and $\SL_4\RR$ by $G$, $\tilde G$ and $\tilde G_0$, respectively.

Let $M$ be an open three-manifold which is the interior of a compact boundary incompressible manifold with $k$ torus boundary components $\partial M = \partial_1 \sqcup \cdots \sqcup \partial_k$. Let $dev: \widetilde M \xrightarrow[]{\simeq} \Omega \subset \RP^3$ be the developing map of an indecomposable properly convex projective structure on $M$ and denote the holonomy representation $\rho \co \Gamma \to \tilde G$ where $\Gamma = \pi_1 M$ denotes the fundamental group. As usual denote the peripheral subgroups by $\Delta_i = \pi_1 \partial_i$. We will assume that
$\rho(\Delta_i)$ is diagonalizable over the reals with eigenvectors $p^{(i)}_1, p^{(i)}_2, p^{(i)}_3, p^{(i)}_4 \in \RP^3$ and further that $p^{(i)}_4$ is never an attracting fixed point (not even weakly attracting) for any non-trivial element of $\rho(\Delta_i)$. Then $\rho$ is said to satisfy the \emph{middle eigenvalue condition}, namely that for any nontrivial $\gamma \in \Delta_i$, the eigenvalue of $\rho(\gamma)$ associated to eigenvector $p^{(i)}_4$ is never the largest nor the smallest. The middle eigenvalue condition is a slight weakening of the uniform middle eigenvalue condition defined by Choi (see \cite{ChoiI} for more details).

The holonomy representation $\rho$ for this structure is also the holonomy representation for other (related) convex projective structures defined by different domains. Let us begin by constructing a minimal convex domain for the action of $\rho(\Gamma)$. Any strongly attracting fixed point of an element of $\rho(\Gamma)$ must lie on $\partial \Omega$. The closure of the set of strongly attracting fixed points is called the \emph{limit set} of $\rho(\Gamma)$ and we define $\Omega_{min}$ to be the interior of the closed convex hull of the limit set. Then $\Omega_{min}$ is contained in $\Omega$, is non-empty, open, $\rho$--invariant, convex, and is minimal with respect to these properties. Pulling back $\Omega_{min}$ via $dev$ determines a convex projective structure on a sub-manifold $M_{min} \cong \rho(\Gamma) \backslash \Omega_{min}$ whose inclusion into $M$ is a homotopy equivalence. We now investigate the geometry of $M_{
min}$, in particular of its ends.

\begin{lemma}\label{lem:attractors}
For any $i \in \{1, \ldots, k\}$, each of the three fixed points $p^{(i)}_1, p^{(i)}_2, p^{(i)}_3$ is a strongly attracting fixed point for some element of $\rho(\Delta_i)$. Therefore $p^{(i)}_1, p^{(i)}_2, p^{(i)}_3$ lie on $\partial \Omega$ and $\partial \Omega_{min}$.
\end{lemma}
\begin{proof}
The peripheral holonomy group $\rho(\Delta_i)$ is obtained by exponentiating a lattice, $\Lambda$, inside a two-dimensional Lie subalgebra, $\mathfrak c^{(i)}$, of the Cartan subalgebra $\mathfrak a^{(i)}\subset \mathfrak{g}$ corresponding to the basis $p^{(i)}_1, p^{(i)}_2, p^{(i)}_3, p^{(i)}_4 \in \RP^3$. 
Let $A^{(i)}_{123} = \exp \mathfrak a^{(i)}_{123}$ be the (2--dimensional) subgroup of $\SL^\pm(\operatorname{span} \{p^{(i)}_1, p^{(i)}_2, p^{(i)}_3\}) \cong \SL^\pm_3 \RR$ of elements that fix each of $p^{(i)}_1, p^{(i)}_2, p^{(i)}_3$. Consider the natural projection $\varpi_{123} \co \mathfrak a^{(i)} \to \mathfrak a^{(i)}_{123}$. 
It follows from the middle eigenvalue condition that the restriction of $ \varpi_{123}$ to $\mathfrak c^{(i)}$ is injective. To see this, observe that if $A\in \ker(\varpi_{123})\cap \mathfrak{c}^{(i)}$ then the eigenvalues of $A$ corresponding to the eigenvectors $p_1^{(i)}$, $p_2^{(i)}$, $p_3^{(i)}$ and $p_4^{(i)}$ are $\lambda$, $\lambda$, $\lambda$ and $\mu$, respectively. As a result we see that $\mu$ is either the smallest or largest eigenvalue of $A$ which contradicts the middle eigenvalue condition. Thus by dimensional considerations we conclude that $\varpi_{123}(\mathfrak c^{(i)}) = \mathfrak a^{(i)}_{123}$. Furthermore, if $A\in \mathfrak{c}^{(i)}$, $1\leq j\leq 3$, and the $p_j^{(i)}$--eigenvalue of $\varpi_{123}(A)$ is the largest eigenvalue then the $p_j^{(i)}$--eigenvalue of $A$ is also the largest eigenvalue for $A$. 

Let $1\leq j\leq 3$ and let $D_j$ be the subset of $\mathfrak{a}^{(i)}_{123}$ consisting of elements where the $p^{(i)}_j$--eigenvalue is the largest. It is easy to see that $D_j$ is a non-empty open cone, which implies that $D_j$ has non-trivial intersection with $\varpi_{123}(\Lambda)$. As a result we can find an element of $\rho(\Delta_i)$ such that the $p^{(i)}_j$--eigenvalue is the largest. Such an element has $p^{(i)}_j$ as an attracting fixed point, which completes the proof of the lemma.
\end{proof}

We make the following definition, following Goldman~\cite{Go} in the two-dimensional setting.
\begin{definition}\label{def:principal}
For each $i$, there is a unique (open) triangle $T^{(i)} \subset \overline \Omega_{min}$ spanned by the points $p^{(i)}_1, p^{(i)}_2, p^{(i)}_3$. Any $\rho(\Gamma)$ translate of $T^{(i)}$ will be called a \emph{principal totally geodesic triangle}. The group $\rho(\Delta_i)$ acts properly on $T^{(i)}$ and the quotient is called a \emph{principal totally geodesic torus}.
\end{definition}

We will show that $M_{min}$ admits a natural compactification whose boundary consists of principal totally geodesic tori.
First, we introduce a third convex domain $\Omega_{max}$ defined as follows. Let $\Omega^*$ denote the convex domain dual to $\Omega$. Then $\Gamma$ acts properly discontinuously  on $\Omega^*$ via~$\rho$ with diagonalizable peripheral holonomy, and we may perform the same construction as above: Let $(\Omega^*)_{min}$ denote the interior of the closed convex hull of the limit set for the $\rho(\Gamma)$ action on $\RP^{3\ast}$. We define $\Omega_{max}$ to be the convex domain dual to $(\Omega^*)_{min}$. It is the maximal properly convex, $\rho$--invariant domain because its dual is minimal.
Next, observe that the fixed points of the dual action of $\rho(\Delta_i)$ are the hyperplanes $P_1^{(i)}, P_2^{(i)}, P_3^{(i)},P_4^{(i)}$ spanned respectively by $\{p^{(i)}_2, p^{(i)}_3, p^{(i)}_4\}$, $\{p^{(i)}_1, p^{(i)}_3, p^{(i)}_4\}$, $\{p^{(i)}_1, p^{(i)}_2, p^{(i)}_4\}$ and $\{p^{(i)}_1,p^{(i)}_2,p^{(i)}_3\}$. The hyperplanes $P^{(i)}_1,P^{(i)}_2$ and $P^{(i)}_3$ are the attracting fixed points of $\rho(\Delta_i)$ in $\RP^{3\ast}$. Hence $P_1^{(i)}, P_2^{(i)}, P_3^{(i)}$ are points on the boundary of any convex domain in $\RP^{3\ast}$ preserved by $\rho$, including on $\partial(\Omega^*)_{min}$. Dually, they are three support hyperplanes for any convex domain in $\RP^3$ preserved by $\rho$, in particular for $\Omega$, $\Omega_{min}$ and $\Omega_{max}$. They bound a convex (but not properly convex) open triangular prism $U^{(i)}$ which is separated by $T^{(i)}$ into two components $\mathcal T^{(i)}_+, \mathcal T^{(i)}_-$, each of which is an open 
tetrahedron.

\begin{lemma}\label{lem:convexdomain}\
\begin{enumerate}
\item Each principal totally geodesic triangle is contained in  $\partial\Omega_{min}$.
\item $\Omega_{min}$ is the interior of the intersection of the positive half-spaces bounded by the planes $\rho(\Gamma)P^{(i)}_4$ containing principal totally geodesic triangles.
\item $\Omega_{max} \setminus \overline \Omega_{min}$ is the disjoint union over all $i \in \{1, \ldots, k\}$ and all $\gamma \in \Gamma / \Delta_i$, of open tetrahedra $\rho(\gamma)\mathcal T^{(i)}_-$, where $\mathcal T^{(i)}_-$ is the tetrahedron lying on the opposite side from $\Omega_{min}$ of the principal triangle $T^{(i)}$.
\end{enumerate}
\end{lemma}

\begin{figure}[ht!]
{
\centering

\def\svgwidth{3.2in}
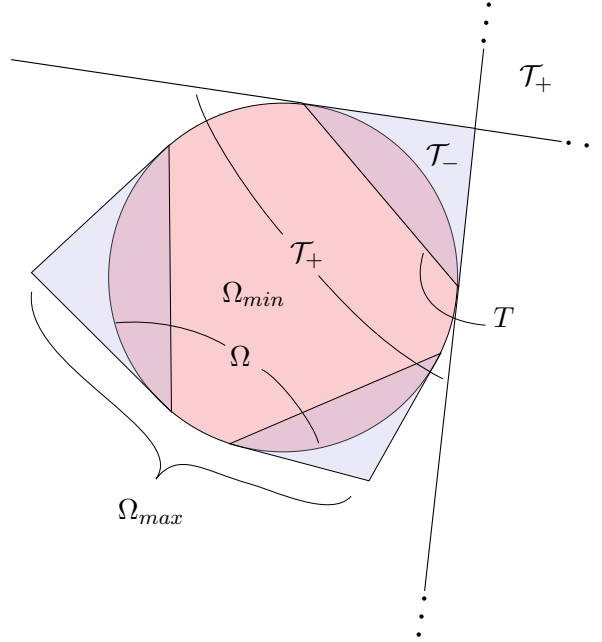

}
  \caption{\label{fig:twotets} A schematic of the configuration $\Omega_{min}\subset \Omega\subset \Omega_{max}$ in Lemma \ref{lem:convexdomain}. Here $\Omega_{min}$ is red, $\Omega$ is the union of red and purple, and $\Omega_{max}$ is the union of red, purple and blue.}
\end{figure}

\begin{proof}
Let $S^{(1)}, \ldots, S^{(k)}$ be a pairwise disjoint collection of embedded tori in $M$ which are parallel to the $k$ boundary components $\partial_1, \ldots, \partial_k$ respectively. 
 Let $D^{(i)}$ denote the lift to the universal cover $\widetilde M = \Omega$ of $S^{(i)}$ 
that is invariant under $\Delta_i$. Then $D^{(i)}$ is $(\Gamma,\Delta_i)$ \emph{precisely invariant}, meaning that $\rho(\gamma) D^{(i)} \cap D^{(i)} \neq \emptyset$ if and only if $\gamma \in \Delta_i$.
It follows from Lemma~\ref{lem:attractors} that $\partial D^{(i)} = \partial T^{(i)}$ and therefore $T^{(i)}$ is also $(\Gamma, \Delta_i)$ precisely invariant. Next, $D^{(i)}$ divides $\Omega$ into two components and one of these, the one which covers the end of $M$ bounded by $S^{(i)}$, is also $(\Gamma, \Delta_i)$ precisely invariant. 
 As a result, the limit set of $\rho(\Gamma)$ lies on one side of $\partial D^{(i)} = \partial T^{(i)}$ on $\partial \Omega$. Hence the limit set lies entirely in one of the two closed tetrahedra $\overline{\mathcal T^{(i)}_
+}$ or $\overline{\mathcal T^{(i)}_-}$ and we take the labeling convention that it is $ \overline{\mathcal T^{(i)}_
+} $. Let $\Omega^{(i)}_+ = \Omega \cap \mathcal T^{(i)}_+$. It then follows that
the intersection
$$\bigcap_{\gamma \in \Gamma} \bigcap_{i=1}^k \rho(\gamma)\overline{\Omega^{(i)}_+}$$
is a closed convex set containing the limit set of $\rho(\Gamma)$. We denote its interior, which must be non-empty, by~$\Omega'$. In fact, we will show soon that $\Omega'$ is the minimal convex domain $\Omega_{min}$. Until then we conclude the simple fact that $T^{(i)}$ is contained in $\partial \Omega_{min}$ because the vertices of $T^{(i)}$ must be contained in the boundary of any invariant convex domain and because $T^{(i)}$ lies in a support plane for $\partial \Omega'$.

Since the hyperplane $P^{(i)}_4$ containing $T^{(i)}$ is a support plane for $\Omega_{min}$, it is also a point in the boundary of the dual convex domain $(\Omega_{min})^* = (\Omega^*)_{max}$.
By applying the above to the dual convex domain $\Omega^*$ in place of $\Omega$, we see therefore that the point $p^{(i)}_4$ belongs to the boundary of $\Omega_{max}$. In particular, the entire tetrahedron $\mathcal T^{(i)}_-$ is contained in $\Omega_{max}$; see Figure~\ref{fig:twotets}. Since the action of $\rho(\Gamma)$ is properly discontinuous on $\Omega_{max}$ (or on any invariant open properly convex domain), we may now conclude that $\rho(\Gamma)$ acts properly on the subdomain $\Omega'_{\partial}$ consisting of the union of $\Omega'$ with all of the principal totally geodesic (open) triangles, and similarly on the union $\Omega_{min, \partial}$ of $\Omega_{min}$ with all of the principal totally geodesic triangles. The quotient of either set by $\rho(\Gamma)$ is a submanifold with boundary of $M_{max}:=\rho(\Gamma) \backslash \Omega_{max}$. The boundary of either is the collection of principal totally geodesic tori $\{ \rho(\Delta_i) \backslash T^{(i)}\}_{i=1}^k$. We conclude that $\Omega_{min} = \Omega'$. 
It
follows that the construction of $\Omega'$ is independent of the domain $\Omega$. The proof of (2) is thus completed by applying the above argument in the case that $\Omega$ is the interior of the intersection of the positive half-spaces bounded by the hyperplanes containing principal triangles. The third statement of the Lemma then follows immediately.
\end{proof}

\begin{figure}[ht!]
    \centering
    \subfloat[]{\includegraphics[scale=.3]{deform01_t}}
    \quad\;
    \subfloat[]{\includegraphics[scale=.3]{deform02_t}}

    \subfloat[]{\includegraphics[scale=.3]{deform03_t}}
    \quad\;
    \subfloat[]{\includegraphics[scale=.3]{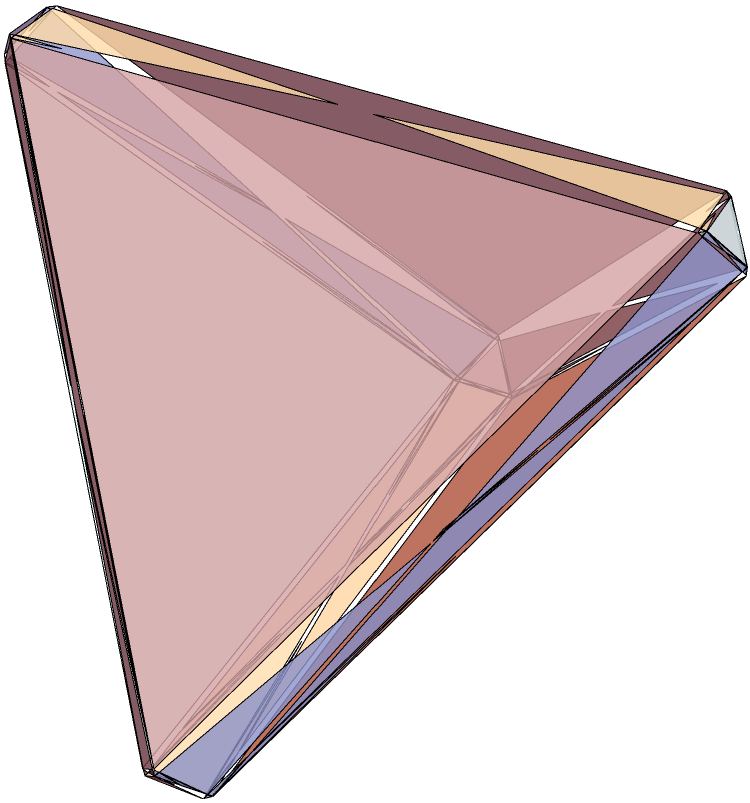}}
    \caption{The principal triangles on the boundary of $\Omega_{min}$ as the representation $\rho$ deforms away from the holonomy representation of the complete hyperbolic structure.}
    \label{fig:omegamin}
\end{figure}

\begin{definition}
The manifold $\overline M_{min} = \rho(\Gamma) \backslash \Omega_{min, \partial}$ is called a \emph{properly convex projective manifold with totally geodesic torus boundary}. By abuse, $M_{min} = \rho(\Gamma) \backslash \Omega_{min}$ will also be said to have totally geodesic boundary; see Figure \ref{fig:omegamin}.
\end{definition}

\begin{definition}\label{def:max}
The manifold $M_{max} = \rho(\Gamma) \backslash \Omega_{max}$ is the \emph{maximal thickening} of $M$. The tetrahedron $\mathcal T^{(i)}_-$ (or any of its orbits) is a \emph{principal tetrahedron} and its quotient by $\rho(\Delta_i)$ is a \emph{principal collar} of $M_{max}$; see Figure~\ref{fig:omegamax}.
\end{definition}

\begin{figure}[ht!]
    \centering
    \subfloat[]{\includegraphics[scale=.3]{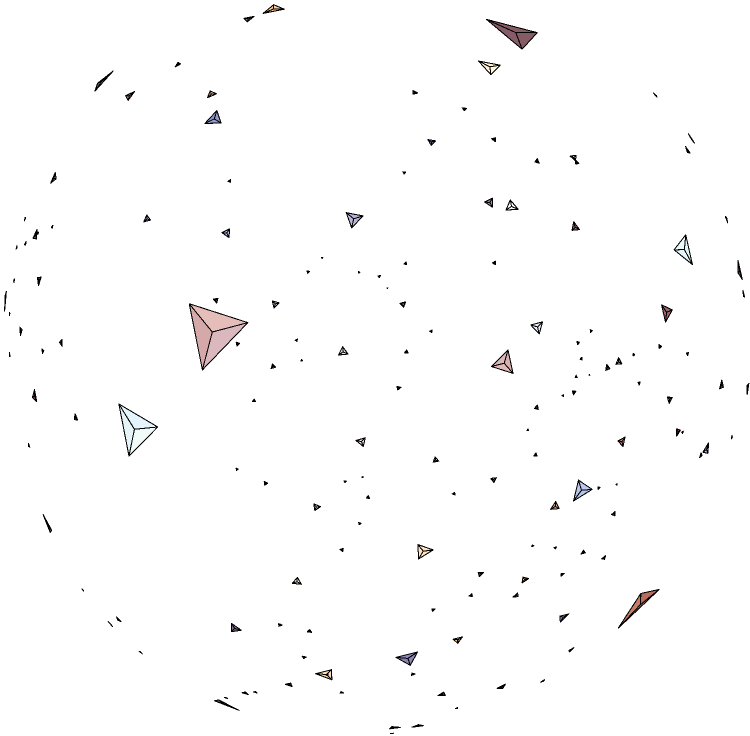}}
    \quad\;
    \subfloat[]{\includegraphics[scale=.3]{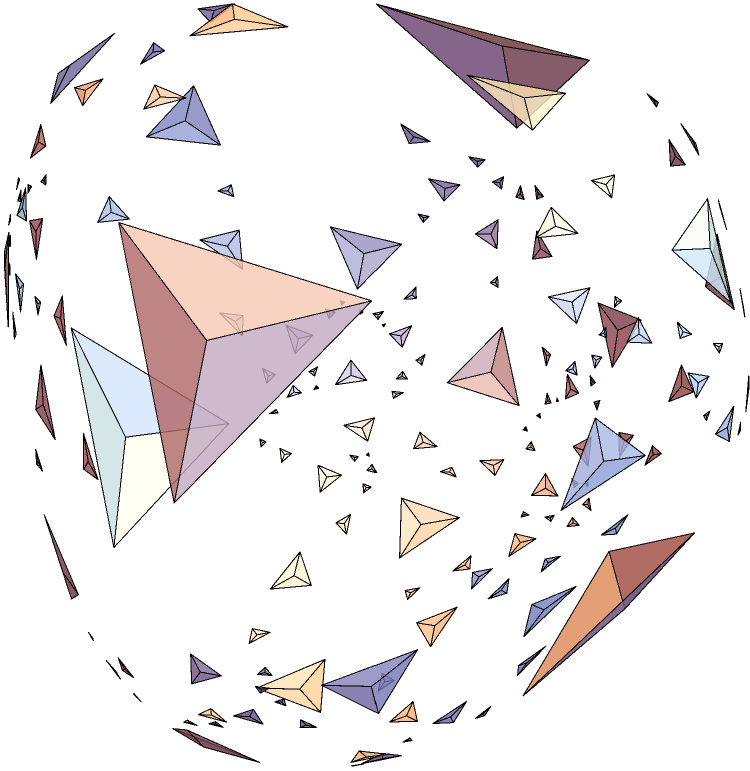}}

    \subfloat[]{\includegraphics[scale=.3]{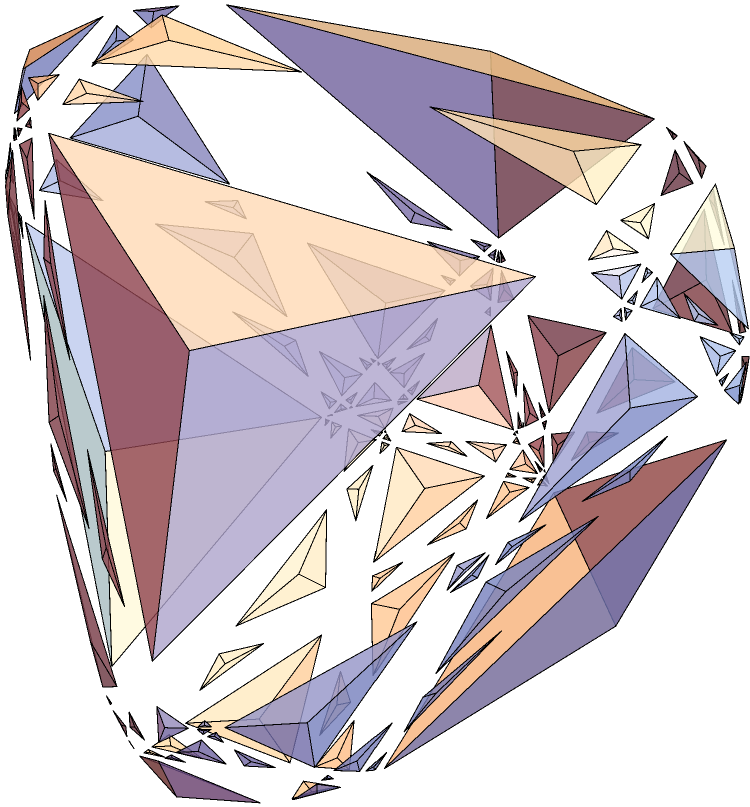}}
    \quad\;
    \subfloat[]{\includegraphics[scale=.3]{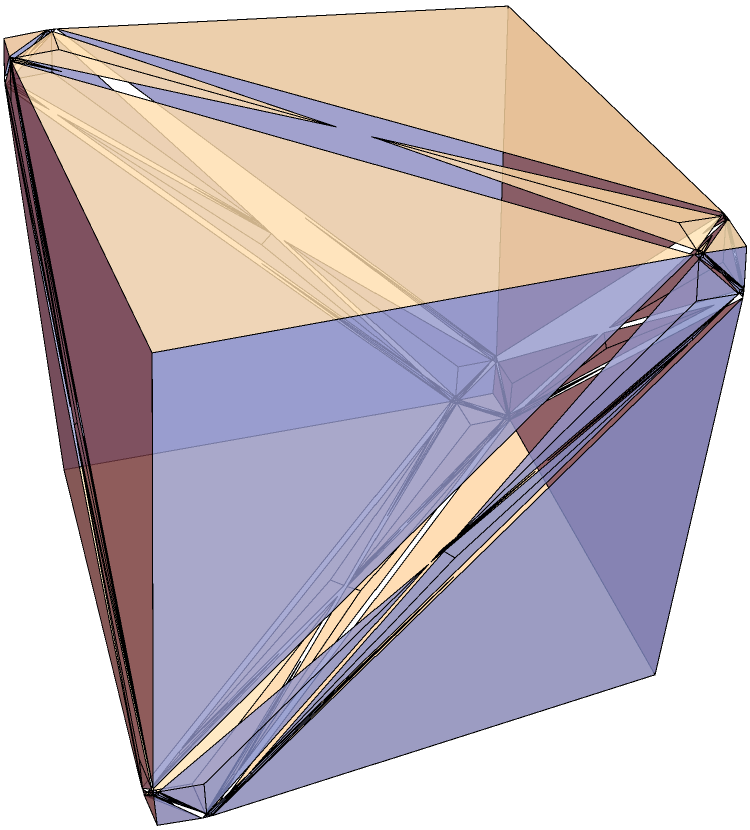}}

    \caption{The principal tetrahedra of $\Omega_{max}$ as the representation $\rho$ deforms away from the holonomy representation of the complete hyperbolic structure.}
    \label{fig:omegamax}
\end{figure}

\begin{proof}[Proof of Theorem \ref{thm:deform}]
 By assumption, the complete hyperbolic structure on $M$ is infinitesimally rigid rel $\partial M$.  By Theorem \ref{thm:diagonalizable} there is a path $\rho_t$ of representations through $\rho_0 = \rhyp$ such that $\rho_t(\Delta_i)$ is diagonalizable over the reals for all $t\neq 0$ and $i \in \{1, \ldots, k\}$. The holonomy principle of Cooper--Long--Tillmann (Theorem~\ref{thm:koszul}) or Choi~\cite[Corollary 1.1]{Choi2} guarantees that there is some $\varepsilon > 0$, such that for $t \in (0, \varepsilon)$, the representation $\rho_t$ is the holonomy representation of a properly convex projective structure, nearby the complete hyperbolic structure. Further, the ends of this structure have generalized cusps, although we will not use this fact here.

 By construction, $\res_i(\rho_t)$ belongs to $\slice_i$. Hence for $t > 0$, $\rho_t(\Delta_i)$ is a lattice in a two-dimensional diagonalizable subgroup conjugate to $C_{a, b}$, as defined in Section~\ref{subsec:slice}, where $(a,b) = (a^{(i)}_t, b^{(i)}_t) \neq (0,0)$ depend on~$i$ and continuously on~$t$. Since all elements of $C_{a,b}$ have a common fixed point with eigenvalue one, the middle eigenvalue condition is satisfied for $\rho_t(\Delta_i)$ for all $t > 0$ and $i \in \{1, \ldots, k\}$.
 For each $t$, the minimal convex sub-manifold $M_{min,t} \subset M$ is a properly convex projective manifold with totally geodesic boundary. Further, the inclusion $M_{min,t} \hookrightarrow M$ is a homotopy equivalence which is isotopic to a homeomorphism. Such an isotopy may be constructed explicitly in the universal cover by flowing the points lying in a principal tetrahedron of $\widetilde M$ (realized as the convex domain $\Omega_t$) radially toward the principal triangle bounding that tetrahedron.
\end{proof}

 \begin{remark} In the context of the proof of Theorem~\ref{thm:deform}, we note that for each end, the continuous deformation of the principal totally geodesic triangles, which open up from the parabolic fixed point of the hyperbolic structure, may be observed explicitly in terms of the parameters $(a,b)$ of the last paragraph of the proof. Indeed, the eigenvectors of $C_{a,b}$ were described explicitly in Section~\ref{subsec:slice}.
 \end{remark}

We close this subsection with one more lemma which will be needed in the next section. The region $U^{(i)}= \mathcal T^{(i)}_+ \cup T^{(i)} \cup \mathcal T^{(i)}_-$ is a triangular prism. Its boundary $\partial U^{(i)}$ is the union of three totally geodesic bigons.

\begin{lemma}\label{lem:proper-contain}
The intersection $\partial \Omega_{max} \cap \partial U^{(i)}$ is equal to $\partial \mathcal T^{(i)}_- \setminus T^{(i)}$. In particular, $\partial \Omega_{min} \cap \partial \mathcal T^{(i)}_+$ contains $\overline{T^{(i)}}$ but does not contain any point of any other (open) face of $\mathcal T^{(i)}_+$ of positive codimension that is adjacent to $p^{(i)}_4$.
\end{lemma}

\begin{proof}
Suppose $p \in \partial \Omega_{max}$ lies on a face $\mathscr F$ of $\mathcal T^{(i)}_+$ of positive codimension adjacent to $p^{(i)}_4$. Although $p^{(i)}_4$ is never an attractor for the action of any $\rho(\gamma)$ on $\mathcal T^{(i)}_+$, it is, by dimension count, an attractor for the action of some $\rho(\gamma)$ on $\mathscr F$. Then by convexity $\Omega_{max}$ contains $\mathcal T^{(i)}_+$ and therefore all of $U^{(i)}$. However this is impossible since $U^{(i)}$ is not properly convex and $\Omega_{max}$ is properly convex.
\end{proof}

%
%
%

 \section{Gluing convex projective structures}\label{gluing}

 In this section we prove Theorem~\ref{thm:convex-gluing}, which states that convex projective structures may be glued along principal totally geodesic torus boundary components whenever the holonomy matching condition \eqref{eqn:matching} is satisfied. This will complete the proof of Theorem~\ref{thm:main}. Theorem~\ref{thm:convex-gluing} is proved by induction using the following two Lemmas:
\begin{lemma}\label{lem:glue-different}
Let $M_1 = \Gamma_1 \backslash \Omega_1$, and $M_2 = \Gamma_2 \backslash \Omega_2$ be two properly convex projective three-manifolds.
Let $f \co \partial_1 \to \partial_2$ be a homeomorphism between principal totally geodesic torus boundary components $\partial_1$ and $\partial_2$ of $M_1$ and $M_2$, respectively, satisfying the holonomy matching condition: There exists $g \in \tilde G$ such that
\begin{align}\label{eqn:g-equivariance}
f_* \gamma &= g \gamma g^{-1}
\end{align}
for all $\gamma \in \Delta_1$, where $f_* \co \Delta_1 \to \Delta_2$ is the group homomorphism induced by $f$ on the fundamental groups $\Delta_1$ and $\Delta_2$ of $\partial_1$ and $\partial_2$ respectively. Then the topological manifold $M_1 \cup_f M_2$ admits a properly convex structure which restricts to the given properly convex projective structures on $M_1$ and $M_2$.
\end{lemma}

\begin{lemma}\label{lem:glue-same}
Let $M = \Gamma \backslash \Omega$ be a properly convex projective three-manifold and let $f \co \partial_1 \to \partial_2$ be a homeomorphism between principal totally geodesic boundary components $\partial_1, \partial_2$ of $M$ that satisfies the holonomy matching condition: There exists $g \in \tilde G$ such that
\begin{align}\label{eqn:g-equivariance-2}
f_* \gamma &= g \gamma g^{-1}
\end{align}
for all $\gamma \in \Delta_1$, where $f_* \co \Delta_1 \to \Delta_2$ is the group homomorphism induced by $f$ on the fundamental groups $\Delta_1$ and $\Delta_2$ of $\partial_1$ and $\partial_2$ respectively. Then the topological manifold $M_f$ admits a properly convex structure which restricts to the given properly convex projective structure on $M$.
\end{lemma}
We now give a detailed proof of Lemma~\ref{lem:glue-different}. The proof of Lemma~\ref{lem:glue-same} is nearly identical.

For convenience, throughout this subsection, we will identify the universal cover of $M_i$ with $\Omega_i$ and the fundamental group $\pi_1 M_i$ with $\Gamma_i \subset \tilde G$, for both $i \in \{1,2\}$. For $i \in \{1,2\}$, let $\partial_i = \Delta_i \backslash T^{(i)}$ be the given principal totally geodesic boundary torus of $M_i$, where $T^{(i)}$ is a totally geodesic triangle contained in~$\partial \Omega_i$. 
For each $i$, let $\mathcal T^{(i)}_+$ and $\mathcal T^{(i)}_-$ be defined as in the previous section, so that $\Omega_i \subset \mathcal T^{(i)}_+$ and $\mathcal T^{(i)}_-$ is a principal tetrahedron for $\Omega_i$. By translating $\Omega_1$ by $g$ and possibly a reflection in the centralizer $Z(\Delta_2)$, we may assume henceforth that $g = 1$ and that $\Omega_1$ and $\Omega_2$ are positioned so that $T^{(1)} = T^{(2)} =: T$, $\mathcal T^{(2)}_- = \mathcal T^{(1)}_+ =: \mathcal T^{(1)}$ and $\mathcal T^{(1)}_- = \mathcal T^{(2)}_+ =: \mathcal T^{(2)}$. Then, by~\eqref{eqn:g-equivariance}, the identity projective transformation $g = 1$ descends to a \emph{projective} gluing map in the isotopy class of $f$ that glues $M_1$ to $M_2$ along their totally geodesic boundary components. Any neighborhood in $ \Omega_1 \cup T \cup \Omega_2$ of a point $p \in T$ serves as a chart defining a projective structure on a neighborhood of $\partial_1 = \partial_2$ in the glued up manifold $N = M_1 \cup_f M_2$. The
charts on $M_1$ and $M_2$ are compatible with these new charts, hence $N$ is endowed with a projective structure in which $\partial_1 = \partial_2$ is an embedded totally geodesic torus. To prove Theorem~\ref{thm:convex-gluing}, we now show that this projective structure is properly convex.

The fundamental group of $N$ naturally identifies with the free product $\Gamma = \Gamma_1 *_{f_*} \Gamma_2$ amalgamated over the identification $f_* \co \Delta_1 \to \Delta_2$. Let $\Delta$ denote the inclusions of $\Delta_1$ and $\Delta_2$, which are identified, in $\Gamma$. We will denote the product of two elements $\alpha, \beta$ in the abstract group $\Gamma$ by the notation $\alpha \star \beta$ in order to avoid confusion with matrix multiplication in $\tilde G$. The universal cover $\widetilde N$ is described combinatorially as:
\begin{align*}
\widetilde N &= \Gamma \times \Omega_{1,\partial} / \sim_1 \cup \ \ \Gamma \times \Omega_{2,\partial} / \sim_2
\end{align*}
where $\Omega_{1,\partial}$ denotes the union of $\Omega_1$ with the $\Gamma_1$ orbit of $T$, and $\sim_1$ is the equivalence relation generated by $(\gamma, p) \sim_1 (\gamma \gamma_1^{-1}, \gamma_1p)$ for all $\gamma_1 \in \Gamma_1$, and similarly for $\Omega_{2,\partial}$ and $\sim_2$. We refer to each $\{\gamma\}\times \Omega_{i,\partial}$ as a \emph{tile}. If $p \in T = \Omega_{1,\partial} \cap \Omega_{2,\partial}$, then we consider the points $(\gamma, p) \in \Gamma \times \Omega_{1,\partial}$ and $(\gamma, p) \in \Gamma \times \Omega_{2,\partial}$ to be identified.
The developing map $dev$ for the natural projective structure on $N$ is defined by the formula $dev([\gamma, p]) = \rho(\gamma) p$, for any $\gamma \in \Gamma$ and $p \in \Omega_{1,\partial} \cup \Omega_{2,\partial}$, where $\rho \co \Gamma \to \tilde G$, the holonomy representation, is defined by the property that its restriction to $\Gamma_i$ is the inclusion map for each $i \in \{1,2\}$. In other words if $\gamma_1, \ldots, \gamma_m$ are elements of $\Gamma_1 \cup \Gamma_2$, then $\rho(\gamma_1 \star \cdots \star \gamma_m) = \gamma_1 \cdots \gamma_m$. 
We will also consider an augmented version of $\widetilde N$, defined by
\begin{align*}
\widetilde N_{aug} &= \Gamma \times \overline \Omega_1 / \sim_1 \cup \ \ \Gamma \times \overline \Omega_2 / \sim_2
\end{align*}
which includes the full boundaries of the convex tiles. The developing map $dev$ extends naturally to an augmented developing map $dev \co \widetilde N_{aug} \to \RP^3$. It is a local embedding, even at points of the tile boundaries $[\gamma, \partial \Omega_i]$.

We show that the developing map is injective with image contained in an affine chart via a ping-pong lemma.

\begin{lemma}\label{lem:ping}
Let $\gamma_1 \in \Gamma_1 \setminus \Delta$ and $\gamma_2 \in \Gamma_2 \setminus \Delta$. Then
\begin{enumerate}
\item $\gamma_1 \overline{\mathcal T^{(2)}} \subset \mathcal T^{(1)} \setminus \Omega_1$. 
\item $\gamma_2 \overline{\mathcal T^{(1)}} \subset \mathcal T^{(2)} \setminus \Omega_2$. 
\end{enumerate}
\end{lemma}
\begin{proof}
We prove only the first statement as the second follows by symmetry.
Note that $\mathcal T^{(2)} = \mathcal T^{(1)}_-$ is contained in $(\Omega_1)_{max}$ (Definition~\ref{def:max}). Hence $\gamma_1 \mathcal T^{(2)} \subset (\Omega_1)_{max}$ as well. $(\Omega_1)_{max}$ is contained in the triangular prism $U = \mathcal T^{(1)} \cup T \cup \mathcal T^{(2)}$, which is convex (but not properly convex) and bounded by three bigons. Since $\gamma_1 \notin \Delta$, $\gamma_1 \mathcal T^{(2)}$ is not equal to $\mathcal T^{(2)}$  and therefore the two tetrahedra do not intersect. Hence $\gamma_1 \mathcal T^{(2)}$ must lie in $\mathcal T^{(1)}$. By Lemma~\ref{lem:proper-contain}, we also have that $\overline{\gamma_1 \mathcal T^{(2)}} = \gamma_1 \overline{\mathcal T^{(2)}}$ lies in $\mathcal T^{(1)}$, that is no point of $\gamma_1 \overline{\mathcal T^{(2)}}$ intersects $\partial \mathcal T^{(1)}$.  Of course $\gamma_1$ preserves $\Omega_1$, so $\gamma_1 \overline{ \mathcal T^{(2)}}$ does not intersect $\Omega_1$ because $\overline{\mathcal T^{(2)}}$ does not intersect~$\Omega_1$.
\end{proof}

\begin{lemma}\label{lem:pong}
The augmented developing map $dev \co \widetilde N_{aug} \to \RP^3$ is injective and its image is contained in $\overline{(\Omega_1)}_{max}$. 
\end{lemma}

\begin{proof}
We already know that $dev$ is an embedding when restricted to the union of any two adjacent closed tiles. To show that $dev$ is a global embedding, it suffices to show that
$dev([\{\gamma\} \times \overline \Omega_1]) = \rho(\gamma)\overline \Omega_1$ does not intersect $\overline \Omega_{2}$ nor $\overline \Omega_1$ as long as $\gamma \notin \Gamma_1 \cup \Gamma_2$.

Assume that $\gamma \notin \Gamma_1 \cup \Gamma_2$. Then $\gamma$ may be expressed as an alternating product $\gamma = \gamma_1 \star \cdots \star \gamma_m$ of $m \geq 2$ elements $\gamma_i \in \Gamma_1 \cup \Gamma_2 \setminus \Delta$ such that for $i =1, \ldots, m-1$, $\gamma_i \in \Gamma_1$ if and only if $\gamma_{i+1} \in \Gamma_2$. There are two possibilities to consider. First, assume $\gamma_m$ lies in $\Gamma_2 \setminus \Delta$. Then by Lemma~\ref{lem:ping},
\begin{align*}
\rho(\gamma)\overline \Omega_1 = \gamma_1 \cdots \gamma_{m-1} \gamma_m \overline \Omega_1 &\subset \gamma_1 \cdots \gamma_{m-1} (\mathcal T^{(2)} \setminus \Omega_2)\\ & \subset \gamma_1 \cdots \gamma_{m-2} (\mathcal T^{(1)} \setminus \Omega_1) \\ &\vdots \\
&\subset (\mathcal T^{(j)} \setminus \Omega_j)
\end{align*}
where $j = 1$ if $\gamma_1 \in \Gamma_1$ (equivalently if $m$ is even) or $j=2$ if $\gamma_1 \in \Gamma_2$. Hence $\rho(\gamma) \overline \Omega_1$ does not intersect either $\overline \Omega_1$ or $\overline \Omega_2$. The other possibility is that $\gamma_m \in \Gamma_1 \setminus \Delta$.  In this case,  $\rho(\gamma)\overline \Omega_1 = \gamma_1 \cdots \gamma_{m-1} \overline \Omega_1$, with $\gamma_{m-1} \in \Gamma_2 \setminus \Delta$ and we proceed as in the previous case replacing $\gamma$ with $\gamma_1 \star \cdots \star \gamma_{m-1}$.
This completes the proof that $dev$ is injective on $\widetilde N_{aug}$. Indeed $dev$ is a closed map, so $dev$ is an embedding.

We may also see from the above ping-pong argument that the image of $dev$ is contained in $\overline{(\Omega_1)}_{max}$. For any $\gamma \in \Gamma_1$, we have that $\rho(\gamma) \overline \Omega_1 = \overline \Omega_1 \subset  \overline{(\Omega_1)}_{max}$ and $\rho(\gamma) \overline \Omega_2 \subset \rho(\gamma) \overline{\mathcal T^{(2)}} \subset \overline{(\Omega_1)}_{max}$.
If $\gamma \in \Gamma_2 \setminus \Delta$, then $\rho(\gamma) \overline\Omega_2 =\overline\Omega_2 \subset \overline{(\Omega_1)}_{max}$ and $\rho(\gamma) \overline \Omega_1 \subset \rho(\gamma) \overline{\mathcal T^{(1)}} \subset \mathcal T^{(2)} \subset  (\Omega_1)_{max}$.
Finally, let $\gamma \notin \Gamma_1 \cup \Gamma_2$. Then as above $\gamma = \gamma_1 \star \gamma_2 \star \cdots \star \gamma_m$ and $\rho(\gamma) \overline \Omega_1 \subset \mathcal T^{(2)} \subset (\Omega_1)_{max}$ if $\gamma_1 \in \Gamma_{2}$. If $\gamma_1 \in \Gamma_{1}$, then $\rho(\gamma_2 \star \cdots \star \gamma_m) \overline \Omega_1$ is contained in $\mathcal T^{(2)}$, so $\rho(\gamma_1) \rho(\gamma_2 \star \cdots \star \gamma_m)\overline \Omega_1 = \rho(\gamma) \overline \Omega_1$ is contained in $(\Omega_1)_{max}$. It follows similarly that if $\gamma\notin \Gamma_1\cup\Gamma_2$ then $\rho(\gamma)\overline\Omega_2 \subset (\Omega_1)_{max}$.
\end{proof}

Finally, we prove:

\begin{lemma}\label{lem:prop-convex}
The projective manifold $N$ is properly convex.
\end{lemma}

To prove Lemma~\ref{lem:prop-convex}, we will need the following basic result about convex sets in Euclidean space. It is similar to a well-known theorem of Nakajima~\cite{1928227} and Tietze~\cite{Tietze1928}. We include a proof for convenience.
\begin{lemma}\label{lem:classical}
Suppose $A, B \subset \RR^d$ are closed convex subsets with non-empty interior and non-empty intersection. 
Suppose that $A \cup B$ satisfies the following local convexity condition along $C = \partial(A \cup B) \cap (A \cap B)$: 
At each point $z \in C$ there is a local support plane, i.e. a hyperplane containing $z$ and bounding a closed half-space that contains a neighborhood of $z$ in $A \cup B$. Suppose further there is a point of $A \cap B$ in the interior of $A \cup B$. Then $A \cup B$ is convex.
\end{lemma}
\begin{proof} 
We consider first the interior $\mathrm{Int}(A \cup B)$. Fix a point $x \in \mathrm{Int}(A)$ and let $S$ be the set of all $y \in \mathrm{Int}(A \cup B)$ such that $[x,y] \subset \mathrm{Int}(A \cup B)$. $S$ is clearly an open subset of  $\mathrm{Int}(A \cup B)$. We show $S$ is also closed in $\mathrm{Int}(A \cup B)$. Suppose $y_n \in S$ converges to $y \in \mathrm{Int}(A \cup B)$ but $y \notin S$. The open interval $(x,y)$ is contained in $A \cup B$ and intersects $\partial (A \cup B)$ in at least one point~$z$. Let $H$ be a hyperplane supporting $A \cup B$ locally near $z$, guaranteed to exist by the local convexity assumption if $z \in C$ or by the convexity of $A$ (resp. of $B$) if $z \in \partial A \setminus C$ (resp. if $z \in \partial B \setminus C$). Since $(x,y)$ does not cross $H$ transversely at $z$, we must have $[x,y] \subset H$. However, this is a contradiction: if $z \in A$, the local support plane $H$ does not intersect the convex set $\mathrm{Int}(A)$ so it can not contain $x$, and if $z \in B$, then $H$ does not intersect $\mathrm{Int}(B)$ so it does not contain $y$. Hence $S$ is closed in $\mathrm{Int}(A \cup B)$.
We note that $\mathrm{Int}(A \cup B)$ is connected, since by assumption a point of $A \cap B$ is contained in $\mathrm{Int}(A \cup B)$. Hence $S = \mathrm{Int}(A \cup B)$. Hence for all $x \in \mathrm{Int}(A)$ and $y \in \mathrm{Int}(A \cup B)$, $[x,y] \subset \mathrm{Int}(A \cup B)$ and similarly for $x \in \mathrm{Int}(B)$ by symmetry. It now follows by taking limits that $A \cup B$ is convex.
\end{proof}

\begin{proof}[Proof of Lemma~\ref{lem:prop-convex}]
By Lemma~\ref{lem:pong}, the augmented developing map is an embedding into an affine chart $\mathbb A$ of $\RP^3$.
The image of each tile of $\widetilde N_{aug}$ is convex.
The union of two adjacent tiles is locally convex at the boundary of the interface between them. To check this, it suffices to examine $\overline\Omega_1 \cup \overline\Omega_2$. We may assume that $\overline{U^{(1)}} = \overline{\mathcal T^{(1)}_+} \cup \overline{\mathcal T^{(1)}_-}$ intersects the affine chart $\mathbb A$ in an infinite triangular prism (with vertex at infinity). Since $\overline\Omega_1 \cup \overline\Omega_2$ is contained in  $\overline{U^{(1)}}$ and since $\partial (\overline\Omega_1 \cup \overline\Omega_2) \cap (\overline\Omega_1 \cap \overline\Omega_2) = \partial \overline T $ is contained in $\partial \overline{U^{(1)}}$, it follows that $\overline\Omega_1 \cup \overline\Omega_2$ is locally convex at the boundary of the interface $\overline T$  between the two tiles: one of the three planes bounding $\overline{U^{(1)}}$ supports $\overline\Omega_1 \cup \overline\Omega_2$ at each point of $\partial T$ .
The union $\overline\Omega_1 \cup \overline\Omega_2$ is convex by Lemma~\ref{lem:classical}.
Next, since no three of the closed tiles of $\widetilde N_{aug}$ meet non-trivially, it follows by induction and  Lemma~\ref{lem:classical} that the image under the augmented developing map of any finite connected union of tiles of $\widetilde N_{aug}$ is convex. 
 The image of $\widetilde N_{aug}$ is an increasing union of such convex sets and hence is convex. 
 The image of $\widetilde N$, which is the interior of the image of $\widetilde N_{aug}$, is therefore also convex, and indeed properly convex because it is contained in $\overline{(\Omega_1)}_{max}$.
\end{proof}

This completes the proof Lemma~\ref{lem:glue-different}.  The proof of Lemma~\ref{lem:glue-same} is nearly the same and so we include only the following sketch which highlights the required modifications.
\begin{enumerate}
\item Possibly after modifying $g$ by a reflection in the centralizer $Z(\Delta_2)$, we may assume that $\Omega$ and $g \Omega$ are positioned so that $g \mathcal T^{(1)}_\mp = \mathcal T^{(2)}_\pm$, where for both $i \in \{1,2\}$, $\mathcal T^{(i)}_+ \supset \Omega$ and $\mathcal T^{(i)}_-$ is the principal tetrahedron for $\Omega$ preserved by $\Delta_i$ as above.
 \item The fundamental group of $M_f$ is the HNN extension $\Gamma\ast_{f_\ast}$. The universal cover of $M_f$ is combinatorially a union of tiles $\widetilde M_f = \Gamma\ast_{f_\ast} \! \times  \Omega_{\partial} \,/\sim$, where $\Omega_{\partial}$ is the union of $\Omega$ and the principal triangles covering the boundary component $\partial_1$ and $\partial_2$ and the equivalence relation is generated by $(\gamma,p)\sim(\gamma\gamma_1^{-1},\gamma_1 p)$, for all $\gamma_1\in \Gamma$.
 \item The developing map $dev$ for the natural projective structure on $M_f$ is defined by the formula $dev([\gamma, p]) = \rho(\gamma) p$, for any $\gamma \in \Gamma\ast_{f_*}$ and $p \in \Omega_{\partial}$, where $\rho \co \Gamma\ast_{f_*} \to \tilde G$, the holonomy representation, is defined by the property that its restriction to $\Gamma$ is the inclusion map and $\rho$ applied the stable letter of $\Gamma\ast_{f_*}$ is the gluing transformation $g$. The developing map $dev$ extends to the augmented domain $\widetilde M_{f, aug} = \Gamma\ast_{f_*}\! \times \overline{\Omega} \, / \sim$.
 \item  The following analogue of Lemma \ref{lem:ping} holds: 
 \begin{itemize}
  \item $g\mathcal{T}^{(1)}_{\mp} =\mathcal{T}^{(2)}_\pm$  (already arranged above)
  \item For any $i\in \{1,2\}$ and $\gamma\in \Gamma\setminus \Delta_i$, $\gamma\overline{\mathcal{T}^{(i)}_-}\subset \mathcal{T}^{(i)}_+\backslash \Omega$ (follows immediately because $\Omega_{max}$ contains $\gamma \mathcal T^{(i)}_-$ and is contained in the union $U^{(i)} = \mathcal T^{(i)}_+ \cup T^{(i)} \cup \mathcal T^{(i)}_-$.)
 \end{itemize}
\item Plugging (4) into the argument from Lemma~\ref{lem:pong} shows that $dev$ is an embedding with image contained in $\overline{\Omega}_{max}$. Finally, as in Lemma~\ref{lem:prop-convex}, the image of $\widetilde M_{f,aug}$ under $dev$ is locally convex showing that the projective structure on $M_f$ is properly convex.
\end{enumerate}
\begin{proof}[Proof of Theorem~\ref{thm:convex-gluing}]
The theorem follows by induction by applying Lemmas~\ref{lem:glue-different} and~\ref{lem:glue-same} in succession to glue together the given collection of convex projective manifolds along various pairs of boundary components using the given gluing maps in any order.
\end{proof}
Suppose we are given a collection of cusped hyperbolic manifolds, along with gluing homeomorphisms between some of their torus boundary components. In general, it is not known whether projective structures with totally geodesic boundary satisfying the appropriate holonomy matching hypotheses~\eqref{eqn:matching} of Theorem~\ref{thm:convex-gluing} can be found. To find such glueable structures would seem to require a description of the global deformation space of properly convex projective structures with principal totally geodesic boundary on a given manifold. As of the writing of this article, there is no cusped hyperbolic manifold for which this global deformation space has been computed, abstractly nor computationally (although there are certain Coxeter orbifolds for which the deformation space has been computed, see Section~\ref{reflectionorbs}).
 Nonetheless, the matching condition is automatically satisfied in the case of doubling a convex projective manifold with principal totally geodesic boundary. Indeed:
\begin{proof}[Proof of Theorem~\ref{thm:main}]
 Since $M$ is assumed to be infinitesimally projectively rigid rel  boundary, Theorem~\ref{thm:deform} tells us that we can deform the complete hyperbolic structure on $M$ to a convex projective structure where each boundary component is a principal totally geodesic torus. The gluing condition \eqref{eqn:matching} is trivially satisfied for the identity gluing, and thus Theorem \ref{thm:convex-gluing} ensures that the double $2M$ admits a properly convex projective structure.
\end{proof}

 \begin{figure}[ht!]
    \centering
    \subfloat[]{\includegraphics[scale=.3]{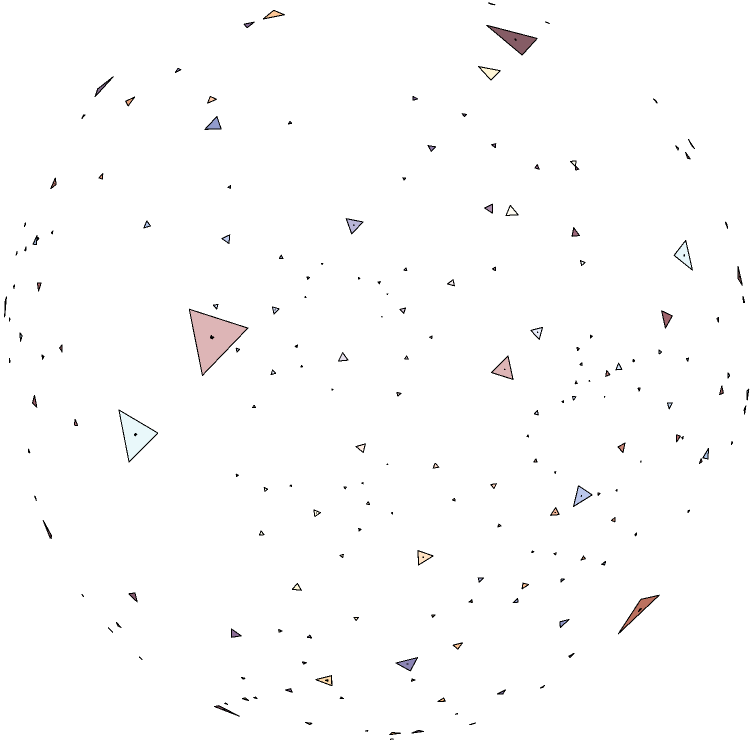}}
    \quad\;
    \subfloat[]{\includegraphics[scale=.3]{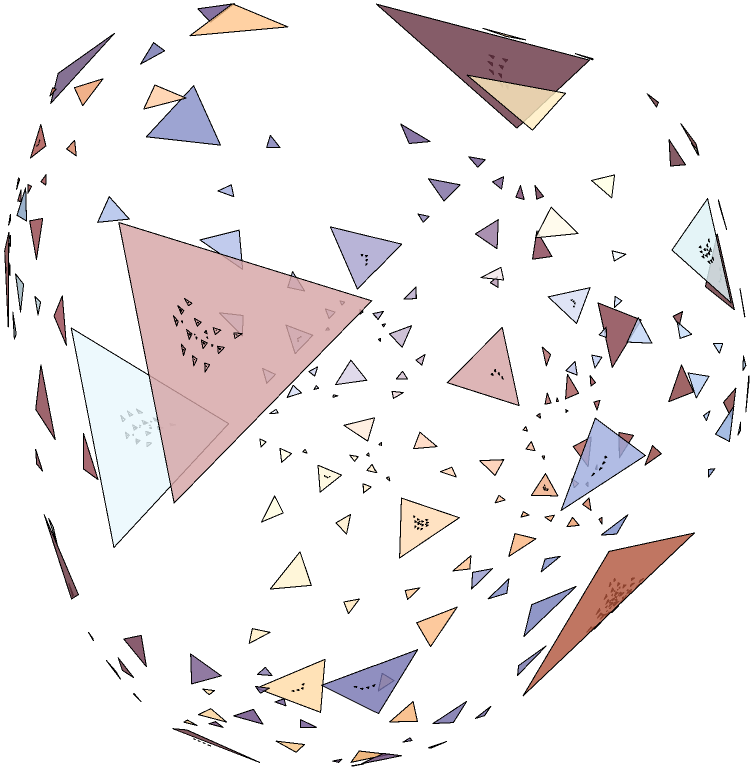}}

    \subfloat[]{\includegraphics[scale=.3]{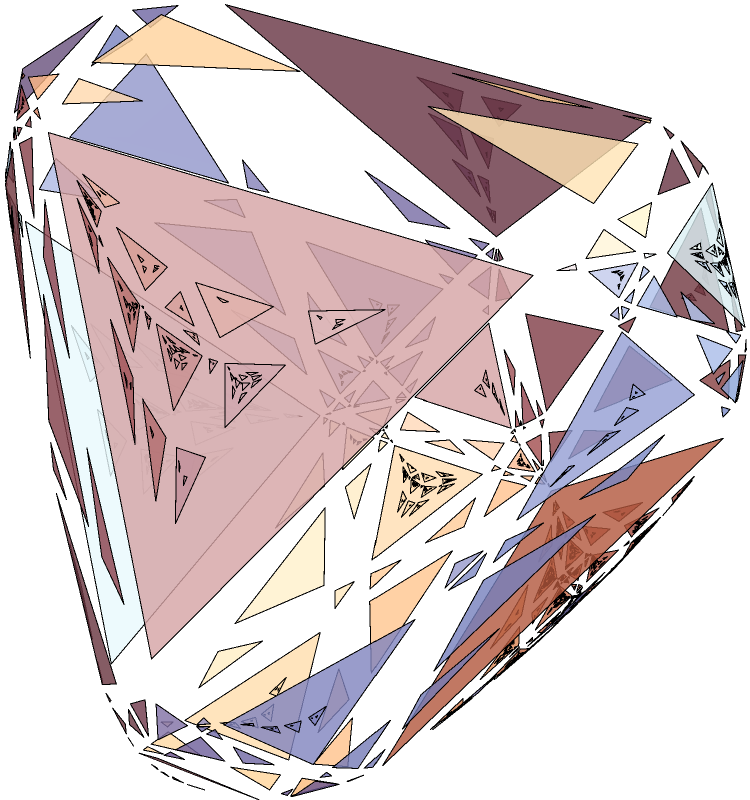}}
    \quad\;
    \subfloat[]{\includegraphics[scale=.3]{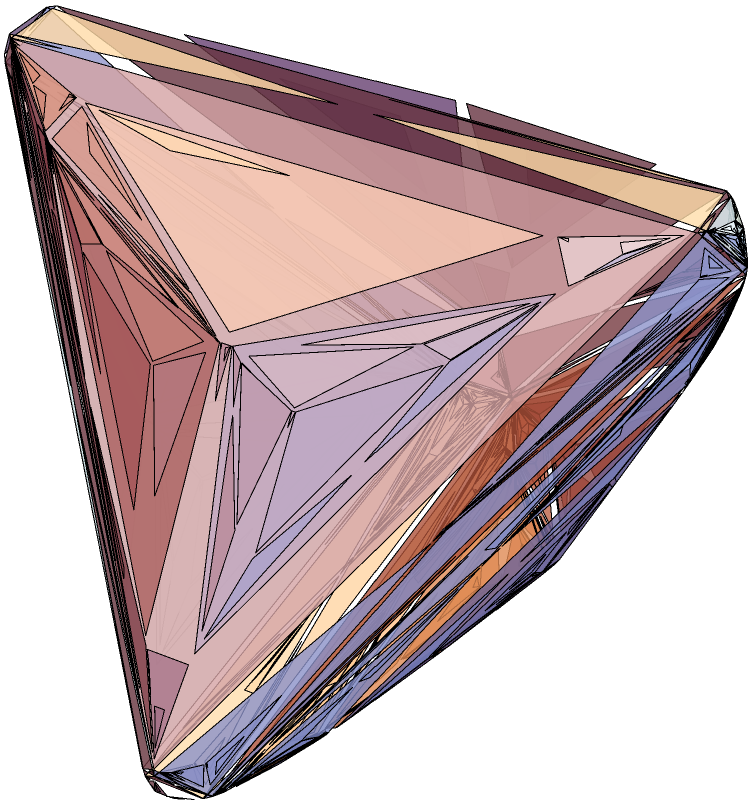}}
    \caption{A collection of properly embedded triangles in the convex domains defining convex projective structures on the double of the figure eight knot complement. Going from (A) to (D), the projective structure on each of the two pieces deforms away from the complete hyperbolic structure. }
    \label{fig:omegadouble}
\end{figure}

%
%
%
\section{Gluing covering manifolds of reflection $3$--orbifolds}\label{reflectionorbs}

We conclude the paper by giving examples of properly convex projective structures on non-hyperbolic manifolds which are not obtained by doubling. To do this, we apply the Convex Gluing Theorem~\ref{thm:convex-gluing} to a collection of highly symmetric properly convex projective structures with principal totally geodesic boundary. The symmetry of these structures, which are covers of certain convex projective reflection orbifolds, greatly restricts their boundary holonomy groups. This allows us to determine which combinations of these manifolds admit structures satisfying the matching hypothesis~\eqref{eqn:matching} of Theorem~\ref{thm:convex-gluing}.

\subsection{Euclidean hex tori}
We will be interested in three-dimensional hyperbolic reflection orbifolds with cusps isomorphic to the Euclidean $(3,3,3)$--triangle orbifold. We now recall the geometric features of this two-orbifold and describe the geometry of its torus covers.

Let $\Delta_{3,3,3}$ denote the subgroup of the isometries $\mathrm{Isom}\, \mathbb{E}^2$ of the Euclidean plane generated by reflections in the sides of an equilateral triangle $T$ in the Euclidean plane. We denote the resulting quotient orbifold  by $\mathcal S_{3,3,3} = \Delta_{3,3,3} \backslash \mathbb{E}^2$.

If $\Delta_H$ denotes the subgroup of all translations in $\Delta_{3,3,3}$, then $\Delta_H$ is the maximal torsion-free subgroup of $\Delta_{3,3,3}$ and all other torsion-free subgroups $\Delta$ of $\Delta_{3,3,3}$ are contained in $\Delta_H$. In other words, the torus $\mathcal S_H = \Delta_H \backslash \mathbb E^2$ is the minimal torus cover of $\mathcal S_{3,3,3}$, and every orientable surface $\mathcal S = \Delta \backslash \mathbb{E}^2$ covering $\mathcal S_{3,3,3}$ is a Euclidean torus which may be decomposed into regular hexagons each made up of six copies of the equilateral triangle $T$. We call such a torus a \emph{Euclidean regular hex torus}, or just \emph{hex torus} for short. The fundamental group $\Delta$ of a Euclidean regular hex torus $\mathcal S$ is called a \emph{hex torus group}.

Let $k$ be a natural number and let $\Delta'$ be a hex torus group. Then the group $\Delta = k \Delta'$ of all $k^{th}$ powers of elements of $\Delta'$ is a subgroup of $\Delta'$. If $\mathcal S'$ and $\mathcal S$ are the hex tori associated to $\Delta'$ and $\Delta$, then there is a natural homeomorphism, denoted by $k \co \mathcal S' \to \mathcal S$ which, in the universal cover $\mathbb E^2$, simply scales by $k$.

\begin{definition}
Let $\mathcal S_1, \mathcal S_2$ be hex tori corresponding to the subgroups $\Delta_1, \Delta_2 \subset \Delta_H$.
Then a homeomorphism $f \co \mathcal S_1 \to \mathcal S_2$ is said to be a \emph{lattice $\mathbb Q$--isometry} or a \emph{lattice $\frac{k_2}{k_1}$--isometry} if there exists a hex torus $\mathcal S'$ with associated group $\Delta'$ and natural numbers $k_1, k_2$ such that $\Delta_1 = k_1 \Delta'$, $\Delta_2 = k_2 \Delta'$, and so that the map $f' = k_2^{-1} f k_1 \co \mathcal S' \to \mathcal S'$ is induced by conjugation by an element of $\Delta_{3,3,3}$.
When $k_1 = k_2 =1$, we call $f$ a \emph{lattice isometry}.
\end{definition}

We now define an invariant, called the \emph{hex shape equivalence class} which may be used to determine when a lattice $\mathbb Q$--isometry exists between two hex tori.
  Let $v_1, v_2$ be a pair of vectors of minimal length such that the translations $t_{v_1}$ and $t_{v_2}$ generate $\Delta_H$. Assume further that the angle between $v_1$ and $v_2$ is equal to $2\pi / 3$. Then $v_1, v_2$ are unique up to the action of the dihedral group $D_6$ of order $12$, generated by the order six rotation $\begin{pmatrix} v_1 \\ v_2 \end{pmatrix} \mapsto \begin{pmatrix} 0 & -1\\ 1 & 1 \end{pmatrix} \begin{pmatrix} v_1 \\ v_2 \end{pmatrix}$ and the reflection $\begin{pmatrix} v_1 \\ v_2 \end{pmatrix} \mapsto \begin{pmatrix} 1 & 0\\ -1 & -1 \end{pmatrix} \begin{pmatrix} v_1 \\ v_2 \end{pmatrix}$.
Let $w_1 = p_{11} v_1 + p_{12} v_2$ and $w_2 = p_{21} v_1 + p_{22} v_2$ be an ordered pair of vectors such that the translations $t_{w_1}$ and $t_{w_2}$ generate a hex torus group~$\Delta$ with corresponding hex torus denoted $\mathcal S$.
Then we can encode the shape of $\Delta$ with the $2 \times 2$ integer matrix $A_\Delta$ appearing in the equation $$\begin{pmatrix} w_1\\ w_2\end{pmatrix} = \begin{pmatrix} p_{11} & p_{12}\\ p_{21} & p_{22} \end{pmatrix} \begin{pmatrix} v_1\\ v_2\end{pmatrix}.$$
The matrix $A_{\Delta}$ is only well-defined up to multiplication on the left by $\SL^{\pm}_2 \mathbb Z$ and multiplication on the right by the matrices of the dihedral group $D_6$ above. We call the equivalence class of matrices $A_{\Delta}$ the \emph{hex cusp shape} of $\mathcal S$ (or of $\Delta$). The following is elementary.

\begin{lemma}\label{lem:hex-shape}
Let $\Delta_1, \Delta_2 \subset \Delta_{3,3,3}$ be the fundamental groups of two regular Euclidean hex tori $\mathcal S_1, \mathcal S_2$. Let $A_{\Delta_1}$ and $A_{\Delta_2}$ be matrices representing the respective hex cusp shapes of $\mathcal S_1$ and $\mathcal S_2$. Then there exists a lattice $\frac{k_2}{k_1}$--isometry $f \co \mathcal S_1 \to \mathcal S_2$ if and only if $k_2 [A_{\Delta_1}] = k_1 [A_{\Delta_2}]$, \ie if there exists $B \in \SL^{\pm}_2 \mathbb Z$ and $C \in D_6$ such that $k_2 A_{\Delta_1} = k_1 B A_{\Delta_2} C$.
\end{lemma}
For example, if $\Delta_1 = \langle 5 v_1 - 5 v_2, 5 v_1 + 5 v_2 \rangle$ and $\Delta_2 = \langle  8 v_1, 4 v_2 \rangle$, then representatives of the hex cusp shape of the corresponding hex tori $\mathcal S_1$ and $\mathcal S_2$ are given by $A_{\Delta_1} = \begin{pmatrix} 5 & -5\\ 5 & 5 \end{pmatrix}$ and $A_{\Delta_2} = \begin{pmatrix} 8 & 0\\ 0 & 4 \end{pmatrix}$. There exists a lattice $\frac{4}{5}$--isometry $\mathcal S_1 \to \mathcal S_2$ because $$4 A_{\Delta_1} = 5 \begin{pmatrix} 1& 1\\ 0 & 1 \end{pmatrix}A_{\Delta_2} \begin{pmatrix} 0 & -1\\ 1 & 1 \end{pmatrix}.$$

\begin{remark}\label{rem:finitely-many}
Given a hex torus $\mathcal S$, the \emph{area} $a$ of $\mathcal S$, defined as the index of the covering $\mathcal S \to \mathcal S_H$, is simply the absolute value of the determinant of any matrix in the hex cusp shape. It follows from the Lemma that if $\mathcal S_1$ and $\mathcal S_2$ are lattice $\frac{k_2}{k_1}$--isometric, then their respective areas $a_1$ and $a_2$ satisfy $k_2^2 a_1 = k_1^2 a_2$. Therefore the ratio $a_2/a_1$ is a square, and the conformal factor $\frac{k_2}{k_1}$ of the $\mathbb Q$--isometry is given by $\sqrt{a_2/a_1}$.
Further, there are only finitely many equivalence classes of hex tori whose area is at most $a$. Hence, it is quite easy in practice to determine whether a lattice $\mathbb Q$--isomorphism $\mathcal S_1 \to \mathcal S_2$ exists as long as the areas of $\mathcal S_1$ and $\mathcal S_2$ are small. If one exists, then at most $12 = |D_6|$ exist.
\end{remark}

\subsection{Projective hex tori}
The deformation space $\mathfrak{C}(\mathcal S_{3,3,3})$ of marked convex real projective structures on the triangle orbifold $\mathcal S_{3,3,3}$ is homeomorphic to $\mathbb{R}$ and we now briefly describe this correspondence (see Goldman~\cite{Gol} and Vinberg~\cite{Vin} for more details). Let $r_1, r_2, r_3$ denote the reflections in the three sides of the triangle $T$, generating $\Delta_{3,3,3}$.
We identify the triangle $T$ with the positive octant in the projective $2$--sphere $\SS^2 = (\RR^3\setminus\{0\})/\RR^+$, which is the interior of the convex hull of three directions $e_1 = (1,0,0)$, $e_2 = (0,1,0)$, $e_3 = (0,0,1)$ of $\SS^2$. For each $\tau\in \RR$, let $s=e^{\tau/3}$, and define the representation $\zeta_\tau \co \Delta_{3,3,3} \to \SL^\pm_3 \RR$ by:
$$ \zeta_\tau(r_1) =\begin{pmatrix}
-1 & 0 & 0 \\
s & 1 & 0 \\
\frac{1}{s} & 0 & 1\\
\end{pmatrix},\quad
\zeta_\tau(r_2) =\begin{pmatrix}
1 & \frac{1}{s} & 0 \\
0 & -1 & 0 \\
0 & s & 1\\
\end{pmatrix},\quad
\zeta_\tau(r_3) =\begin{pmatrix}
1 & 0 & s \\
0 & 1 & \frac{1}{s} \\
0 & 0 & -1\\
\end{pmatrix}.
$$
For each $\tau$, these three elements are projective reflections in the sides of $T$ defining a representation $\zeta_\tau$ which is discrete and injective. The union of tiles, $\Omega_{\tau}= \cup_{\gamma \in \Delta_{3,3,3}} \zeta_{\tau}(\gamma) \overline{T}$, is a convex open domain in $\SS^2$. When $\tau = 0$, the convex domain $\Omega_{0}$ is the affine chart defined by $x_1 + x_2 + x_3 > 0$, where $(x_1, x_2, x_3)$ are coordinates with respect to the standard basis, and the representation $\zeta_0$ preserves a Euclidean metric on this affine chart; we identify $\Omega_0$ with the Euclidean plane $\mathbb E^2$ and think of $\zeta_0$ as the inclusion into $\operatorname{Isom}\mathbb E^2 \subset \SL^\pm_3 \RR$.
 For $\tau \neq 0$, $\Omega_{\tau}$ is the interior of the convex hull of the directions of the three vectors
$$l_1(\tau)=\epsilon_\tau(-1,s,0), \quad l_2(\tau)=\epsilon_\tau(0,-1,s) \quad \textrm{and} \quad l_3(\tau)=\epsilon_\tau(s,0,-1),$$
where $\epsilon_\tau=\tau/|\tau|$; see Figure \ref{hexTorus}. Any $\zeta_\tau$--equivariant homeomorphism of the universal cover $\widetilde S_{3,3,3} = \mathbb E^2$ with $\Omega_\tau$ is a developing map for the unique convex projective structure on $\mathcal S_{3,3,3}$ associated to the representation $\zeta_\tau$.
Of course, for any regular hex torus $\mathcal S$ covering $\mathcal S_{3,3,3}$ and for any $\tau \in \RR$, we obtain a convex projective structures on $\mathcal S$ by pullback. If $\Delta \subset \Delta_{3,3,3}$ denotes the fundamental group of $\mathcal S$, then the convex projective torus $\zeta_\tau(\Delta) \backslash \Omega_\tau$ is called a \emph{convex projective hex torus}; indeed, such a torus may be decomposed into ``regular'' hexagons, each of which is the union of six copies of the fundamental triangle $T$.

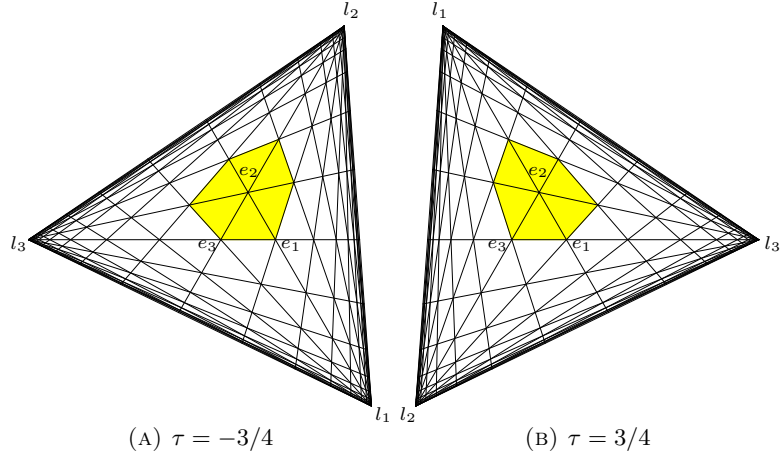
\begin{figure}[ht!]
\labellist
\tiny\hair 2pt
\pinlabel $l_1$ at 455 0
\pinlabel $l_2$ at 415 508
\pinlabel $l_3$ at -3 215
\pinlabel $e_1$ at 339 212
\pinlabel $e_2$ at 286 305
\pinlabel $e_3$ at 234 212
\endlabellist
\subfloat[$\tau=-3/4$]{\includegraphics[scale=.30]{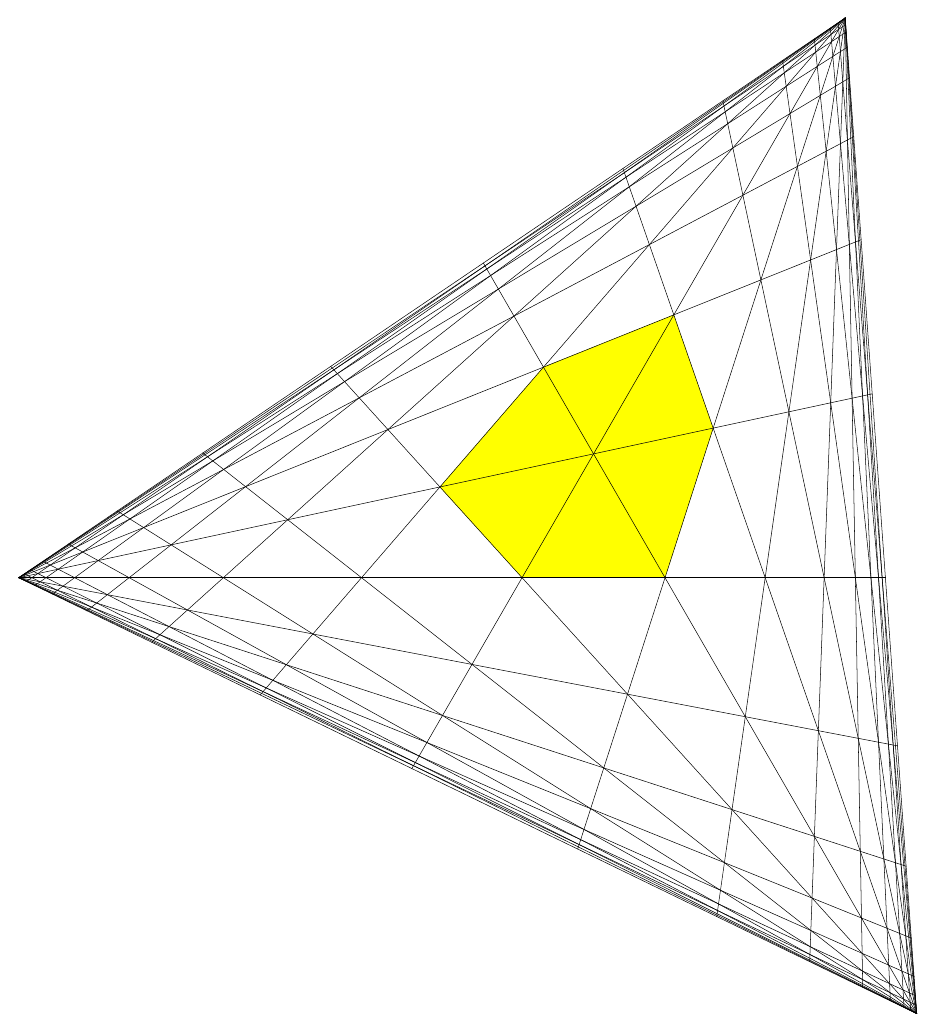}}
\quad
\labellist
\tiny\hair 2pt
\pinlabel $l_1$ at 40 508
\pinlabel $l_2$ at 0 0
\pinlabel $l_3$ at 458 215
\pinlabel $e_1$ at 218 212
\pinlabel $e_2$ at 163 305
\pinlabel $e_3$ at 112 212
\endlabellist
\centering
\subfloat[$\tau=3/4$]{\includegraphics[scale=.30]{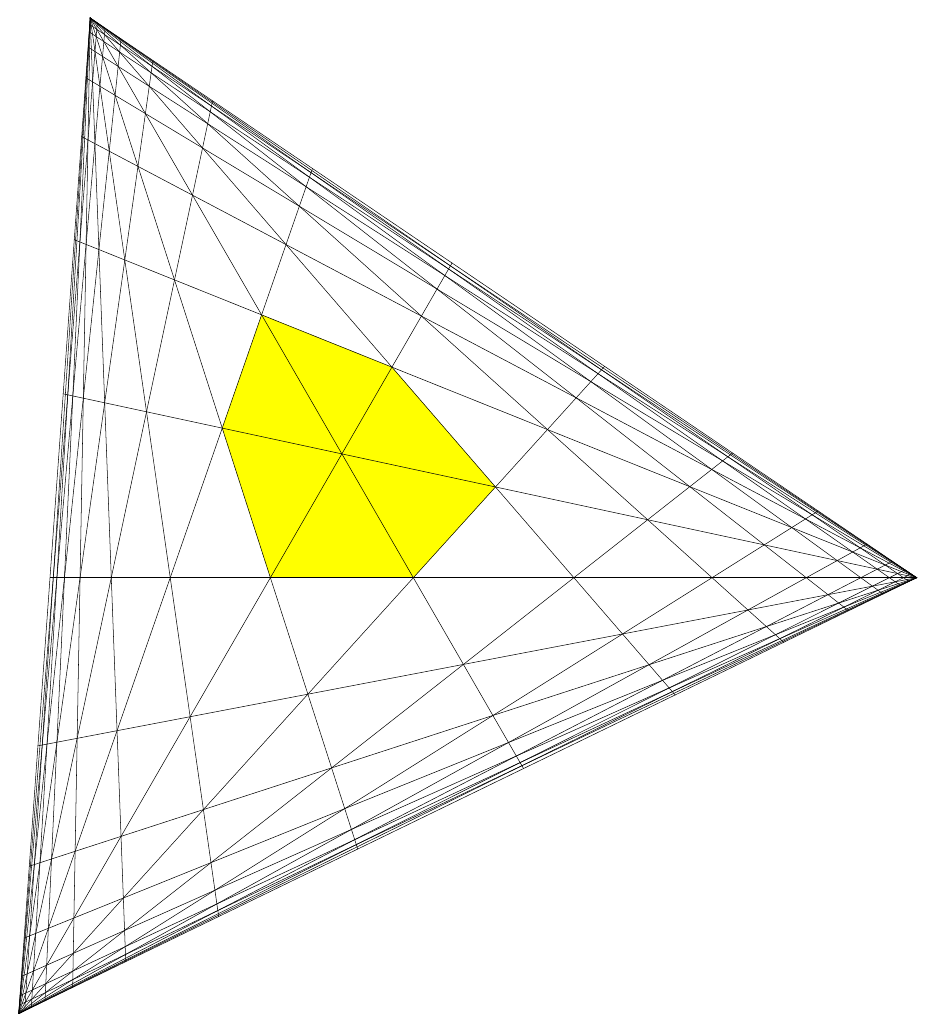}}
\caption{Fundamental domains of dual minimal hex tori}\label{hexTorus}
\end{figure}

The maximal torus subgroup $\Delta_H$ of $\Delta_{3,3,3}$ is generated by three elements (which are translations in the Euclidean structure):
$$
g_1=r_3 r_2 r_3 r_1,\quad g_2=r_1 r_3 r_1 r_2,\quad g_3=r_2 r_1 r_2 r_3.
$$
For $\tau \neq 0$, let $h_\tau$ denote the change of basis taking $(l_1(\tau), l_2(\tau), l_3(\tau))$ to the standard basis.
Then for $i \in \{1,2,3\}$, $h_\tau \zeta_\tau(g_i) h_\tau^{-1} = \exp(\tau \alpha_i)$, where
$$
\alpha_1 = \begin{pmatrix}
-1 & 0 & 0 \\
0 & 0 & 0 \\
0 & 0 & 1\\
\end{pmatrix}, \quad
\alpha_2 = \begin{pmatrix}
1 & 0 & 0 \\
0 & -1 & 0 \\
0 & 0 & 0\\
\end{pmatrix}, \quad
\alpha_3 = \begin{pmatrix}
0 & 0 & 0 \\
0 & 1 & 0 \\
0 & 0 & -1\\
\end{pmatrix}.
$$
It is then clear that for any $\gamma \in \Delta_H$ and natural number $k$,
\begin{equation}\label{eqn:speed-up} h_\tau \zeta_{\tau}(k \gamma)h_\tau^{-1} = h_{k \tau} \zeta_{k \tau} (\gamma) h_{k \tau}^{-1}.\end{equation} Therefore, if $\Delta$ is any hex torus group, the path of representations $\zeta_\tau$ restricted to $k\Delta$ looks the same, up to conjugation, as the $k$--times faster path of representations $\zeta_{k \tau}$ restricted to $\Delta$. Using this we prove:

\begin{lemma}\label{QIsometry}
Let $\mathcal S_1, \mathcal S_2$ be Euclidean regular hex tori with fundamental groups $\Delta_1, \Delta_2 \subset \Delta_{3,3,3}$ and assume that $f \co \mathcal S_1 \to \mathcal S_2$ is a lattice $\tfrac{k_2}{k_1}$--isometry, with $f_* \co \Delta_1 \to \Delta_2$ denoting the isomorphism at the level of fundamental group.
Then for any $\tau \neq 0$, there exists $g \in \SL^\pm_3 \RR$ such that $g \zeta_{\tau/k_1}(\gamma) g^{-1} = \zeta_{\tau/k_2}( f_*(\gamma))$ for all $\gamma \in \Delta_1$.
\end{lemma}
\begin{proof}
By assumption, there exists $\Delta'$ such that $\Delta_1 = k_1 \Delta'$ and $\Delta_2 = k_2 \Delta'$ and $k_2^{-1} f_* k_1 \co \Delta' \to \Delta'$ is conjugation by an element $\delta \in \Delta_{3,3,3}$. Then for any $\gamma \in \Delta_1$, write $\gamma = k_1 \gamma'$, where $\gamma' \in \Delta'$ and observe:
\begin{align*}
k_2^{-1} f_* k_1 \gamma' &= \delta \gamma' \delta^{-1} \\ \implies \quad\quad\;\; f_* (\gamma) &= k_2 \delta \gamma' \delta^{-1} \\ \implies \zeta_{\tau/k_2} (f_* (\gamma)) &= \zeta_{\tau/k_2} \left( \frac{k_2}{k_1} k_1 \delta \gamma' \delta^{-1}\right)\\ &= h \zeta_{\tau/k_1} (\delta k_1 \gamma' \delta^{-1})h^{-1} \\ &= g \zeta_{\tau/k_1}( \gamma) g^{-1}
\end{align*}
where $h \in \SL^\pm_3 \RR$ is determined by~\eqref{eqn:speed-up} and $g = h \zeta_{\tau/k_1}(\delta)$.
\end{proof}

\subsection{Gluing $3$--manifolds which cover reflection orbifolds}
Let $\mathcal O$ be a finite volume hyperbolic three-orbifold. Each cusp of $\mathcal O$ is naturally the product $\partial \times \RR$ of a Euclidean two-orbifold $\partial$ with an interval. We think of $\mathcal O$ as the interior of an ``orbifold with boundary'', where the boundary includes a copy of each cusp cross-section $\partial$ at infinity. Note that if $\mathcal O$ is a reflection orbifold, then any convex projective structure on $\mathcal O$ induces a convex projective structure on its boundary components $\partial$. Indeed, the fundamental group $\Delta$ of $\partial$ fixes a unique point with eigenvalue one under the holonomy representation of any convex projective structure on $\mathcal O$; the convex projective structure induced on $\partial$ may be seen in the link of this point.
Now, let $\partial \cong \mathcal S_{3,3,3}$ be a boundary component isomorphic to the Euclidean $(3,3,3)$--triangle orbifold. Then, for any manifold covering $M \to \mathcal O$, the torus boundary components $\widetilde \partial_1, \ldots, \widetilde \partial_n$ of $M$ which cover $\partial$ are each naturally endowed with a regular Euclidean hex structure.

\begin{theorem}\label{DifferentCover}
For $i \in \{1, \ldots, m\}$, let $M_i$ be a hyperbolic three-manifold of finite volume which covers a reflection orbifold $\mathcal{O}_i$. For $\alpha \in \{1, \ldots, n_i\}$, let $(\widetilde{\partial_i})_\alpha$ be a torus component of $\partial M_i$ which covers a component $\partial_i \cong \mathcal S_{3,3,3}$ of $\partial \mathcal{O}_i$.
Let $\mathscr F$ be a collection of homeomorphisms of the form $f^{\alpha, \beta}_{i,j} \co (\widetilde{\partial_i})_\alpha \to (\widetilde{\partial_j})_\beta$ which identify the boundary components $\{(\widetilde \partial_i)_\alpha\}_{i,\alpha}$ in disjoint pairs.
Assume that:
\begin{enumerate}
\item There are natural numbers $k_1, \ldots, k_m$ such that any $f^{\alpha, \beta}_{i,j} \in \mathscr F$ is a lattice $\frac{k_j}{k_i}$--isometry of the induced regular hex torus structure on $(\widetilde{\partial_i})_\alpha$, and  $(\widetilde{\partial_j})_\beta$.
\item For each $i \in \{1, \ldots, m\}$, there exists a non-trivial continuous path $c_i : (-\epsilon_i,\epsilon_i) \rightarrow \mathfrak{C}(\mathcal{O}_i)$ passing through the finite volume complete hyperbolic structure on $\mathcal O_i$ at $c_i(0)$ such that the convex projective structure on $\partial_i$ induced by $c_i(t)$ is not constant.
\end{enumerate}
Then there exists a properly convex projective structure on the manifold obtained by gluing together the building blocks $M_1, \ldots, M_m$ using the gluing maps in $\mathscr F$.
\end{theorem}
\begin{proof}

By assumption, the convex projective structure on $\partial_i \cong \mathcal S_{3,3,3}$ is not constant, so for some $t_i \neq 0$, the structure is isomorphic to the non-Euclidean convex projective structure associated to some $\tau_i \neq 0$ described in the previous section. It follows easily from the fact that the fundamental group of $\mathcal O_i$ is generated by reflections that the restriction of the holonomy representation $\rho_{t_i}$ for $c_i(t_i)$ to the fundamental group $(\Delta_{3,3,3})_i$ of $\partial_i$ has a fixed point $p_i$ with eigenvalue one and that in a suitable basis $\rho_{t_i} \big |_{(\Delta_{3,3,3})_i} = \zeta_{\tau_i} \oplus id_1$, where $id_1$ denotes the one-dimensional trivial representation. Hence, for each $\alpha \in \{1,\ldots,n_i\}$ the restrictions of $\rho_{t_i}$ to the fundamental group $(\Delta_i)_\alpha$ of $(\widetilde \partial_i)_\alpha$ is diagonalizable and satisfies the middle eigenvalue condition. Therefore the pull-back of $c_i(t_i)$ to $M_i$ induces a convex
projective structure which after
removing a collar neighborhood, as in Section~\ref{boundarygeom}, has principal totally geodesic boundary along its (relevant) components.

Further, by continuity, $\tau_i$ may be varied continuously in a small neighborhood $[0, \delta_i)$. Hence, we may choose $t_1, \ldots, t_m$ so that $(\tau_1, \ldots, \tau_m) = \mu (1/k_1, \ldots, 1/k_m)$ for some small $\mu \in \RR^+$. It now follows from the assumption~(1) and Lemma~\ref{QIsometry} that the maps $f^{\alpha, \beta}_{i,j}$ of $\mathscr F$ satisfy the holonomy matching condition~(\ref{eqn:matching}) and Theorem~\ref{thm:convex-gluing} implies the result.
\end{proof}

\subsection{Application: Two-colorable tetrahedral hyperbolic manifolds}

A finite volume hyperbolic three-manifold is \emph{tetrahedral} if it admits an ideal triangulation consisting of regular ideal tetrahedra.
A triangulated three-manifold is \emph{two-colorable} if the tetrahedra can be colored using two colors so that no two adjacent tetrahedra have the same color.
Two-colorability may be easily checked by \verb"Regina"~\cite{Regina}. For example, the figure eight knot complement and its sister manifold are tetrahedral and two-colorable.
Recently Fominykh--Garoufalidis--Goerner--Tarkaev--Vesnin~\cite{FGGTV} gave a census of all orientable tetrahedral manifolds with at most 25 tetrahedra.

It is well-known (see e.g. \cite[Remark 5.5]{FGGTV}) that every two-colorable tetrahedral manifold $M$ is a cover of the Bianchi Orbifold $\mathcal O_3 = \mathrm{PSL}(2,\ZZ[\zeta]) \backslash \HH^3$ of discriminant $D = -3$, where $\zeta= (1+\sqrt{-3})/2$. Further $\mathcal O_3$ covers the reflection orbifold $\mathcal O^{2,2,3}_{3,3,3}$ determined by reflections in the faces of the partially ideal tetrahedron with dihedral angles $\pi/2$, $\pi/2$, $\pi/3$ at the three edges bounding one face and $\pi/3$ along the other edges, incident on the unique ideal vertex. The reflection orbifold $\mathcal O^{2,2,3}_{3,3,3}$ has one cusp and its cross section is the Euclidean $(3,3,3)$--triangle orbifold. Further, it follows from Benoist's work~\cite{BenoistCD4} that the hyperbolic structure on the orbifold $\mathcal O^{2,2,3}_{3,3,3}$ may be deformed non-trivially to nearby convex projective structures. So the deformation hypothesis~(2) of Theorem~\ref{DifferentCover} is satisfied for tetrahedral two-colorable manifolds, and therefore the theorem may be applied given appropriate gluing maps between torus boundary components that respect the hex structure as in hypothesis (1). Indeed, such gluing maps seem to be easy to find. We now give some examples.

Of the 29 orientable tetrahedral manifolds with at most $8$ tetrahedra, 20 are two-colorable, and each of those has at most 2 cusps. In Table \ref{tab:Census} below, these 20 manifolds are listed along with the hex cusp shapes of the cusps.
\begin{table}[ht!]
\centering
{\renewcommand{\arraystretch}{1.5}
\renewcommand{\tabcolsep}{4pt}
\begin{tabular}{clcc|clcc}
Name & Name \cite{FGGTV} & $\partial_1$ & $\partial_2$ & Name & Name \cite{FGGTV} & $\partial_1$ & $\partial_2$\\
\hline
\verb"m003" & $\mathrm{otet02}_{0000}$ & $\left[\begin{smallmatrix}
2 & 0 \\
0 & 2 \\
\end{smallmatrix}\right]$ &

& \verb"t12845" & $\mathrm{otet08}_{0001}$ & $\left[\begin{smallmatrix}
3 & 2 \\
1 & 5 \\
\end{smallmatrix}\right]$ & $\left[\begin{smallmatrix}
1 & 0 \\
0 & 3 \\
\end{smallmatrix}\right]$\\

\verb"m004" & $\mathrm{otet02}_{0001}$ & $\left[\begin{smallmatrix}
1 & 0 \\
0 & 4 \\
\end{smallmatrix}\right]$ &

&\verb"t12840" & $\mathrm{otet08}_{0002}$ & $\left[\begin{smallmatrix}
3 & 0 \\
2 & 4 \\
\end{smallmatrix}\right]$ & $\left[\begin{smallmatrix}
1 & 0 \\
0 & 4 \\
\end{smallmatrix}\right]$\\

\verb"m202" & $\mathrm{otet04}_{0000}$ & $\left[\begin{smallmatrix}
3 & 1 \\
2 & 3 \\
\end{smallmatrix}\right]$ & $\left[\begin{smallmatrix}
1 & 0 \\
0 & 1 \\
\end{smallmatrix}\right]$ &\verb"t12842" & $\mathrm{otet08}_{0003}$ & $\left[\begin{smallmatrix}
3 & 0 \\
2 & 4 \\
\end{smallmatrix}\right]$ & $\left[\begin{smallmatrix}
2 & 1 \\
0 & 2 \\
\end{smallmatrix}\right]$\\

\verb"m203" & $\mathrm{otet04}_{0001}$ & $\left[\begin{smallmatrix}
2 & 1 \\
0 & 3 \\
\end{smallmatrix}\right]$ & $\left[\begin{smallmatrix}
1 & 0 \\
0 & 2 \\
\end{smallmatrix}\right]$ &\verb"t12843" & $\mathrm{otet08}_{0004}$ & $\left[\begin{smallmatrix}
3 & 2 \\
2 & 6 \\
\end{smallmatrix}\right]$ & $\left[\begin{smallmatrix}
1 & 0 \\
0 & 2 \\
\end{smallmatrix}\right]$\\

\verb"m206" & $\mathrm{otet04}_{0002}$ & $\left[\begin{smallmatrix}
2 & 0 \\
0 & 4 \\
\end{smallmatrix}\right]$ &

&\verb"t12844" & $\mathrm{otet08}_{0005}$ & $\left[\begin{smallmatrix}
3 & 2 \\
2 & 6 \\
\end{smallmatrix}\right]$ & $\left[\begin{smallmatrix}
1 & 0 \\
0 & 2 \\
\end{smallmatrix}\right]$\\

\verb"m207" & $\mathrm{otet04}_{0003}$ & $\left[\begin{smallmatrix}
3 & 1 \\
1 & 3 \\
\end{smallmatrix}\right]$ &

&\verb"t12837" & $\mathrm{otet08}_{0006}$ & $\left[\begin{smallmatrix}
3 & 1 \\
2 & 6 \\
\end{smallmatrix}\right]$ & \\

\verb"s959" & $\mathrm{otet06}_{0002}$ & $\left[\begin{smallmatrix}
3 & 0 \\
0 & 3 \\
\end{smallmatrix}\right]$ & $\left[\begin{smallmatrix}
2 & 1 \\
1 & 2 \\
\end{smallmatrix}\right]$ &\verb"t12839" & $\mathrm{otet08}_{0007}$ & $\left[\begin{smallmatrix}
4 & 2 \\
0 & 4 \\
\end{smallmatrix}\right]$ & \\

\verb"s961" & $\mathrm{otet06}_{0003}$ & $\left[\begin{smallmatrix}
3 & 0 \\
2 & 4 \\
\end{smallmatrix}\right]$ &

&\verb"t12838" & $\mathrm{otet08}_{0008}$ & $\left[\begin{smallmatrix}
4 & 2 \\
2 & 5 \\
\end{smallmatrix}\right]$ & \\

\verb"s960" & $\mathrm{otet06}_{0004}$ & $\left[\begin{smallmatrix}
4 & 2 \\
2 & 4 \\
\end{smallmatrix}\right]$ &

&\verb"t12836" & $\mathrm{otet08}_{0009}$ & $\left[\begin{smallmatrix}
3 & 2 \\
0 & 3 \\
\end{smallmatrix}\right]$ & $\left[\begin{smallmatrix}
2 & 1 \\
1 & 4 \\
\end{smallmatrix}\right]$ \\

\verb"s958" & $\mathrm{otet06}_{0006}$ & $\left[\begin{smallmatrix}
3 & 2 \\
0 & 4 \\
\end{smallmatrix}\right]$ &

&\verb"t12841" & $\mathrm{otet08}_{0010}$ & $\left[\begin{smallmatrix}
4 & 2 \\
2 & 4 \\
\end{smallmatrix}\right]$ & $\left[\begin{smallmatrix}
2 & 0 \\
0 & 2 \\
\end{smallmatrix}\right]$\\
\hline
\end{tabular}}
    \caption{Hex torus cusps of orientable, two-colorable tetrahedral manifolds with at most $8$ tetrahedra}
    \label{tab:Census}
\end{table}
The hex torus structure at the cusps of tetrahedral manifolds can be calculated easily using \verb"SnapPy"~\cite{SnapPy}.
In Figure~\ref{Cusp}(A), a fundamental triangle for the cusp of $\mathcal O^{2,2,3}_{3,3,3}$ is drawn in yellow, while the cusp pattern of the tessellation of $\HH^3$ by regular ideal tetrahedra is drawn in black. The hex torus structure of one boundary component $\partial_1$ of the tetrahedral two-colorable manifold {\tt m207} is shown in Figure~\ref{Cusp}(B): $\Delta_H$ (resp. the hex lattice $\Delta_1$ of $\partial_1$) corresponds to vertices of triangle with black edges (resp. red dots), and two generators of $\Delta_{1}$ are shown as pink arrows; we easily read of the hex cusp shape: $A_{\Delta_1} = \left[\begin{smallmatrix}
3 & 1 \\
1 & 3 \\
\end{smallmatrix}\right]$.

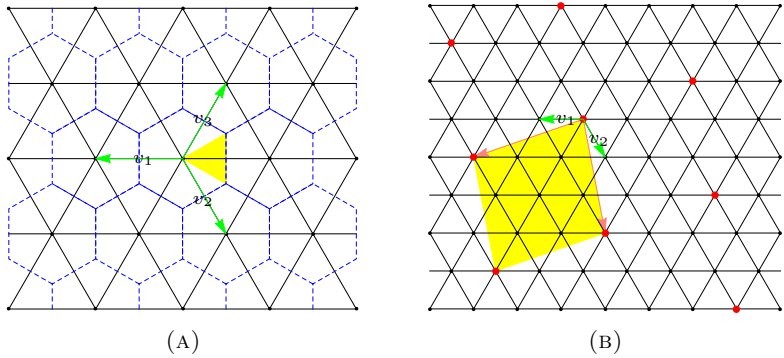
\begin{figure}[ht!]
\labellist
\tiny\hair 2pt
\pinlabel $v_1$ at 141 156
\pinlabel $v_2$ at 200 115
\pinlabel $v_3$ at 200 195
\endlabellist
\subfloat[]{\includegraphics[scale=.39]{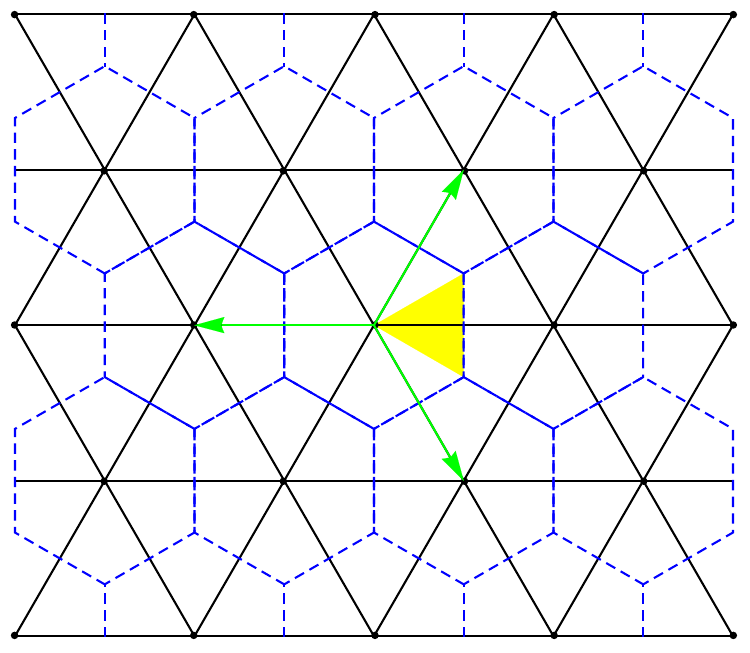}}
\quad\quad
\labellist
\tiny\hair 2pt
\pinlabel $v_1$ at 155 212
\pinlabel $v_2$ at 190 190
\endlabellist
\centering
\subfloat[]{\includegraphics[scale=.35]{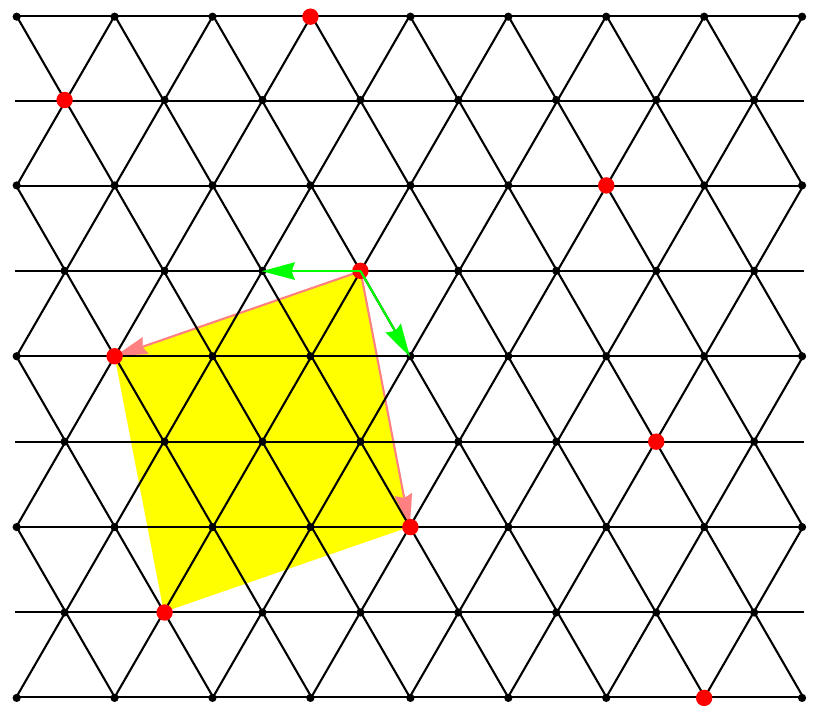}}
\caption{(A) Hex torus cusps of orientable two-colorable tetrahedral manifolds, and (B) A hex torus cusp $\partial_1$ of {\tt m207}}\label{Cusp}
\end{figure}

From Table~\ref{tab:Census}, we may easily read off several new examples of closed three-manifolds that admit properly convex projective structures.
\begin{theorem}\label{thm:some-examples}
In each line of Table~\ref{tab:thm} below, there exist gluing maps identifying the boundary components of the given building blocks in pairs so that the resulting closed three-manifold admits a properly convex projective structure. The gluing maps may be read off from Table~\ref{tab:Census}.
\begin{table}[ht!]
\begin{center}
\begin{tabular}{|c|c|}
\hline
Building blocks & Gluing maps\\
\hline $\mathtt{m003}, \mathtt{s959}, \mathtt{s960}$ & $\partial_1 \mathtt{m003} \to \partial_1 \mathtt{s959}$\\ & $\partial_2 \mathtt{s959} \to \partial_1 \mathtt{s960}$\\ \hline
$\mathtt{m003}, \mathtt{t12841}, \mathtt{s960}$ &$\partial_1 \mathtt{m003} \to \partial_2 \mathtt{t12841}$\\ & $\partial_1 \mathtt{t12841} \to \partial_1 \mathtt{s960}$\\ \hline
$\mathtt{m004}, \mathtt{t12840}, \mathtt{s961}$ &$\partial_1 \mathtt{m004} \to \partial_2 \mathtt{t12840}$\\ & $\partial_1 \mathtt{t12840} \to \partial_1 \mathtt{s961}$\\ \hline
$\mathtt{t12843}, \mathtt{t12844}$ &$\partial_1 \mathtt{t12843} \to \partial_1 \mathtt{t12844}$\\ & $\partial_2 \mathtt{t12843} \to \partial_2 \mathtt{t12844}$\\ \hline
$\mathtt{t12842}, \mathtt{t12839}, \mathtt{s961}$ &$\partial_1 \mathtt{t12842} \to \partial_1 \mathtt{s961}$\\ & $\partial_2 \mathtt{t12842} \to \partial_1 \mathtt{t12839}$\\ \hline
\end{tabular}
\end{center}
\caption{Building blocks and boundary pairings for Theorem~\ref{thm:some-examples}}
\label{tab:thm}
\end{table}
\end{theorem}
\begin{proof}
Examination of Table~\ref{tab:Census} verifies that each given identification of boundary tori satisfies the hypotheses of Lemma~\ref{lem:hex-shape} and therefore a lattice $\mathbb Q$--isometry exists between the given boundary components.  For example, $\partial_1 ({\tt m003})$ is lattice $\tfrac{3}{2}$--isometric to $\partial_1 ({\tt s959})$ and
$\partial_2 ({\tt s959})$ is lattice $\tfrac{2}{1}$--isometric to $\partial_1 ({\tt s960})$, and so on.
The result follows from Theorem~\ref{DifferentCover}.
\end{proof}
\begin{remark}
As in Remark~\ref{rem:finitely-many}, there may be more than one (but no more than $12$) gluing map identifying a given pair of boundary components in Theorem~\ref{thm:some-examples}. Different gluing maps might or might not produce homeomorphic manifolds. We do not check carefully here the number of different homeomorphism types of closed three-manifolds for which Theorem~\ref{thm:some-examples} gives properly convex projective structures. It is at least five, but seems likely to be more.
\end{remark}

\subsection{Application: The cusp covering conjecture}

Finally, we briefly mention an application of the Cusp Covering Conjecture, proven by Wise (see \cite{Wis} and \cite[Corollary 16.15]{Wis1}).
\begin{theorem}[Cusp Covering Conjecture, Wise~\cite{Wis, Wis1}]
Let $M$ be a finite volume hyperbolic $3$--manifold and let $\partial M = \partial_1 \sqcup \partial_2 \sqcup \dotsm \sqcup \partial_r$.
There exist finite covers $\partial_i^\circ \rightarrow \partial_i$ such that for any further finite covers $\partial_i^c \rightarrow \partial_i^\circ$,
there is a finite cover $\widehat{M} \rightarrow M$ such that for each $i$, each cover of $\partial_i$ appearing on the boundary of $\widehat{M}$ is isomorphic to $\partial_i^c \to \partial_i$.
\end{theorem}
Let $\mathcal O$ be a reflection orbifold with one cusp $\partial \cong \mathcal S_{3,3,3}$ satisfying the hypothesis~(2) of Theorem~\ref{DifferentCover} and let $M$ be a manifold cover. It follows easily from the Cusp Covering Conjecture that there is a finite cover $\widehat M$ of $M$ (in fact infinitely many such) so that the inclusion of the fundamental group $\Delta^c_i$ of each boundary component $\partial^c_i$ into the fundamental group $\Delta_{3,3,3}$ of $\partial$ has the same image. Hence, we may glue together in pairs the boundary components of $\widehat M$ via a lattice isometry of the hex torus structure. The resulting manifolds admit convex projective structures by Theorem~\ref{DifferentCover}. We may also arrange that for each $i$, the image of some $\Delta^c_i \hookrightarrow \Delta_{3,3,3}$ is a multiple $k \Delta'$ of the fundamental group of any given hex torus subgroup $\Delta'$. Let $\Delta'$ to be hex lattice isomorphic to, for example, the hex torus fundamental group of the 
boundary torus $\partial M'$ of some one-cusped manifold $M'$ in Table~\ref{tab:Census}. Then any boundary component of the resulting cover $\widehat M$ of $\mathcal O$ may be glued to the boundary of a copy of $M'$ by a lattice $k$--isometry. Let $N$ be a manifold obtained by gluing together some pairs of boundary components of $\widehat M$ via hex lattices isometries and by gluing on copies of $M'$ with hex lattice $k$--isometries to the remaining boundary components. Then $N$ admits a properly convex projective structure by Theorem~\ref{DifferentCover}. This gives one way to obtain many interesting examples of closed non-hyperbolic three-manifolds $N$ which admits convex projective structures.


\end{document}

%% file: four-d-slice-paraboloid-sphere.pdf_tex
\begingroup%
  \makeatletter%
  \providecommand\color[2][]{%
    \errmessage{(Inkscape) Color is used for the text in Inkscape, but the package 'color.sty' is not loaded}%
    \renewcommand\color[2][]{}%
  }%
  \providecommand\transparent[1]{%
    \errmessage{(Inkscape) Transparency is used (non-zero) for the text in Inkscape, but the package 'transparent.sty' is not loaded}%
    \renewcommand\transparent[1]{}%
  }%
  \providecommand\rotatebox[2]{#2}%
  \ifx\svgwidth\undefined%
    \setlength{\unitlength}{785.65004883bp}%
    \ifx\svgscale\undefined%
      \relax%
    \else%
      \setlength{\unitlength}{\unitlength * \real{\svgscale}}%
    \fi%
  \else%
    \setlength{\unitlength}{\svgwidth}%
  \fi%
  \global\let\svgwidth\undefined%
  \global\let\svgscale\undefined%
  \makeatother%
  \begin{picture}(1,0.31872492)%
    \put(0,0){\includegraphics[width=\unitlength]{four-d-slice-paraboloid-sphere.pdf}}%
    \put(0.13398174,0.0149856){\color[rgb]{0,0,0}\makebox(0,0)[lb]{\smash{$x_1$}}}%
    \put(0.09187105,0.05199466){\color[rgb]{0,0,0}\makebox(0,0)[lb]{\smash{$x_2$}}}%
    \put(0.00849217,0.07053014){\color[rgb]{0,0,0}\makebox(0,0)[lb]{\smash{$x_3$}}}%
    \put(0.40956464,0.21560455){\color[rgb]{0,0,0}\makebox(0,0)[lb]{\smash{$p_1$}}}%
    \put(0.36325683,0.07875899){\color[rgb]{0,0,0}\makebox(0,0)[lb]{\smash{$p_3$}}}%
    \put(0.07218285,0.2114281){\color[rgb]{0,0,0}\makebox(0,0)[lb]{\smash{$p_2$}}}%
    \put(0.43814353,0.12546891){\color[rgb]{0,0,0}\makebox(0,0)[lb]{\smash{$S_t$}}}%
    \put(0.03085809,0.14307923){\color[rgb]{0,0,0}\makebox(0,0)[lb]{\smash{$P_t$}}}%
    \put(0.86249033,0.15397207){\color[rgb]{0,0,0}\makebox(0,0)[lb]{\smash{$S_t$}}}%
    \put(0.75983158,0.27834711){\color[rgb]{0,0,0}\makebox(0,0)[lb]{\smash{$p_\infty$}}}%
    \put(0.91310514,0.23354155){\color[rgb]{0,0,0}\makebox(0,0)[lb]{\smash{$p_1$}}}%
    \put(0.56432555,0.25461698){\color[rgb]{0,0,0}\makebox(0,0)[lb]{\smash{$p_2$}}}%
    \put(0.73718773,0.21120427){\color[rgb]{0,0,0}\makebox(0,0)[lb]{\smash{$p_3$}}}%
  \end{picture}%
\endgroup%

%% file: convexdomain-2.pdf_tex
\begingroup%
  \makeatletter%
  \providecommand\color[2][]{%
    \errmessage{(Inkscape) Color is used for the text in Inkscape, but the package 'color.sty' is not loaded}%
    \renewcommand\color[2][]{}%
  }%
  \providecommand\transparent[1]{%
    \errmessage{(Inkscape) Transparency is used (non-zero) for the text in Inkscape, but the package 'transparent.sty' is not loaded}%
    \renewcommand\transparent[1]{}%
  }%
  \providecommand\rotatebox[2]{#2}%
  \ifx\svgwidth\undefined%
    \setlength{\unitlength}{629.23139648bp}%
    \ifx\svgscale\undefined%
      \relax%
    \else%
      \setlength{\unitlength}{\unitlength * \real{\svgscale}}%
    \fi%
  \else%
    \setlength{\unitlength}{\svgwidth}%
  \fi%
  \global\let\svgwidth\undefined%
  \global\let\svgscale\undefined%
  \makeatother%
  \begin{picture}(1,1.06371333)%
    \put(0,0){\includegraphics[width=\unitlength]{convexdomain-2.pdf}}%
    \put(0.46846293,0.6268813){\color[rgb]{0,0,0}\makebox(0,0)[lb]{\smash{$\mathcal T_+$}}}%
    \put(0.84352368,0.91548739){\color[rgb]{0,0,0}\makebox(0,0)[lb]{\smash{$\mathcal T_+$}}}%
    \put(0.69095659,0.78580535){\color[rgb]{0,0,0}\makebox(0,0)[lb]{\smash{$\mathcal T_-$}}}%
    \put(0.80156774,0.52008436){\color[rgb]{0,0,0}\makebox(0,0)[lb]{\smash{$T$}}}%
    \put(0.35785179,0.5607689){\color[rgb]{0,0,0}\makebox(0,0)[lb]{\smash{$\Omega_{min}$}}}%
    \put(0.37056571,0.45651472){\color[rgb]{0,0,0}\makebox(0,0)[lb]{\smash{$\Omega$}}}%
    \put(0.18748521,0.20477903){\color[rgb]{0,0,0}\makebox(0,0)[lb]{\smash{$\Omega_{max}$}}}%
  \end{picture}%
\endgroup%